\documentclass[10pt]{amsart}

\usepackage[a4paper,margin=1in]{geometry}
\usepackage{amsmath,amssymb,mathtools,bm}
\usepackage{enumitem}
\usepackage{graphicx,booktabs}
\usepackage{mathrsfs}
\usepackage{float}
\usepackage{microtype}
\usepackage{aliascnt}
\usepackage{xcolor}
\usepackage[colorlinks=true,linkcolor=blue,citecolor=blue,urlcolor=blue]{hyperref}
\usepackage[nameinlink,capitalize,noabbrev]{cleveref}

\allowdisplaybreaks
\newtheorem{theorem}{Theorem}[section]
\newaliascnt{proposition}{theorem}
\newtheorem{proposition}[proposition]{Proposition}
\aliascntresetthe{proposition} \newaliascnt{lemma}{theorem}
\newtheorem{lemma}[lemma]{Lemma}
\aliascntresetthe{lemma} \newaliascnt{corollary}{theorem}
\newtheorem{corollary}[corollary]{Corollary}
\aliascntresetthe{corollary} \newaliascnt{assumption}{theorem}
\newtheorem{assumption}[assumption]{Assumption}
\aliascntresetthe{assumption}
\theoremstyle{definition}
\newaliascnt{definition}{theorem}

\aliascntresetthe{definition} \newaliascnt{alg}{theorem}
\newtheorem{alg}[alg]{Algorithm}
\aliascntresetthe{alg}
\theoremstyle{remark}
\newaliascnt{remark}{theorem}
\newtheorem{remark}[remark]{Remark}
\aliascntresetthe{remark}

\crefname{theorem}{Theorem}{Theorems} \crefname{proposition}{Proposition}{Propositions} \crefname{lemma}{Lemma}{Lemmas} \crefname{corollary}{Corollary}{Corollaries} \crefname{assumption}{Assumption}{Assumptions} \crefname{definition}{Definition}{Definitions} \crefname{alg}{Algorithm}{Algorithms} \crefname{remark}{Remark}{Remarks}

\DeclareMathOperator*{\argmin}{arg\,min}
\newcommand{\Lips}{\operatorname{Lip}}
\newcommand{\Div}{\operatorname{div}}
\newcommand{\KLdiv}{\operatorname{KL}}
\newcommand{\R}{\mathbb{R}}

\newcommand{\Pone}{\mathcal{P}_1}
\newcommand{\Hh}{\mathcal{H}}
\newcommand{\J}{\mathcal{J}}
\newcommand{\dd}{\,\mathrm{d}}
\newcommand{\norm}[1]{\left\lVert #1\right\rVert}
\newcommand{\abs}[1]{\left\lvert #1\right\rvert}
\newcommand{\ip}[2]{\left\langle #1,#2\right\rangle}
\newcommand{\eps}{\varepsilon}

\numberwithin{equation}{section}
\title[Alternating score--system iterations]{Convergence of Alternating Score--System Iterations for Regularized Landau Particle Dynamics}

\author{Yeoneung Kim}
\address{Yonsei University, Department of Industrial Engineering, Seoul,
	Republic of Korea}
\email{yeoneung@yonsei.ac.kr}

\author{Donghyun Lee}
\address{Department of Mathematical Sciences, Pohang University of Science and Technology, Pohang, Republic of Korea}
\email{donglee@postech.ac.kr}

\author{Sung-Jun Son}
\address{Center for Mathematical Machine Learning and its Applications (CM2LA), Department of Mathematics, Pohang University of Science and Technology, Pohang, Republic of Korea} 
\email{sungjun129@postech.ac.kr}

\date{\today}

\begin{document}
	
	\begin{abstract}
		We analyze an alternating iteration for score-based particle approximations of the homogeneous Landau equation. Each step computes a regularized score along the previous trajectory and recomputes the trajectory with that score frozen. First, at fixed positive regularization, we prove existence, uniqueness, and $W_1$-Lipschitz stability of a Tikhonov score minimizer in a vector-valued reproducing kernel Hilbert space embedded in $C_b^2$. For bounded Lipschitz collision matrices, this yields factorial convergence to the coupled regularized dynamics on every finite interval, with constants independent of particle number. Global objective gaps and system residuals control inexact solves, and structural kernel assumptions ensure momentum and energy conservation. The mean-field and alternating-iteration limits commute at fixed regularization in this regime. Second, score-difference cancellation extends the theory to far-field-truncated Landau kernels with $\gamma\ge-2$ and the uncut case $\gamma=-2$. For the uncut three-dimensional Coulomb kernel, a separate continuum construction gives local factorial convergence from bounded, compactly supported densities. Third, we characterize the obstruction to removing regularization at fixed particle number and establish a separate conditional density-level convergence criterion under uniform regularity and approximation assumptions, using relative entropy. A heat-equation counterexample shows that an unregularized density-level Picard iteration can diverge. Atomic Coulomb and simultaneous particle--zero-regularization limits remain open.
	\end{abstract}
	
	\maketitle
	
	\section{Introduction}
	\label{sec:introduction}
	
	\subsection{The Landau equation and regularized empirical scores}
	
	The spatially homogeneous Landau equation models grazing collisions in a plasma and admits deterministic particle formulations that preserve the pairwise structure of its collision operator \cite{CarrilloHuWangWu2020}. For a probability density $f_t$ on $\R^d$, $d\ge2$, it reads
	\begin{equation}\label{eq:intro-draft-landau-cauchy}
		\partial_t f_t(v)
		=\nabla_v\!\cdot\!\int_{\R^d}A_\gamma(v-w)
		\bigl(f_t(w)\nabla_v f_t(v)-f_t(v)\nabla_w f_t(w)\bigr)\dd w,
		\qquad f_{t=0}=f_0,
	\end{equation}
	with collision kernel
	\begin{equation}\label{eq:intro-draft-physical-kernel}
		A_\gamma(z)=C_\gamma\abs z^{\gamma}\bigl(\abs z^2I_d-z\otimes z\bigr)
		=C_\gamma\abs z^{\gamma+2}\Bigl(I_d-\frac{z\otimes z}{\abs z^2}\Bigr),
		\qquad z\ne0,
	\end{equation}
	where $C_\gamma>0$ and $-d\le\gamma\le1$; $\gamma=0$ gives Maxwell molecules, and $\gamma=-d$ the Coulomb interaction, physically most relevant for $d=3$ \cite{HuangWang2025}. For sufficiently regular solutions, the identities $A_\gamma(-z)=A_\gamma(z)=A_\gamma(z)^\top$, $A_\gamma(z)z=0$, and $A_\gamma\succeq0$ yield conservation of mass, momentum, and energy and the decay of the entropy $\int f_t\log f_t$. Since $f(w)\nabla f(v)-f(v)\nabla f(w)=f(v)f(w)(q_f(v)-q_f(w))$ with the \emph{score} $q_f:=\nabla\log f$, a smooth positive solution satisfies
	\begin{equation}\label{eq:intro-draft-landau-continuity}
		\partial_t f_t+\nabla_v\cdot\bigl(f_tU[f_t]\bigr)=0,
		\qquad
		U[f](v):=-\int_{\R^d}A_\gamma(v-w)\bigl(q_f(v)-q_f(w)\bigr)f(w)\dd w.
	\end{equation}
	Score-based particle methods \cite{HuangWang2025,IlinHuWang2025} estimate the score from a particle cloud and move the particles along $U$, avoiding the reconstruction and differentiation of the density.
	
	The score is estimated by implicit score matching. For a positive density $f$ with finite Fisher information, integration by parts turns the loss $\int\abs{s-q_f}^2f$ into $\int(\abs s^2+2\Div s)f$ up to an $s$-independent constant (the Hyv\"arinen identity \cite{Hyvarinen2005}; see \cref{prop:density-level-identity}), which involves no derivative of $f$ and can be approximated from samples. For a particle configuration $Z=(z_1,\ldots,z_N)$ with empirical measure $\mu_Z:=N^{-1}\sum_i\delta_{z_i}$, this leads to the empirical objective $\J_{\mu_Z}^0(s)=N^{-1}\sum_i(\abs{s(z_i)}^2+2\Div s(z_i))$. Its infimum over $C_c^\infty(\R^d;\R^d)$ is $-\infty$ (\cref{rem:empirical-minus-infinity}): a field can vanish at every sample point while its divergence there is arbitrarily negative. We therefore fix a separable Hilbert space $\Hh$ of vector fields continuously embedded in $C_b^2(\R^d;\R^d)$, so that $\Hh$ is a vector-valued reproducing kernel Hilbert space (RKHS), and use the Tikhonov-regularized functional
	\begin{equation}\label{eq:regularized-score-objective-intro}
		\J_\mu^\lambda(s)
		:=\int_{\R^d}\bigl(\abs s^2+2\Div s\bigr)\dd\mu+\lambda\norm{s}_{\Hh}^2,
		\qquad
		S_\lambda(\mu):=\argmin_{s\in\Hh}\J_\mu^\lambda(s),
		\qquad \lambda>0.
	\end{equation}
	The parameter $\lambda$ regularizes the score step; for fixed $\lambda$, $S_\lambda(f\dd v)$ need not coincide with $q_f$. A Gaussian RKHS with fixed bandwidth is one example. For empirical measures, $S_\lambda(\mu_Z)$ solves a finite symmetric positive-definite linear system \cite{SriperumbudurEtAl2017,ZhouShiZhu2020}; see \cref{sec:stable-score-learner}.
	
	\subsection{The alternating problem and main contributions}
	
	For a score field $s$ and a configuration $Z=(z_1,\ldots,z_N)$, consider the score-driven force
	\begin{equation}\label{eq:intro-force}
		F_i(Z;s):=-\frac1N\sum_{j=1}^NA(z_i-z_j)\bigl(s(z_i)-s(z_j)\bigr),
		\qquad i=1,\ldots,N.
	\end{equation}
	Starting from a common initial configuration $Z^0$ and the constant path $X_t^0\equiv Z^0$, the alternating iteration is
	\begin{equation}\label{eq:intro-iteration}
		\underbrace{s_t^k=S_\lambda(\mu_{X_t^k})}_{\text{score step}},\qquad
		\underbrace{\dot X_t^{k+1}=F(X_t^{k+1};s_t^k),\qquad X_0^{k+1}=Z^0}_{\text{system step}}.
	\end{equation}
	Each iteration updates a whole trajectory on $[0,T]$; it is not a time step. The score is estimated from the old path, while the frozen score and the collision kernel are evaluated along the new path.
	
	Existing analyses of score-based particle methods either take the learned score as given and bound the solution error by the score-matching or self-consistency loss \cite{HuangWang2025,IlinHuWang2025,Ilin2025,BoffiVandenEijnden2023,LuWuXiang2024,ShenEtAl2022}, or replace score matching by the mollified score of a kernel density estimate, for which stability in the input measure is obtained directly from the mollifier \cite{CarrilloDelgadinoWu2023}. To our knowledge, the present paper gives the first convergence theory in which the score is the actual minimizer of the empirical score-matching objective and no a priori accuracy of the learned score relative to the density score is assumed: the alternation \eqref{eq:intro-iteration} converges factorially on every finite interval, and the deviation of an inexact score from the regularized minimizer is controlled a posteriori by the exact objective gap. The analytic input that makes this possible is the $W_1$-stability of the Tikhonov-regularized score map in the input measure, with constants independent of the particle number (\cref{thm:main-score}); once it is available, classical Volterra estimates, Dobrushin's argument, and the cancellation mechanism of \cite{CarrilloDelgadinoWu2023} apply to the iteration. The convergence statements are qualitative in $\lambda$: the constants degenerate as $\lambda\downarrow0$ (\cref{rem:lambda-scaling}), and the fast convergence observed numerically (Appendix~\ref{sm:numerics}) is not explained by them.
	
	Once \cref{thm:main-score} is available, the coupled system $\dot X_t=F(X_t;S_\lambda(\mu_{X_t}))$ is a globally Lipschitz ODE, so its solution could in principle be computed by time marching with the exact RKHS score evaluated at every step. We analyze the alternation \eqref{eq:intro-iteration} instead because it is the structure of the schemes that separate score learning from the system solve: the score step is a separate optimization over a whole trajectory, as in global-in-time training \cite{KimEtAl2026}, and each system step can then be solved by any ODE solver. Schemes that retrain the score inside every time step, as in \cite{HuangWang2025}, are a different algorithm and are not analyzed here. \Cref{thm:inexact} is the bridge to such implementations: any learner whose excess objective is certified inherits the convergence bound.
	
	This formulation raises three logically distinct questions. The first is whether the empirical score problem has a unique minimizer that depends stably on the moving empirical measure; see \cref{thm:main-score}. The second is whether the alternating iteration converges to a solution of the coupled learned-score dynamics; see \cref{thm:main-outer}. The third is whether the resulting dynamics converges to the physical Landau equation as the particle number tends to infinity, the score regularization vanishes, and the collision regularization is removed; see \cref{thm:main-mean-field,cor:main-joint-limit} for the limit $N\to\infty$ at fixed regularization, and \cref{sec:stable-score-learner,sec:singular-effective-interaction,sec:landau-connection} for the partial treatment of the two regularization limits. Our main contributions are the following.
	\begin{enumerate}[label=(\roman*),leftmargin=2em]
		\item For every $\lambda>0$, $S_\lambda(\mu)$ exists, is unique, and
		satisfies $\norm{S_\lambda(\mu)-S_\lambda(\nu)}_{L^\infty}\le L_{S,\lambda}W_1(\mu,\nu)$ with $L_{S,\lambda}=O(\lambda^{-2})$ (Lax--Milgram and Kantorovich--Rubinstein duality). The exact identity
		\begin{equation}\label{eq:intro-gap}
			\J_\mu^\lambda(s)-\J_\mu^\lambda(S_\lambda(\mu))
			=\int\abs{s-S_\lambda(\mu)}^2\dd\mu
			+\lambda\norm{s-S_\lambda(\mu)}_{\Hh}^2
		\end{equation}
		converts a score-objective gap into a quantitative score error.
		\item For bounded, globally Lipschitz $A$ and every $\lambda>0$, the alternating
		iteration converges on every finite interval to the unique solution of $\dot X_t=F(X_t;S_\lambda(\mu_{X_t}))$, $X_0=Z^0$. The convergence is factorial in the iteration index, requires no smallness of $T$, and has constants independent of $N$; it rests on a causal Volterra recursion. Under the standard structural conditions on $A$, every exact system step conserves momentum and energy.
		\item Inexact score and system solves are controlled a posteriori through
		\eqref{eq:intro-gap} and the system residual, with a persistent-error floor (\cref{thm:inexact}).
		\item The $N$-independent estimates yield a mean-field limit at fixed
		$\lambda$ via a Dobrushin estimate, and the alternating-iteration and mean-field limits commute.
		\item At fixed $N$, the unregularized empirical problem has a minimizer only
		under an exact compatibility condition; otherwise a component invisible at the sample points grows like $\lambda^{-1}$ (\cref{thm:empirical-lambda-zero-dichotomy}). At the density level, for the uncut three-dimensional Coulomb kernel, a reverse-relative-entropy estimate with score forcing gives $\KLdiv(g_t^\lambda\Vert f_t)=O_T(\lambda)$ and $\norm{g_t^\lambda-f_t}_{L^1}=O_T(\sqrt\lambda)$ for the regularized learned-score density $g^\lambda$, under uniform regularity and RKHS source hypotheses and $\KLdiv(g_0^\lambda\Vert f_0)=O(\lambda)$ (\cref{cor:density-level-zero-lambda-KL}). These hypotheses include a uniform bound on $\norm{\nabla\log g_t^\lambda}_{\Hh}$, which excludes not only Maxwellian scores but every score space contained in $L^2(\R^3)$, in particular the Gaussian RKHS (\cref{rem:density-level-lambda-conditional}); the rate is therefore a conditional criterion rather than a regularization-removal theorem for physical data.
		\item The effective interaction $A(v-w)(s(v)-s(w))$ gains one power near the
		diagonal from the score difference. This mechanism, introduced in \cite{CarrilloDelgadinoWu2023} for a mollified score, transfers to the Tikhonov score and extends the fixed-regularization particle and mean-field analysis to far-field-truncated Landau kernels with $\gamma\ge-2$ and to the uncut case $\gamma=-2$ (\cref{thm:effective-interaction-extension,prop:singular-landau-range}); for $d=2$ the latter is the Coulomb case. For the three-dimensional Coulomb exponent $\gamma=-3$, a common-label argument gives a local continuum result for bounded, compactly supported densities (\cref{thm:continuum-Coulomb-waveform}), but not an atomic particle theorem.
	\end{enumerate}
	Thus the part of the third question settled without conditional hypotheses is the limit $N\to\infty$ with $\lambda>0$ and $A$ fixed.
	
	\subsection{Related literature}
	
	Score-based Landau particle methods and density-level relative-entropy estimates were developed in \cite{HuangWang2025,IlinHuWang2025}. A Coulomb relative-entropy stability estimate appears in \cite{Ilin2025}, while relative-entropy weak--strong stability between physical Landau solutions, including the Coulomb case, was established in \cite{TabaryWeakStrong2025}; in the Coulomb case, global existence of smooth physical solutions was recently established in \cite{GuillenSilvestre2025}. These estimates explain how a prescribed score error affects a density-level equation, but not convergence of the repeated empirical score minimization and trajectory-level system solves considered here. \Cref{thm:main-outer,thm:inexact} control a different quantity: the deviation of the computed score from the regularized empirical minimizer, through the objective gap \eqref{eq:intro-gap}, and the resulting trajectory error of the alternating scheme. They do not bound the error relative to the score of the physical density, so the two sets of results are complementary.
	
	The Landau equation also admits a gradient-flow formulation \cite{CarrilloDelgadinoDesvillettesWu2024}. The closest analytic precedent is \cite{CarrilloDelgadinoWu2023}, which studies the regularized Landau equation with the mollified score $\nabla G^\varepsilon*\log(f*G^\varepsilon)$: its Lemma~2.5 establishes the Lipschitz estimate for $A(v-w)(s(v)-s(w))$ with weight $\abs{v-w}^{\gamma+2}$ that underlies our \cref{sec:singular-effective-interaction}, and its Theorem~1.3 proves the mean-field limit of the associated particle system in $W_\infty$ for $\gamma\in(-3,0]$, up to the maximal existence time of the regularized continuum solution (global for $\gamma\ge-2$), with an initial minimal-separation condition imposed only for $\gamma\in(-3,-2)$. The present paper differs in the regularization and in the object analyzed: the score is the minimizer of a data-dependent empirical objective rather than a mollified log-density, so its stability in the input measure (\cref{thm:main-score}) must be proved before any of these estimates can be used, and what is analyzed is the alternating iteration rather than the coupled system. \Cref{sec:singular-effective-interaction} makes the transfer explicit and adds two statements without counterpart in \cite{CarrilloDelgadinoWu2023}: the RKHS obstruction below $\gamma=-2$ (\cref{prop:coulomb-effective-obstruction}) and the common-label Coulomb continuum alternating iteration (\cref{thm:continuum-Coulomb-waveform}). We do not obtain a Tikhonov analogue of the very-soft-potential particle theorem of \cite{CarrilloDelgadinoWu2023}; see \cref{rem:coulomb-next-step}. RKHS estimation of log-density gradients inside deterministic Fokker--Planck particle flows was proposed in \cite{MaoutsaReichOpper2020}, where convergence of the interacting empirical procedure was left open. For Stein variational gradient descent, a deterministic interacting particle method in which integration by parts against a kernel replaces the score of the empirical measure, the mean-field limit was established in \cite{LuLuNolen2019}. At the density level, the regularized Stein variational gradient flow of \cite{HeBalasubramanianSriperumbudurLu2025} uses a Tikhonov-type kernel regularization of the score as its velocity; that work analyzes the resulting flow and its time discretization, whereas here the score enters the Landau interaction, is computed from empirical measures, and the object analyzed is the alternating iteration. Related score-error and self-consistency analyses include \cite{BoffiVandenEijnden2023,LuWuXiang2024,ShenEtAl2022,LiHuraultSolomon2023,IlinSushkoHu2026}, and a global-in-time joint neural parameterization of the Landau score and particle flow is studied in \cite{KimEtAl2026}.
	
	Factorial Volterra decay for whole-interval waveform relaxation is classical \cite{MiekkalaNevanlinna1987}; in particular, \cite{Botchev2024} proves $\mathcal O((Ct)^k/k!)$ convergence for nonlinear ODE splittings, including residual and inexact-solve errors. In this terminology, \eqref{eq:intro-iteration} is a waveform relaxation in which only the score argument is lagged. That theory does not include a data-dependent empirical score minimization, RKHS regularization, particle-number-uniform force estimates, or a mean-field limit. Closer in structure, \cite{SzpruchTanTse2019} combines a Picard iteration in the law with particle approximation for McKean--Vlasov SDEs; there the dependence of the drift on the law is prescribed, whereas here it passes through an empirical score minimization. Our contribution is the stability of the empirical score map needed to close this Volterra mechanism for the Landau interaction.
	
	Vector-valued RKHS score matching has operator and representer formulations as well as static consistency theory \cite{ZhouShiZhu2020}, and Lipschitz stability of Tikhonov-regularized kernel risk minimizers with respect to the data distribution under the Fortet--Mourier metric has been studied in \cite{ZhangXuSun2024} for pointwise loss functions; the score-matching objective involves $\Div s$ and is not of that form. We use this regularization framework to derive deterministic $W_1$ stability of the score-matching minimizer, in $L^\infty$ and with explicit constants, for input measures with finite first moment, and to control its feedback into the evolving particle system. The fixed-$N$ analysis further makes explicit the compatibility condition for removing regularization and the distinction between convergence at the sample points and convergence of the full score in $\Hh$. Propagation of chaos for stochastic Landau particle systems is classical for Maxwell molecules, hard potentials, and moderately soft potentials \cite{FournierGuillin2017,FournierHauray2016} and has recently been extended to singular potentials \cite{TabaryChaos2025,FengWang2025}; these results do not apply directly to the deterministic learned-score flow studied here or remove its Tikhonov parameter. Recent asymptotic-preserving JKO work observes that a time-discrete fixed-point implementation of score matching can converge to an incorrect optimizer if the dependence of the future density on the score is omitted \cite{HuangWang2026AP}; that variational fixed point differs from our trajectory-level alternating iteration. An on-the-fly neural score method for the Vlasov--Maxwell--Landau system preserves the collision invariants and dissipates an estimated entropy \cite{IlinHu2026VML}, but its score--system iteration is not analyzed.
	
	The individual tools used here are classical: the Hyv\"arinen identity, Tikhonov regularization, Picard--Volterra iteration, and Dobrushin-type stability. The central analytical task is to control the feedback between score estimation and particle evolution. At fixed $\lambda>0$, the score-stability estimate closes the exact and inexact iteration bounds with constants independent of $N$ and yields a compatible mean-field limit under the stated kernel assumptions. For the three-dimensional Coulomb kernel, integrated estimates in common initial labels allow us to construct a common local class of bi-Lipschitz flows and prove convergence despite the discontinuity of the pair interaction. The empirical compatibility theorem and the separate conditional density-level limit identify the distinct requirements for removing score regularization.
	
	\subsection{Notation and organization}
	
	We write $\abs\cdot$ for the Euclidean norm and the induced matrix norm, $\norm{A}_{L^\infty}:=\sup_z\abs{A(z)}$, $I_d$ for the identity, $(z\otimes w)_{ij}:=z_iw_j$, and $A\succeq B$ if $A-B$ is positive semidefinite. For bounded $s$, $\norm{s}_{L^\infty}=\norm{s}_\infty :=\sup_x\abs{s(x)}$ and $\Lips(s):=\sup_{x\ne y}\abs{s(x)-s(y)}/\abs{x-y}$; $Ds$, $D^2s$, and $\Div s$ are the first and second derivatives and the divergence, and $C_b^k$ is the space of $C^k$ functions with bounded derivatives up to order $k$. For a kernel $K(x,z)$, $\partial_a^{(1)}K:=\partial_{x_a}K$ and $\partial_a^{(2)}K:=\partial_{z_a}K$, componentwise. For a bounded operator $T$ between Hilbert spaces, $T^*$, $T^\dagger$, $\operatorname{Ran}T$, and $\ker T$ are the adjoint, the Moore--Penrose pseudoinverse, the range, and the kernel. $\Pone(\R^d)$ is the set of probability measures with finite first moment, $\Gamma(\mu,\nu)$ the set of couplings, $W_1(\mu,\nu):=\inf_{\pi\in\Gamma(\mu,\nu)}\int\abs{v-w}\dd\pi$, and $\Phi_\#\mu:=\mu\circ\Phi^{-1}$. For configurations $Z=(z_i)_{i=1}^N$ and $Y=(y_i)_{i=1}^N$,
	\begin{equation}\label{eq:particle-norm}
		\abs{Z-Y}_N:=\frac1N\sum_{i=1}^N\abs{z_i-y_i}.
	\end{equation}
	We write $f\dd v$ for the measure with density $f$, $\norm{h}_{L^2(f)}^2:=\int\abs h^2f\dd v$, $\langle v\rangle:=(1+\abs v^2)^{1/2}$, and $\KLdiv(g\Vert f):=\int g\log(g/f)\dd v$, set to $+\infty$ unless $g\ll f$. Finally, $a_\lambda=O_T(b_\lambda)$ means $\abs{a_\lambda}\le C_T\abs{b_\lambda}$ for all small $\lambda>0$ with $C_T$ independent of $\lambda$; generic constants $C$ may change from line to line, and independence of $N$ or $\lambda$ is always stated explicitly.
	
	\Cref{sec:framework} states the assumptions and main theorems. \Cref{sec:score-identification,sec:stable-score-learner} treat score identification and the regularized score map, \cref{sec:system-step,sec:exact-outer} prove convergence of the alternating iteration, and \cref{sec:mean-field,sec:inexact,sec:singular-effective-interaction,sec:landau-connection} treat the mean-field limit, inexact solves, singular kernels, and the connection to the Landau equation. \Cref{sec:conclusion} summarizes the open problems. Numerical experiments are reported in Appendix~\ref{sm:numerics}.
	
	\section{Mathematical framework and main results}
	\label{sec:framework}
	
	\subsection{Score space and collision kernel}
	
	\begin{assumption}[Hypotheses on the score space]\label{ass:score-space}
		The space $\Hh$ is a separable Hilbert space of vector fields $s:\R^d\to\R^d$, continuously embedded in $C_b^2(\R^d;\R^d)$. There are finite constants $c_0,c_1,c_D,c_2,c_H$ such that, for every $s\in\Hh$,
		\begin{align}
			\norm{s}_{L^\infty}
			&\le c_0\norm{s}_{\Hh},
			&
			\Lips(s)
			&\le c_1\norm{s}_{\Hh},
			\label{eq:rkhs-bounds-1}
			\\
			\norm{\Div s}_{L^\infty}
			&\le c_D\norm{s}_{\Hh},
			&
			\Lips(\Div s)
			&\le c_2\norm{s}_{\Hh},
			\label{eq:rkhs-bounds-2}
			\\
			\norm{D^2s}_{L^\infty}
			&\le c_H\norm{s}_{\Hh}.
			\label{eq:rkhs-bounds-Hessian}
		\end{align}
	\end{assumption}
	
	The estimates for score stability and convergence of the exact iteration use $c_0,c_1,c_D,c_2$; the Hessian bound justifies differentiation of kernel sections in $\Hh$ (\cref{sec:stable-score-learner}). Since point evaluation is continuous on $\Hh$, the Riesz representation theorem gives a matrix-valued reproducing kernel $K:\R^d\times\R^d\to\R^{d\times d}$ with $K(\cdot,x)c\in\Hh$ and $\ip{K(\cdot,x)c}{s}_{\Hh}=c\cdot s(x)$ for $x,c\in\R^d$ and $s\in\Hh$. For instance, for each $\sigma>0$ the Gaussian RKHS with kernel $K_\sigma(x,y)=e^{-\abs{x-y}^2/(2\sigma^2)}I_d$ is separable and continuously embedded in $C_b^2$, so it satisfies \cref{ass:score-space}. The space $\Hh$ is fixed throughout; in particular it does not change with $N$ or during the alternating iteration.
	
	\begin{assumption}[Regularized collision kernel]\label{ass:kernel}
		The matrix $A:\R^d\to\R^{d\times d}$ is bounded and globally Lipschitz, and we write $a_0:=\norm{A}_{L^\infty}$ and $a_1:=\Lips(A)$. For the conservation statements only, we additionally assume
		\begin{equation}\label{eq:A-structure}
			A(-z)=A(z)=A(z)^\top,
			\qquad
			A(z)z=0,
			\qquad z\in\R^d.
		\end{equation}
	\end{assumption}
	
	These hypotheses describe a cutoff or mollified collision model, not the unregularized Coulomb matrix. For example, if $\chi_R\in C_c^\infty([0,\infty))$ satisfies $0\le\chi_R\le1$, $\chi_R=1$ on $[0,R]$, and $\chi_R=0$ on $[2R,\infty)$, then for fixed $(\gamma,\eta,R)$ with $\eta,R>0$ the matrix $C_\gamma\chi_R(\abs z)(\abs z^2+\eta^2)^{\gamma/2}(\abs z^2I_d-z\otimes z)$ satisfies \cref{ass:kernel}, including \eqref{eq:A-structure}; so does
	\begin{equation}\label{eq:numerical-kernel}
		A_{\mathrm{test}}(z)
		:=\frac{\abs z^2I_d-z\otimes z}{1+\abs z^2},
	\end{equation}
	which is used in the numerical experiments. The particle force is \eqref{eq:intro-force}.
	
	\subsection{The alternating score--system iteration}
	
	\begin{alg}[Alternating score--system iteration]
		\label{alg:alternating}
		Fix $Z^0\in(\R^d)^N$ and $T>0$, and set $X_t^0\equiv Z^0$ on $[0,T]$. Given the $k$th iterate $X^k\in C([0,T];(\R^d)^N)$:
		\begin{enumerate}[label=\arabic*.,leftmargin=2em]
			\item \emph{Score step.} For every $t\in[0,T]$, set
			\begin{equation}\label{eq:score-step}
				s_t^k:=S_\lambda(\mu_{X_t^k}).
			\end{equation}
			\item \emph{System step.} With the score path $s^k$ frozen, solve
			\begin{equation}\label{eq:system-step}
				\dot X_t^{k+1}=F(X_t^{k+1};s_t^k),
				\qquad
				X_0^{k+1}=Z^0.
			\end{equation}
		\end{enumerate}
	\end{alg}
	
	We call $X^k$ the \emph{$k$th iterate} and refer to \eqref{eq:score-step} and \eqref{eq:system-step} as the \emph{score step} and the \emph{system step}; the term \emph{time step} is reserved for time discretization. Each iteration consists of one score step followed by one system step, and $k$ counts the iterations. \Cref{alg:alternating} solves both subproblems exactly, while \cref{sec:inexact} allows both to be solved approximately.
	
	\subsection{Main results}
	
	For the statements below, set
	\[
	R_\lambda:=\frac{c_D}{\lambda}, \qquad M_\lambda:=c_0R_\lambda,\qquad
	L_\lambda:=c_1R_\lambda, \qquad L_{S,\lambda}:=\frac{c_0}{\lambda} \bigl(c_2+2c_0c_1R_\lambda\bigr),
	\]
	\[
	V_\lambda:=2a_0M_\lambda, \quad L_{F,\lambda}:=2(2a_1M_\lambda+a_0L_\lambda), \quad \beta_\lambda:=2a_0L_{S,\lambda},
	\quad C_{\mathrm{mf}}:=L_{F,\lambda}+\beta_\lambda.
	\]
	These constants depend only on $\Hh$, $A$, and $\lambda$, and not on $N$. We use the convention $\beta_\lambda^0=1$, including when $\beta_\lambda=0$.
	
	\begin{theorem}[Well-posedness and $W_1$-stability of the regularized score]
		\label{thm:main-score}
		Let \cref{ass:score-space} hold, and fix $\lambda>0$. Then, for every $\mu\in\Pone(\R^d)$, the functional $\J_\mu^\lambda$ has a unique minimizer $S_\lambda(\mu)$, which satisfies
		\begin{equation}\label{eq:main-score-bound}
			\norm{S_\lambda(\mu)}_{\Hh}\le R_\lambda,
			\qquad
			\norm{S_\lambda(\mu)}_{L^\infty}\le M_\lambda,
			\qquad
			\Lips(S_\lambda(\mu))\le L_\lambda.
		\end{equation}
		Moreover, for every $\mu,\nu\in\Pone(\R^d)$,
		\begin{equation}\label{eq:main-score-stability}
			\norm{S_\lambda(\mu)-S_\lambda(\nu)}_{L^\infty}
			\le L_{S,\lambda}W_1(\mu,\nu).
		\end{equation}
	\end{theorem}
	
	\begin{theorem}[Factorial convergence of the alternating iteration]
		\label{thm:main-outer}
		Let \cref{ass:score-space,ass:kernel} hold, and fix $\lambda>0$, $T>0$, $N\ge1$, and $Z^0\in(\R^d)^N$. Then every iterate of \cref{alg:alternating} is uniquely defined on $[0,T]$, and as $k\to\infty$ the iterates converge uniformly on $[0,T]$ to the unique path $X^{\lambda,N}\in C^1([0,T];(\R^d)^N)$ satisfying
		\begin{equation}\label{eq:main-coupled-particle}
			\dot X_t^{\lambda,N}
			=F\bigl(X_t^{\lambda,N};S_\lambda(\mu_{X_t^{\lambda,N}})\bigr),
			\qquad
			X_0^{\lambda,N}=Z^0.
		\end{equation}
		More precisely, with $D_k(t):=\abs{X_t^{k+1}-X_t^k}_N$, for every $k\ge0$ and $0\le t\le T$,
		\begin{equation}\label{eq:main-factorial-increment}
			\sup_{0\le r\le t}D_k(r)
			\le
			V_\lambda e^{L_{F,\lambda}t}
			\frac{\beta_\lambda^kt^{k+1}}{(k+1)!}.
		\end{equation}
		Moreover, comparison with the limiting path gives
		\begin{equation}\label{eq:main-factorial-tail}
			\sup_{0\le t\le T}
			\abs{X_t^k-X_t^{\lambda,N}}_N
			\le
			\mathcal E_{k,\lambda}(T)
			:=
			V_\lambda e^{L_{F,\lambda}T}
			\frac{\beta_\lambda^kT^{k+1}}{(k+1)!}.
		\end{equation}
		All constants are independent of $N$.	If, in addition, \eqref{eq:A-structure} holds,	every system step and the limiting path preserve the initial momentum and kinetic energy:
		\begin{equation}\label{eq:main-particle-invariants}
			\frac1N\sum_{i=1}^NX_{t,i}^{k+1}
			=\frac1N\sum_{i=1}^Nz_i^0,
			\qquad
			\frac1N\sum_{i=1}^N\abs{X_{t,i}^{k+1}}^2
			=\frac1N\sum_{i=1}^N\abs{z_i^0}^2,
		\end{equation}
		and likewise with $X^{k+1}$ replaced by $X^{\lambda,N}$.
	\end{theorem}
	
	For $\mu\in\Pone(\R^d)$, define the regularized mean-field drift
	\begin{equation}\label{eq:main-mean-field-drift}
		b_\lambda[\mu](v)
		:=-\int_{\R^d}A(v-w)\bigl(S_\lambda(\mu)(v)-S_\lambda(\mu)(w)\bigr)\dd\mu(w).
	\end{equation}
	A curve $\mu\in C([0,T];\Pone(\R^d))$ is a \emph{characteristic solution} of \eqref{eq:main-mean-field} below if $\mu_t=(\Phi_t)_\#\mu_0$, where $\dot\Phi_t(v)=b_\lambda[\mu_t](\Phi_t(v))$ and $\Phi_0(v)=v$.  Since $b_\lambda[\mu_t]$ is bounded, Lipschitz in $v$ uniformly in $t$, and continuous in $t$ along $W_1$-continuous curves (\cref{lem:mean-field-drift}), every curve in $C([0,T];\Pone(\R^d))$ solving \eqref{eq:main-mean-field} in the sense of distributions is characteristic, by uniqueness for linear continuity equations with such fields \cite[section~8.1]{AmbrosioGigliSavare2008}.
	
	\begin{theorem}[Well-posedness of the regularized mean-field equation and mean-field limit]
		\label{thm:main-mean-field}
		Let \cref{ass:score-space,ass:kernel} hold, and fix $\lambda>0$ and $T>0$. For every $\mu_0\in\Pone(\R^d)$, the equation
		\begin{equation}\label{eq:main-mean-field}
			\partial_t\mu_t+\nabla\cdot\bigl(b_\lambda[\mu_t]\mu_t\bigr)=0,
			\qquad
			\mu_{t=0}=\mu_0,
		\end{equation}
		has a unique characteristic solution in $C([0,T];\Pone(\R^d))$, and the characteristic solutions $\mu,\nu$ with initial laws $\mu_0,\nu_0$ satisfy
		\begin{equation}\label{eq:main-Dobrushin}
			W_1(\mu_t,\nu_t)\le e^{C_{\mathrm{mf}}t}W_1(\mu_0,\nu_0),
			\qquad 0\le t\le T.
		\end{equation}
		If $\mu_0^N=N^{-1}\sum_i\delta_{z_i^0}$ and $X^{\lambda,N}$ solves \eqref{eq:main-coupled-particle} with $Z^0=(z_1^0,\ldots,z_N^0)$, then
		\begin{equation}\label{eq:main-particle-limit}
			\sup_{0\le t\le T}W_1(\mu_{X_t^{\lambda,N}},\mu_t)
			\le e^{C_{\mathrm{mf}}T}W_1(\mu_0^N,\mu_0).
		\end{equation}
		In particular, the left-hand side tends to zero as $N\to\infty$, with $\lambda$ and $A$ fixed, whenever $W_1(\mu_0^N,\mu_0)\to0$.
	\end{theorem}
	
	\begin{corollary}[Commutation of the alternating-iteration and mean-field limits]
		\label{cor:main-joint-limit}
		Under the assumptions of \cref{thm:main-outer,thm:main-mean-field}, let $X^{k,N}$ denote the $k$th iterate with empirical initial law $\mu_0^N$. Then
		\begin{equation}\label{eq:main-joint-error}
			\sup_{0\le t\le T}
			W_1(\mu_{X_t^{k,N}},\mu_t)
			\le
			\mathcal E_{k,\lambda}(T)
			+
			e^{C_{\mathrm{mf}}T}
			W_1(\mu_0^N,\mu_0).
		\end{equation}
		In particular, $\sup_{t\le T}W_1(\mu_{X_t^{k_\ell,N_\ell}},\mu_t)\to0$ for every sequence with $k_\ell\to\infty$ and $W_1(\mu_0^{N_\ell},\mu_0)\to0$.  Moreover, by \cref{thm:continuum-outer-iteration}, both iterated limits exist and equal $\mu$, so the alternating-iteration and mean-field limits commute at fixed $\lambda$ and $A$.
	\end{corollary}
	
	The limit equation \eqref{eq:main-mean-field} is a regularized learned-score equation, not the physical Landau equation; its connection to the latter is discussed in \cref{sec:landau-connection}.
	
	\section{Score identification and an empirical obstruction}
	\label{sec:score-identification}
	
	We restrict throughout to $-d\le\gamma\le1$, so that the Coulomb interaction $\gamma=-d$ is the most singular case considered (parts of the literature also treat $\gamma\ge-d-1$). By \eqref{eq:intro-draft-landau-continuity}, the Landau velocity field depends on $f_t$ only through $f_t$ and its score. The following standard identity, with its boundary assumption made explicit, underlies implicit score matching.
	
	\begin{proposition}\label{prop:density-level-identity}
		Let $f>0$ be a $C^1$ probability density with finite Fisher information,  $\int\abs{\nabla\log f}^2f\dd v<\infty$.  Let $s\in C^1(\R^d;\R^d)$ satisfy
		\begin{equation}\label{eq:density-level-identity-hypotheses}
			\begin{gathered}
				s\in L^2(f\dd v),\qquad
				(\Div s)f\in L^1(\R^d),\qquad
				s\cdot\nabla f\in L^1(\R^d),\\
				\int(\Div s)f\dd v=-\int s\cdot\nabla f\dd v .
			\end{gathered}
		\end{equation}
		Then
		\begin{equation}\label{eq:density-level-identity}
			\int\bigl(\abs{s}^2+2\Div s\bigr)f\dd v
			=\int\abs{s-\nabla\log f}^2f\dd v
			-\int\abs{\nabla\log f}^2f\dd v.
		\end{equation}
		Consequently, over a class of vector fields satisfying \eqref{eq:density-level-identity-hypotheses} and containing $\nabla\log f$ (which requires $f\in C^2$), the density-level minimizer is $\nabla\log f$, uniquely as an element of $L^2(f)$. The hypotheses \eqref{eq:density-level-identity-hypotheses} hold when $s$ is compactly supported and, by \cref{lem:cutoff-ibp} below, whenever $s$ and $\Div s$ are bounded; in particular they hold for every $s\in\Hh$ under \cref{ass:score-space}.
	\end{proposition}
	
	\begin{proof}
		All integrals in \eqref{eq:density-level-identity} are finite because $s,\nabla\log f\in L^2(f\dd v)$ and $(\Div s)f\in L^1$.  The last hypothesis in \eqref{eq:density-level-identity-hypotheses} gives 
		\[
		\int(\Div s)f\dd v=-\int s\cdot\nabla f\dd v =-\int s\cdot\nabla\log f\,f\dd v .
		\]
		Hence
		\begin{align*}
			\int\abs{s-\nabla\log f}^2f\dd v &=\int\abs s^2f\dd v-2\int s\cdot\nabla\log f\,f\dd v +\int\abs{\nabla\log f}^2f\dd v \\
			&=\int\bigl(\abs s^2+2\Div s\bigr)f\dd v +\int\abs{\nabla\log f}^2f\dd v,
		\end{align*}
		which is \eqref{eq:density-level-identity}.  Since the last term does not depend on $s$, the left-hand side of \eqref{eq:density-level-identity} is minimized over such a class exactly when $\norm{s-\nabla\log f}_{L^2(f)}=0$.
	\end{proof}
	
	\begin{lemma}\label{lem:cutoff-ibp}
		Let $f>0$ be a $C^1$ probability density with finite Fisher information, and let $s\in C^1(\R^d;\R^d)$ be such that $s$ and $\Div s$ are bounded.  Then $\nabla f\in L^1(\R^d)$ and the hypotheses \eqref{eq:density-level-identity-hypotheses} hold; in particular $\int(\Div s)f\dd v=-\int s\cdot\nabla f\dd v$.
	\end{lemma}
	\begin{proof}
		By the Cauchy--Schwarz inequality,
		\[
		\int\abs{\nabla f}\dd v	=\int\abs{\nabla\log f}\,f\dd v	\le\Bigl(\int\abs{\nabla\log f}^2f\dd v\Bigr)^{1/2}<\infty .
		\]
		Since $s$ and $\Div s$ are bounded and $\int f\dd v=1$, it follows that	$s\in L^2(f\dd v)$, $(\Div s)f\in L^1$, and $s\cdot\nabla f\in L^1$. For the integration by parts, let $\chi\in C_c^\infty(\R^d)$ satisfy $0\le\chi\le1$, $\chi=1$ on $B_1$, and $\operatorname{supp}\chi\subset B_2$, and put $\chi_R(v):=\chi(v/R)$. The divergence theorem applied to the compactly supported $C^1$ field $f\chi_Rs$ gives
		\[
		\int\chi_R(\Div s)f\dd v +\int f\,s\cdot\nabla\chi_R\dd v =-\int\chi_R\,s\cdot\nabla f\dd v .
		\]
		Since $\abs{\nabla\chi_R}\le C/R$, the middle term is bounded by $C\norm{s}_{L^\infty}/R\to0$. The remaining two terms converge by dominated	convergence, with integrable majorants $\norm{\Div s}_{L^\infty}f$ and $\norm{s}_{L^\infty}\abs{\nabla f}$. Letting $R\to\infty$ proves the claim.
	\end{proof}

	Density-level identification does not justify replacing $f\dd v$ by a finite empirical measure while retaining an unrestricted function class.
	\begin{remark}[The unregularized empirical objective is unbounded below]
		\label{rem:empirical-minus-infinity}
		Let $Z=(z_1,\ldots,z_N)$ have distinct locations $y_1,\ldots,y_m$, and let $\chi_\ell\in C_c^\infty(\R^d)$ equal one near $y_\ell$ and have disjoint  supports.  The fields	$s_M(x)=-\frac Md\sum_\ell\chi_\ell(x)(x-y_\ell)\in C_c^\infty(\R^d;\R^d)$	satisfy $s_M(y_\ell)=0$ and $\Div s_M(y_\ell)=-M$, so the unregularized	empirical objective $\J^0_{\mu_Z}(s_M)$ from the introduction equals $-2M\to-\infty$.  This is the function-space counterpart of the ill-posedness of the empirical score-matching system noted in \cite[Section~4]{SriperumbudurEtAl2017}.  By contrast, under \cref{ass:score-space} one has $\J_\mu^\lambda(s)\ge\lambda\norm s_{\Hh}^2-2c_D\norm s_{\Hh}\ge-c_D^2/\lambda$ for every $s\in\Hh$ and every probability measure $\mu$, so the regularized objective is bounded below. A fixed finite neural architecture may also prevent the construction above, but this alone yields neither a unique minimizer nor a data-Lipschitz minimizer map.
	\end{remark}
	
	\section{Well-posedness and stability of the regularized score map}\label{sec:stable-score-learner}
	
	Throughout this section, \cref{ass:score-space} is in force. By \eqref{eq:rkhs-bounds-1}--\eqref{eq:rkhs-bounds-2}, $\int\abs s^2\dd\mu\le c_0^2\norm s_{\Hh}^2$ and $\int\abs{\Div s}\dd\mu\le c_D\norm s_{\Hh}$ for every $s\in\Hh$ and every Borel probability measure $\mu$, so $\J_\mu^\lambda$ is finite on $\Hh$ without any moment assumption; the first moment is used only for stability in $W_1$.
	
	\begin{proposition}\label{prop:score-existence-gap}
		Under \cref{ass:score-space}, $\J_\mu^\lambda$ has a unique minimizer
		\[
		S_\lambda(\mu):=\argmin_{s\in\Hh}\J_\mu^\lambda(s).
		\]
		It is characterized by
		\begin{equation}\label{eq:euler-score}
			\lambda\ip{S_\lambda(\mu)}{h}_{\Hh}	+\int S_\lambda(\mu)\cdot h\dd\mu+\int\Div h\dd\mu=0, \qquad h\in\Hh.
		\end{equation}
		Moreover,
		\begin{equation}\label{eq:score-H-bound}
			\norm{S_\lambda(\mu)}_{\Hh}\le R_\lambda=\frac{c_D}{\lambda},
		\end{equation}
		and for every $s\in\Hh$,
		\begin{equation}\label{eq:exact-gap}
			\J_\mu^\lambda(s)-\J_\mu^\lambda(S_\lambda(\mu))=\int\abs{s-S_\lambda(\mu)}^2\dd\mu+\lambda\norm{s-S_\lambda(\mu)}_{\Hh}^2.
		\end{equation}
	\end{proposition}
	
	\begin{proof}
		For $s,h\in\Hh$, define
		\[
		B_\mu(s,h):=\lambda\ip{s}{h}_{\Hh}+\int s\cdot h\dd\mu,
		\qquad
		\ell_\mu(h):=-\int\Div h\dd\mu,
		\]
		so that $\J_\mu^\lambda(s)=B_\mu(s,s)-2\ell_\mu(s)$.  By \eqref{eq:rkhs-bounds-1}--\eqref{eq:rkhs-bounds-2}, $\abs{B_\mu(s,h)}\le(\lambda+c_0^2)\norm s_{\Hh}\norm h_{\Hh}$, $B_\mu(s,s)\ge\lambda\norm s_{\Hh}^2$, and $\abs{\ell_\mu(h)}\le c_D\norm h_{\Hh}$.  The Lax--Milgram theorem therefore gives a unique $S\in\Hh$ with $B_\mu(S,h)=\ell_\mu(h)$ for all $h\in\Hh$, which is \eqref{eq:euler-score}.  For $u\in\Hh$, expanding the quadratic functional and using this identity to cancel the cross terms gives
		\[
		\J_\mu^\lambda(S+u)-\J_\mu^\lambda(S)=B_\mu(u,u)
		=\int\abs u^2\dd\mu+\lambda\norm u_{\Hh}^2 .
		\]
		With $u=s-S$, this is \eqref{eq:exact-gap}; in particular $S$ is the unique minimizer, $S=S_\lambda(\mu)$.  Finally, taking $h=S_\lambda(\mu)$ in \eqref{eq:euler-score} gives
		\[
		\lambda\norm{S_\lambda(\mu)}_{\Hh}^2
		\le\lambda\norm{S_\lambda(\mu)}_{\Hh}^2+\int\abs{S_\lambda(\mu)}^2\dd\mu
		=-\int\Div S_\lambda(\mu)\dd\mu
		\le c_D\norm{S_\lambda(\mu)}_{\Hh},
		\]
		which proves \eqref{eq:score-H-bound}.
	\end{proof}
	
	\begin{proposition}\label{prop:score-W1}
		Under \cref{ass:score-space}, for all $\mu,\nu\in\Pone(\R^d)$,
		\begin{equation}\label{eq:score-W1}
			\norm{S_\lambda(\mu)-S_\lambda(\nu)}_{L^\infty}
			\le L_{S,\lambda}W_1(\mu,\nu),
		\end{equation}
		with $L_{S,\lambda}$ as in \cref{sec:framework}. The stronger Hilbert-space estimate is
		\begin{equation}\label{eq:score-H-W1}
			\norm{S_\lambda(\mu)-S_\lambda(\nu)}_{\Hh}
			\le\frac1\lambda
			\left(c_2+2c_0c_1\frac{c_D}{\lambda}\right)W_1(\mu,\nu).
		\end{equation}
	\end{proposition}
	
	\begin{proof}
		Write $S_\mu:=S_\lambda(\mu)$, $S_\nu:=S_\lambda(\nu)$, and $q:=S_\mu-S_\nu$. Subtracting the Euler equations \eqref{eq:euler-score} for $\mu$ and $\nu$, taking $h=q$, and using $\int S_\mu\cdot q\dd\mu-\int S_\nu\cdot q\dd\nu =\int\abs q^2\dd\mu+\int S_\nu\cdot q\dd(\mu-\nu)$ gives
		\begin{equation}\label{eq:score-difference-proof}
			\lambda\norm q_{\Hh}^2
			+\int\abs q^2\dd\mu
			=-\int\bigl(S_\nu\cdot q+\Div q\bigr)\dd(\mu-\nu).
		\end{equation}
		By \cref{ass:score-space} and \eqref{eq:score-H-bound}, $\norm{S_\nu}_{L^\infty}\le c_0R_\lambda$, $\Lips(S_\nu)\le c_1R_\lambda$, $\norm q_{L^\infty}\le c_0\norm q_{\Hh}$, $\Lips(q)\le c_1\norm q_{\Hh}$, and $\Lips(\Div q)\le c_2\norm q_{\Hh}$, so the product rule for Lipschitz constants gives
		\[
		\Lips(S_\nu\cdot q+\Div q)\le\Bigl(c_2+\frac{2c_0c_1c_D}{\lambda}\Bigr)\norm q_{\Hh}.
		\]
		By Kantorovich--Rubinstein duality and $\int\abs q^2\dd\mu\ge0$, \eqref{eq:score-difference-proof} yields
		\[
		\lambda\norm q_{\Hh}^2\le(c_2+2c_0c_1c_D/\lambda)\norm q_{\Hh}W_1(\mu,\nu).
		\]
		Dividing by $\lambda\norm q_{\Hh}$ when $q\ne0$ proves \eqref{eq:score-H-W1}, and $\norm q_{L^\infty}\le c_0\norm q_{\Hh}$ gives \eqref{eq:score-W1}.
	\end{proof}
	
	\begin{proof}[Proof of \cref{thm:main-score}]
		Existence, uniqueness, and the $\Hh$-bound in \eqref{eq:main-score-bound} are contained in \cref{prop:score-existence-gap}; the other two bounds follow from \eqref{eq:rkhs-bounds-1}; and \eqref{eq:main-score-stability} is \eqref{eq:score-W1}.
	\end{proof}
	
	\subsection{Representer formula for empirical measures}
	
	The abstract score step has a finite representer formula.  Let $K:\R^d\times\R^d\to\R^{d\times d}$ be the operator-valued reproducing kernel of $\Hh$, with $K(x,y)=K(y,x)^\top$ and the convention $\ip{K(\cdot,x)c}{h}_{\Hh}=c\cdot h(x)$.  Let $d_x$ denote the Riesz representer of the continuous functional $h\mapsto\Div h(x)$.  The kernel sections are differentiable in the $\Hh$-norm with respect to their second argument: for $c\in\R^d$ and a coordinate direction $e_a$, Taylor's formula and \eqref{eq:rkhs-bounds-Hessian} give
	\[
	\abs{\frac{c\cdot h(x+te_a)-c\cdot h(x)}{t}-c\cdot\partial_ah(x)}
	\le\frac{c_H\abs c\,\abs t}{2}\norm h_{\Hh},
	\qquad h\in\Hh,
	\]
	so the difference quotients $t^{-1}\bigl(K(\cdot,x+te_a)-K(\cdot,x)\bigr)c$ converge in $\Hh$ as $t\to0$ to the representer of $h\mapsto c\cdot\partial_ah(x)$, and this limit is $\partial_a^{(2)}K(\cdot,x)c$ because convergence in $\Hh$ implies pointwise convergence.  Consequently,
	\begin{equation}\label{eq:divergence-representer}
		d_x=\sum_{a=1}^d\partial_a^{(2)}K(\cdot,x)e_a\in\Hh,\qquad \ip{d_x}{h}_{\Hh}=\Div h(x).
	\end{equation}
	For scalar kernels, derivative reproducing properties of this type are established in \cite{Zhou2008}. The following result is a reformulation of \cite[Theorem 3.1]{ZhouShiZhu2020} in terms of nodal values; we include a short proof in our notation.
	
	\begin{proposition}[Representer formula	{\cite[Theorem 3.1]{ZhouShiZhu2020}}]	\label{prop:rkhs-representer}
		Let $\mu_Z$ be the empirical measure $N^{-1}\sum_{i=1}^N\delta_{z_i}$, and put $u_i=S_\lambda(\mu_Z)(z_i)\in\R^d$.  Define the block matrix $G_Z\in\R^{dN\times dN}$ and vector $q_Z\in\R^{dN}$ by
		\begin{equation}\label{eq:block-Gram}
			(G_Z)_{ji}=K(z_j,z_i),	\qquad 	(q_Z)_j=\sum_{i=1}^Nd_{z_i}(z_j).
		\end{equation}
		Then
		\begin{equation}\label{eq:representer-linear-system}
			(G_Z+N\lambda I_{dN})u=-q_Z,
		\end{equation}
		and, for every $x\in\R^d$,
		\begin{equation}\label{eq:representer-score}
			S_\lambda(\mu_Z)(x)	=-\frac1{N\lambda}	\left(\sum_{i=1}^NK(x,z_i)u_i+\sum_{i=1}^Nd_{z_i}(x)\right).
		\end{equation}
		The matrix $G_Z+N\lambda I_{dN}$ is symmetric positive definite, so \eqref{eq:representer-linear-system} has a unique solution.
	\end{proposition}
	
	\begin{proof}
		Write $S=S_\lambda(\mu_Z)$.  By the reproducing identities $S(z_i)\cdot h(z_i)=\ip{K(\cdot,z_i)u_i}{h}_{\Hh}$ and $\Div h(z_i)=\ip{d_{z_i}}{h}_{\Hh}$, the Euler equation \eqref{eq:euler-score} for $\mu_Z$ reads
		\[
		\ip{
			\lambda S
			+\frac1N\sum_{i=1}^N K(\cdot,z_i)u_i
			+\frac1N\sum_{i=1}^N d_{z_i}
		}{h}_{\Hh}
		=0
		\qquad\text{for every }h\in\Hh,
		\]
		so the first argument vanishes in $\Hh$.  Evaluating it at $x$ gives \eqref{eq:representer-score}, and evaluating it at $x=z_j$ and stacking the $N$ vector equations gives \eqref{eq:representer-linear-system}.  The block Gram matrix is symmetric because $K(x,y)=K(y,x)^\top$, and $\xi^\top G_Z\xi=\norm{\sum_iK(\cdot,z_i)\xi_i}_{\Hh}^2\ge0$ for $\xi\in(\R^d)^N$.  Hence $G_Z+N\lambda I_{dN}$ is symmetric positive definite, and the linear system has a unique solution.
	\end{proof}
	
	\subsection{The sharp zero-regularization dichotomy at fixed particle number}
	
	The finite block system \eqref{eq:representer-linear-system} remains solvable as $\lambda\downarrow0$ at the sample values, but this can hide a divergent component of the score in the null space of point evaluation.  The following result identifies that component exactly.
	
	Equip $\mathcal Y_N:=(\R^d)^N$ with the inner product $\ip{u}{v}_N=N^{-1}\sum_i u_i\cdot v_i$ and norm $\norm{u}_N:=\ip{u}{u}_N^{1/2}$, and define
	\begin{equation}\label{eq:empirical-evaluation-operator}
		E_Z:\Hh\to\mathcal Y_N,
		\qquad (E_Zs)_i=s(z_i),
		\qquad
		C_Z:=E_Z^*E_Z,
		\qquad
		b_Z:=\frac1N\sum_{i=1}^Nd_{z_i}.
	\end{equation}
	Thus
	\begin{equation}\label{eq:empirical-operator-objective}
		\J_{\mu_Z}^\lambda(s)
		=\norm{E_Zs}_N^2+2\ip{b_Z}{s}_{\Hh}
		+\lambda\norm{s}_{\Hh}^2,
		\qquad
		S_\lambda(\mu_Z)=-(C_Z+\lambda I)^{-1}b_Z.
	\end{equation}
	
	\begin{theorem}[Fixed-$N$ zero-regularization dichotomy]
		\label{thm:empirical-lambda-zero-dichotomy}
		Let
		\[
		\mathcal R_Z:=\operatorname{Ran}E_Z^*=(\ker E_Z)^\perp
		\]
		and write $b_Z=b_R+b_\perp$ according to $\Hh=\mathcal R_Z\oplus\ker E_Z$.  Then
		\begin{equation}\label{eq:empirical-lambda-decomposition}
			S_\lambda(\mu_Z)
			=-\bigl(C_Z|_{\mathcal R_Z}+\lambda I\bigr)^{-1}b_R
			-\frac1\lambda b_\perp.
		\end{equation}
		Consequently:
		\begin{enumerate}[label=(\roman*)]
			\item The unregularized functional $\J_{\mu_Z}^0$ has a minimizer if and only if
			$b_\perp=0$, equivalently
			\begin{equation}\label{eq:empirical-lambda-compatibility}
				\sum_{i=1}^N\Div h(z_i)=0
				\quad\text{for every }h\in\Hh
				\text{ such that }h(z_i)=0\text{ for all }i.
			\end{equation}
			\item In the compatible case,
			\begin{equation}\label{eq:empirical-lambda-compatible-minimizers}
				\argmin_{s\in\Hh}\J_{\mu_Z}^0(s)
				=-C_Z^\dagger b_Z+\ker E_Z,
			\end{equation}
			and $S_\lambda(\mu_Z)\to-C_Z^\dagger b_Z$ in $\Hh$ as $\lambda\downarrow0$.  Thus Tikhonov regularization selects the minimum-norm minimizer, as in the classical linear theory \cite{EnglHankeNeubauer1996}.
			\item In the incompatible case, as $\lambda\downarrow0$,
			\begin{align}
				\inf_{s\in\Hh}\J_{\mu_Z}^0(s)&=-\infty,
				\qquad
				\lambda S_\lambda(\mu_Z)\longrightarrow-b_\perp,
				\label{eq:empirical-lambda-incompatible-limit}\\
				\min_{s\in\Hh}\J_{\mu_Z}^\lambda(s)
				&=-\frac{\norm{b_\perp}_{\Hh}^2}{\lambda}
				-\ip{b_R}{(C_Z|_{\mathcal R_Z}+\lambda I)^{-1}b_R}_{\Hh}.
				\label{eq:empirical-lambda-minimum}
			\end{align}
		\end{enumerate}
		If $\alpha_Z>0$ is the smallest positive eigenvalue of $C_Z$, then in the compatible case, with $s_Z^\dagger=-C_Z^\dagger b_Z$,
		\begin{equation}\label{eq:empirical-lambda-compatible-rate}
			\norm{S_\lambda(\mu_Z)-s_Z^\dagger}_{\Hh}
			\le\frac{\lambda}{\alpha_Z+\lambda}\norm{s_Z^\dagger}_{\Hh}.
		\end{equation}
	\end{theorem}
	
	\begin{proof}
		Since $E_Z$ has finite rank, $\mathcal R_Z$ is closed and finite-dimensional, and $\Hh=\mathcal R_Z\oplus\ker E_Z$ orthogonally.  The operator $C_Z=E_Z^*E_Z$ is self-adjoint, vanishes on $\ker E_Z$, and maps $\mathcal R_Z$ into itself.  It is positive definite on $\mathcal R_Z$: if $r\in\mathcal R_Z$ and $\ip{C_Zr}{r}_{\Hh}=\norm{E_Zr}_N^2=0$, then $r\in\ker E_Z\cap\mathcal R_Z=\{0\}$.  Projecting the Euler equation $(C_Z+\lambda I)S_\lambda(\mu_Z)=-b_Z$ onto $\mathcal R_Z$ and $\ker E_Z$ gives $(C_Z|_{\mathcal R_Z}+\lambda I)S_R=-b_R$ and $\lambda S_\perp=-b_\perp$, which is \eqref{eq:empirical-lambda-decomposition}.
		
		For $s=r+h$ with $r\in\mathcal R_Z$ and $h\in\ker E_Z$, one has $E_Zh=0$, so
		\begin{equation}\label{eq:dichotomy-splitting}
			\J_{\mu_Z}^0(r+h)
			=\ip{C_Zr}{r}_{\Hh}+2\ip{b_R}{r}_{\Hh}+2\ip{b_\perp}{h}_{\Hh}.
		\end{equation}
		If $b_\perp\ne0$, then $\J_{\mu_Z}^0(-tb_\perp)=-2t\norm{b_\perp}_{\Hh}^2\to-\infty$ as $t\to\infty$, so $\J_{\mu_Z}^0$ has no minimizer; this is the first claim in (iii).  If $b_\perp=0$, then \eqref{eq:dichotomy-splitting} is independent of $h$ and strictly convex in $r\in\mathcal R_Z$; completing the square gives the unique minimizing component $r=-(C_Z|_{\mathcal R_Z})^{-1}b_R=-C_Z^\dagger b_Z$, which proves \eqref{eq:empirical-lambda-compatible-minimizers} and the equivalence in (i).  Moreover, for $h\in\ker E_Z$, $\ip{b_Z}{h}_{\Hh}=N^{-1}\sum_i\Div h(z_i)$; since $b_\perp$ is the orthogonal projection of $b_Z$ onto $\ker E_Z$, the condition $b_\perp=0$ is exactly \eqref{eq:empirical-lambda-compatibility}.
		
		The remaining statements follow from the spectral decomposition of the positive definite operator $C_Z|_{\mathcal R_Z}$, whose eigenvalues satisfy $\tau_j\ge\alpha_Z$.  First, $\lambda(C_Z|_{\mathcal R_Z}+\lambda I)^{-1}b_R\to0$, so \eqref{eq:empirical-lambda-decomposition} gives $\lambda S_\lambda(\mu_Z)\to-b_\perp$.  Next, the minimum of $\ip{(C_Z+\lambda I)s}{s}_{\Hh}+2\ip{b_Z}{s}_{\Hh}$ is $-\ip{b_Z}{(C_Z+\lambda I)^{-1}b_Z}_{\Hh}$, which splits as \eqref{eq:empirical-lambda-minimum} along $\Hh=\mathcal R_Z\oplus\ker E_Z$. Finally, if $b_\perp=0$, then $S_\lambda(\mu_Z)-s_Z^\dagger=-\lambda(C_Z|_{\mathcal R_Z}+\lambda I)^{-1}s_Z^\dagger$, so each spectral component is multiplied by $\lambda/(\tau_j+\lambda)\le\lambda/(\alpha_Z+\lambda)$.  This proves \eqref{eq:empirical-lambda-compatible-rate} and the convergence in (ii).
	\end{proof}
	
	The hidden component is computable from the kernel and its first and second derivatives at the particles. Set $D_Z:=Nb_Z=\sum_i d_{z_i}$.  Then $E_ZD_Z=q_Z$, $E_Z^*u=N^{-1}\sum_iK(\cdot,z_i)u_i$, and $E_ZE_Z^*=G_Z/N$, so the orthogonal projection onto $\mathcal R_Z$ is $P_Z=E_Z^*(E_ZE_Z^*)^\dagger E_Z$ and 
	\begin{equation}\label{eq:hidden-nullspace-residual}
		D_Z^\perp
		:=(I-P_Z)D_Z
		=D_Z-\sum_{i=1}^NK(\cdot,z_i)(G_Z^\dagger q_Z)_i
		=Nb_\perp .
	\end{equation}
	Since $q_Z=E_ZD_Z\in\operatorname{Ran}E_Z=\operatorname{Ran}(E_ZE_Z^*) =\operatorname{Ran}G_Z$, one has $G_ZG_Z^\dagger q_Z=q_Z$, hence $\norm{P_ZD_Z}_{\Hh}^2=q_Z^\top G_Z^\dagger q_Z$ and, by Pythagoras,
	\begin{equation}\label{eq:hidden-nullspace-residual-norm}
		\norm{D_Z^\perp}_{\Hh}^2
		=\norm{D_Z}_{\Hh}^2-q_Z^\top G_Z^\dagger q_Z.
	\end{equation}
	By \eqref{eq:representer-linear-system}, the nodal values $u_\lambda=-(G_Z+N\lambda I_{dN})^{-1}q_Z$ converge to $-G_Z^\dagger q_Z$ as $\lambda\downarrow0$, while by \eqref{eq:empirical-lambda-decomposition} the full score contains the component $-b_\perp/\lambda=-D_Z^\perp/(N\lambda)$. Hence convergence of the nodal linear system does not certify an unregularized minimizer in $\Hh$.
	
	The divergence of the full score must be distinguished from the behavior of the particle force. Since $b_\perp\in\ker E_Z$, one has $b_\perp(z_i)=0$ for every particle, and hence $F(Z;b_\perp)=0$. Under \cref{ass:kernel}, the decomposition \eqref{eq:empirical-lambda-decomposition} therefore gives, for each fixed configuration $Z$,
	\[
	F(Z;S_\lambda(\mu_Z))
	\longrightarrow F(Z;-C_Z^\dagger b_Z)
	\qquad\text{as }\lambda\downarrow0.
	\]
	Thus \cref{thm:empirical-lambda-zero-dichotomy} identifies an obstruction to convergence of the full score in $\Hh$; it does not by itself rule out a zero-regularization limit of the coupled particle dynamics. In the other direction, the pointwise force limit above does not establish convergence of trajectories, which requires additional control as the configuration varies.
	
	For the alternating iteration, the score learned at a configuration $Z$ is evaluated at a new configuration $Y$. The component $b_\perp$ need not vanish at the points of $Y$ and may therefore contribute to $F(Y;S_\lambda(\mu_Z))$. Convergence of the nodal values alone does not provide $\lambda$-uniform stability of this iteration.
	
	\begin{corollary}[Blow-up for a single atom in a Gaussian RKHS]	\label{cor:Gaussian-lambda-blowup}
		Let $N=1$ and consider
		\[
		K(x,y)=\exp\!\left(-\frac{\abs{x-y}^2}{2\sigma^2}\right)I_d.
		\]
		For the single atom $\delta_z$,
		\begin{equation}\label{eq:Gaussian-lambda-blowup}
			d_z(x)=\frac{x-z}{\sigma^2}e^{-\abs{x-z}^2/(2\sigma^2)},
			\qquad
			S_\lambda(\delta_z)=-\frac{d_z}{\lambda},
		\end{equation}
		and
		\begin{equation}\label{eq:Gaussian-lambda-blowup-norm}
			\norm{S_\lambda(\delta_z)}_{\Hh}
			=\frac{\sqrt d}{\sigma\lambda},
			\qquad
			\min\J_{\delta_z}^\lambda=-\frac{d}{\sigma^2\lambda}.
		\end{equation}
		Thus the $O(\lambda^{-1})$ bound in \eqref{eq:score-H-bound} is sharp: the optimal constant in \eqref{eq:rkhs-bounds-2} is $c_D=\sup_x\norm{d_x}_{\Hh}=\sqrt d/\sigma$, and with it $\norm{S_\lambda(\delta_z)}_{\Hh}=R_\lambda$.
	\end{corollary}
	
	\begin{proof}
		Write $K=kI_d$ with $k(x,y)=\exp(-\abs{x-y}^2/(2\sigma^2))$.  By \eqref{eq:divergence-representer},
		\[
		d_z(x)=\sum_a\partial_a^{(2)}k(x,z)e_a=\sigma^{-2}(x-z)k(x,z),
		\]
		so $d_z(z)=0$.  Since distinct components are orthogonal and
		\[
		\partial_a^{(1)}\partial_b^{(2)}k(x,z)
		=\bigl(\sigma^{-2}\delta_{ab}-\sigma^{-4}(x_a-z_a)(x_b-z_b)\bigr)k(x,z),
		\]
		derivative reproduction gives
		\[
		\norm{d_z}_{\Hh}^2
		=\sum_{a=1}^d\partial_a^{(1)}\partial_a^{(2)}k(x,z)\big|_{x=z}
		=\frac d{\sigma^2}.
		\]
		For $N=1$ one has $b_Z=d_z$ and $E_Zd_z=d_z(z)=0$, so $b_R=0$ and $b_\perp=d_z$.  \Cref{thm:empirical-lambda-zero-dichotomy} now gives $S_\lambda(\delta_z)=-d_z/\lambda$ by \eqref{eq:empirical-lambda-decomposition} and $\min\J_{\delta_z}^\lambda=-\norm{d_z}_{\Hh}^2/\lambda$ by \eqref{eq:empirical-lambda-minimum}, which proves \eqref{eq:Gaussian-lambda-blowup}--\eqref{eq:Gaussian-lambda-blowup-norm}.
	\end{proof}
	
	\begin{remark}\label{rem:LS-sharpness}
		The order $\lambda^{-2}$ of $L_{S,\lambda}$ in \eqref{eq:score-W1} comes from the a priori bound $\norm{S_\nu}_{\Hh}\le R_\lambda=c_D/\lambda$, which enters the product estimate for $\Lips(S_\nu\cdot q)$ through both $\norm{S_\nu}_{L^\infty}\le c_0R_\lambda$ and $\Lips(S_\nu)\le c_1R_\lambda$ and is then multiplied by the factor $\lambda^{-1}$ coming from the coercivity of the regularized objective. On single atoms in the Gaussian RKHS the dependence is only $\lambda^{-1}$: by \cref{cor:Gaussian-lambda-blowup}, $S_\lambda(\delta_z)=-d_z/\lambda$, and $\norm{d_z-d_y}_{\Hh}\le c_2\abs{z-y}$ by \cref{ass:score-space}, so $\norm{S_\lambda(\delta_z)-S_\lambda(\delta_y)}_{L^\infty} \le c_0c_2\,W_1(\delta_z,\delta_y)/\lambda$.
	\end{remark}
	
	A finite-dimensional feature-space version of the score step, with a closed form and an a posteriori certificate for approximate score solves, is given in Appendix~\ref{sm:features}.
	
	\section{Estimates for the system step and conservation laws}
	\label{sec:system-step}
	
	Throughout this section, \cref{ass:kernel} is in force, and we use the mean-particle distance \eqref{eq:particle-norm} and the force \eqref{eq:intro-force}.
	
	\begin{lemma}\label{lem:force-estimates}
		Suppose $\norm s_\infty\le M$ and $\Lips(s)\le L$.  Then
		\begin{align}
			\abs{F(Z;s)}_N&\le 2a_0M,\label{eq:force-bound}\\
			\abs{F(Z;s)-F(Y;s)}_N
			&\le L_F\abs{Z-Y}_N,
			&L_F&:=2(2a_1M+a_0L).\label{eq:force-state-lip}
		\end{align}
		For any two bounded scores $s,r$,
		\begin{equation}\label{eq:force-score-lip}
			\abs{F(Z;s)-F(Z;r)}_N
			\le2a_0\norm{s-r}_{L^\infty}.
		\end{equation}
		Also,
		\begin{equation}\label{eq:empirical-W1-coupling}
			W_1(\mu_Z,\mu_Y)\le\abs{Z-Y}_N.
		\end{equation}
	\end{lemma}
	
	\begin{proof}
		Write $Z=(z_1,\ldots,z_N)$ and $Y=(y_1,\ldots,y_N)$.  Since $\norm A_{L^\infty}=a_0$ and $\abs{s(z_i)-s(z_j)}\le2M$, each component satisfies $\abs{F_i(Z;s)}\le2a_0M$; averaging over $i$ gives \eqref{eq:force-bound}.  For \eqref{eq:force-state-lip}, fix a pair $(i,j)$ and insert $A(y_i-y_j)\bigl(s(z_i)-s(z_j)\bigr)$.  Using $\Lips(A)=a_1$, $\norm s_\infty\le M$, $\norm A_{L^\infty}=a_0$, and $\Lips(s)\le L$,
		\[
		\begin{aligned}
			&\abs{A(z_i-z_j)\bigl(s(z_i)-s(z_j)\bigr)
				-A(y_i-y_j)\bigl(s(y_i)-s(y_j)\bigr)}\\
			&\qquad\le a_1\abs{(z_i-z_j)-(y_i-y_j)}\,2M
			+a_0\bigl(\abs{s(z_i)-s(y_i)}+\abs{s(z_j)-s(y_j)}\bigr)\\
			&\qquad\le(2a_1M+a_0L)\bigl(\abs{z_i-y_i}+\abs{z_j-y_j}\bigr).
		\end{aligned}
		\]
		Averaging over $i$ and $j$ gives
		\[
		\abs{F(Z;s)-F(Y;s)}_N
		\le\frac{2a_1M+a_0L}{N^2}\sum_{i,j=1}^N
		\bigl(\abs{z_i-y_i}+\abs{z_j-y_j}\bigr)
		=L_F\abs{Z-Y}_N .
		\]
		For \eqref{eq:force-score-lip}, each $i$ satisfies $\abs{F_i(Z;s)-F_i(Z;r)} \le N^{-1}\sum_ja_0\abs{(s-r)(z_i)-(s-r)(z_j)}$, and the right-hand side is at most $2a_0\norm{s-r}_{L^\infty}$.  Finally, the coupling $\pi:=N^{-1}\sum_{i=1}^N\delta_{(z_i,y_i)}$ has marginals $\mu_Z$ and $\mu_Y$, so $W_1(\mu_Z,\mu_Y)\le\int\abs{x-y}\dd\pi(x,y)=\abs{Z-Y}_N$.
	\end{proof}
	
	Every iterate of \cref{alg:alternating} is well defined on $[0,T]$, by induction on $k$ starting from the constant path $X^0$.  If $X^k\in C([0,T];(\R^d)^N)$, then $t\mapsto\mu_{X_t^k}$ is $W_1$-continuous by \eqref{eq:empirical-W1-coupling}, so $t\mapsto s_t^k$ is continuous in $L^\infty$ by \eqref{eq:score-W1}, and $\norm{s_t^k}_\infty\le M_\lambda$, $\Lips(s_t^k)\le L_\lambda$ for all $t$ by \eqref{eq:main-score-bound}.  By \eqref{eq:force-state-lip} and \eqref{eq:force-score-lip}, the field $(t,Z)\mapsto F(Z;s_t^k)$ is then continuous and globally Lipschitz in $Z$ with the $t$-independent constant $L_{F,\lambda}$.  Since $\abs{\cdot}_N$ is equivalent to the Euclidean norm for fixed $N$, the Cauchy--Lipschitz theorem gives a unique solution $X^{k+1}\in C^1([0,T];(\R^d)^N)$ of the system step.
	
	The following pairwise calculation is standard for Landau-type particle systems \cite{CarrilloHuWangWu2020,HuangWang2025}; we record it because it holds even for a frozen off-iterate score path.
	
	\begin{proposition}\label{prop:conservation}
		Assume \eqref{eq:A-structure}.  Let $s_t$ be any continuous bounded score path, not necessarily the score of the current particles.  Every solution of $\dot Z_t=F(Z_t;s_t)$ satisfies
		\begin{equation}\label{eq:momentum-energy}
			\frac1N\sum_{i=1}^Nz_i(t)=\frac1N\sum_{i=1}^Nz_i(0),
			\qquad
			\frac1N\sum_{i=1}^N\abs{z_i(t)}^2
			=\frac1N\sum_{i=1}^N\abs{z_i(0)}^2.
		\end{equation}
		Consequently every output of \cref{alg:alternating} has the initial momentum and energy.  No entropy monotonicity is asserted for a frozen, off-iterate score.
	\end{proposition}
	
	\begin{proof}
		Pairing the $(i,j)$ and $(j,i)$ terms and using $A(-z)=A(z)$ gives $\sum_iF_i(Z;s)=0$.  Similarly, symmetrizing in $(i,j)$,
		\[
		\frac{\dd}{\dd t}\frac1N\sum_i\abs{z_i}^2
		=-\frac1{N^2}\sum_{i,j}(z_i-z_j)^\top
		A(z_i-z_j)(s(z_i)-s(z_j))=0,
		\]
		where symmetry and $A(z)z=0$ were used in the last equality.
	\end{proof}
	
	\section{Convergence of the exact alternating iteration}
	\label{sec:exact-outer}
	
	Throughout this section, $N$, $\lambda>0$, and $A$ are fixed, and $V_\lambda$, $L_{F,\lambda}$, and $\beta_\lambda$ are the constants of \cref{sec:framework}. The map $X^k\mapsto X^{k+1}$ need not be a contraction on long time intervals, but its causal structure gives a Volterra recursion whose repeated integration yields factorial decay. The resulting estimates \eqref{eq:main-factorial-increment}--\eqref{eq:main-factorial-tail} have the form of the classical waveform-relaxation bound \cite{MiekkalaNevanlinna1987,Botchev2024}, with Lipschitz constant $L_{F,\lambda}$ for the current argument and $\beta_\lambda$ for the lagged one; what is specific to the present setting is that $\beta_\lambda=2a_0L_{S,\lambda}$ does not depend on $N$, by \cref{thm:main-score} and the labeled coupling \eqref{eq:empirical-W1-coupling}.
	
	\begin{proof}[Proof of \cref{thm:main-outer}]
		\emph{Increments.} By the discussion following \cref{lem:force-estimates}, every iterate is uniquely defined on $[0,T]$, and every exact score satisfies $\norm{s_t^k}_\infty\le M_\lambda$ and $\Lips(s_t^k)\le L_\lambda$ by \eqref{eq:main-score-bound}. Since $X^0\equiv Z^0$, \eqref{eq:force-bound} gives
		\begin{equation}\label{eq:D0}
			D_0(t)=\Bigl|\int_0^tF(X_r^1;s_r^0)\dd r\Bigr|_N\le2a_0M_\lambda t=V_\lambda t.
		\end{equation}
		For $k\ge1$, subtracting the integral forms of the $(k+1)$st and $k$th system steps and splitting the force difference into a state part and a score part gives
		\[
		D_k(t)\le\int_0^t\Bigl(\abs{F(X_r^{k+1};s_r^k)-F(X_r^k;s_r^k)}_N
		+\abs{F(X_r^k;s_r^k)-F(X_r^k;s_r^{k-1})}_N\Bigr)\dd r.
		\]
		By \eqref{eq:force-state-lip}, the first term is at most $L_{F,\lambda}D_k(r)$. By \eqref{eq:force-score-lip}, \eqref{eq:score-W1}, and \eqref{eq:empirical-W1-coupling}, the second is at most $2a_0\norm{s_r^k-s_r^{k-1}}_{L^\infty} \le2a_0L_{S,\lambda}W_1(\mu_{X_r^k},\mu_{X_r^{k-1}})\le\beta_\lambda D_{k-1}(r)$. Hence $D_k(t)\le L_{F,\lambda}\int_0^tD_k(r)\dd r+a_k(t)$ with $a_k(t):=\beta_\lambda\int_0^tD_{k-1}(r)\dd r$. Since $a_k$ is absolutely continuous with $a_k(0)=0$, the integral form of Gronwall's inequality followed by an integration by parts gives
		\begin{equation}\label{eq:volterra-recursion}
			D_k(t)
			\le a_k(t)+L_{F,\lambda}\int_0^te^{L_{F,\lambda}(t-r)}a_k(r)\dd r
			=\beta_\lambda\int_0^te^{L_{F,\lambda}(t-r)}D_{k-1}(r)\dd r.
		\end{equation}
		
		\emph{Factorial bound.} When \eqref{eq:volterra-recursion} is iterated, the exponential factors telescope to $e^{L_{F,\lambda}(t-r_1)}\le e^{L_{F,\lambda}t}$, and \eqref{eq:D0} gives, for $k\ge1$,
		\[
		D_k(t)\le V_\lambda e^{L_{F,\lambda}t}\beta_\lambda^k
		\int_{0\le r_1\le\cdots\le r_k\le t}r_1\dd r_1\cdots\dd r_k
		=V_\lambda e^{L_{F,\lambda}t}\frac{\beta_\lambda^kt^{k+1}}{(k+1)!};
		\]
		the case $k=0$ is \eqref{eq:D0}. Since the right-hand side is nondecreasing in $t$, taking the supremum over $r\in[0,t]$ gives \eqref{eq:main-factorial-increment}.
		
		\emph{Limit and error estimate.} Since the bounds \eqref{eq:main-factorial-increment} are summable, the iterates form a Cauchy sequence in $C([0,T];(\R^d)^N)$; let $X^{\lambda,N}$ be their uniform limit. In the integral form of the system step, the integrand converges uniformly, since by \eqref{eq:force-state-lip}, \eqref{eq:force-score-lip}, \eqref{eq:score-W1}, and \eqref{eq:empirical-W1-coupling},
		\begin{align*}
			&\bigl|F(X_r^{k+1};S_\lambda(\mu_{X_r^k}))
			-F(X_r^{\lambda,N};
			S_\lambda(\mu_{X_r^{\lambda,N}}))\bigr|_N\\
			&\qquad\le
			L_{F,\lambda}\abs{X_r^{k+1}-X_r^{\lambda,N}}_N
			+\beta_\lambda\abs{X_r^k-X_r^{\lambda,N}}_N
			\longrightarrow0.
		\end{align*}
		Passing to the limit in the integral equation shows that $X^{\lambda,N}$ solves \eqref{eq:main-coupled-particle}. Its right-hand side is continuous in $t$, so $X^{\lambda,N}\in C^1([0,T];(\R^d)^N)$.
		
		To estimate the error directly, set $e_k(t):=\abs{X_t^k-X_t^{\lambda,N}}_N$. Since $X^0\equiv Z^0$ and the limiting path has speed at most $V_\lambda$, we have $e_0(t)\le V_\lambda t$. Comparing the $(k+1)$st system step with \eqref{eq:main-coupled-particle} gives
		\[
		e_{k+1}(t)
		\le\int_0^t
		\bigl(L_{F,\lambda}e_{k+1}(r)
		+\beta_\lambda e_k(r)\bigr)\dd r.
		\]
		Gronwall's inequality therefore yields
		\[
		e_{k+1}(t)
		\le\beta_\lambda\int_0^t
		e^{L_{F,\lambda}(t-r)}e_k(r)\dd r.
		\]
		Starting from $e_0(t)\le V_\lambda t$, induction gives
		\[
		e_k(t)\le
		V_\lambda e^{L_{F,\lambda}t}
		\frac{\beta_\lambda^kt^{k+1}}{(k+1)!},
		\qquad k\ge0.
		\]
		Taking the supremum over $t\in[0,T]$ proves \eqref{eq:main-factorial-tail}.
		
		\emph{Uniqueness and invariants.} If $X$ and $Y$ solve \eqref{eq:main-coupled-particle} with the same initial condition, the same splitting gives $\abs{X_t-Y_t}_N\le(L_{F,\lambda}+\beta_\lambda)\int_0^t\abs{X_r-Y_r}_N\dd r$, and Gronwall's inequality gives $X=Y$. If \eqref{eq:A-structure} holds, \eqref{eq:main-particle-invariants} follows from \cref{prop:conservation}, applied to each system step with the frozen score path $s^k$ and to $X^{\lambda,N}$ with the score path $S_\lambda(\mu_{X_t^{\lambda,N}})$.
	\end{proof}
	
	The proof also gives a stopping criterion in terms of the last successive-iterate difference.
	
	\begin{corollary}[A posteriori error estimate for the alternating iteration]\label{cor:outer-stopping}
		Under the assumptions of \cref{thm:main-outer}, let $X^k$ be the exact iterates and put
		\[
		\Delta_k:=\sup_{0\le t\le T}\abs{X_t^{k+1}-X_t^k}_N.
		\]
		Then the remaining error of the exact iteration satisfies
		\begin{equation}\label{eq:outer-stopping}
			\abs{X_t^{k+1}-X_t^{\lambda,N}}_N
			\le\beta_\lambda\int_0^t
			e^{(L_{F,\lambda}+\beta_\lambda)(t-r)}D_k(r)\dd r,
			\qquad 0\le t\le T.
		\end{equation}
		Consequently,
		\begin{equation}\label{eq:outer-stopping-sup}
			\sup_{0\le t\le T}\abs{X_t^{k+1}-X_t^{\lambda,N}}_N
			\le\beta_\lambda\Delta_k
			\frac{e^{(L_{F,\lambda}+\beta_\lambda)T}-1}
			{L_{F,\lambda}+\beta_\lambda},
		\end{equation}
		with the quotient interpreted as $T$ if its denominator vanishes (the case $A\equiv0$, which \cref{ass:kernel} allows).  Thus a certified continuous-time bound on the successive-iterate difference, together with the analytic stability constant, is a rigorous stopping criterion; values at finitely many time nodes alone are not such a certificate.
	\end{corollary}
	
	\begin{proof}
		Let $E(t):=\abs{X_t^{k+1}-X_t^{\lambda,N}}_N$. Comparing the $(k+1)$st system step with \eqref{eq:main-coupled-particle} by the same state--score splitting, and using $\abs{X_r^k-X_r^{\lambda,N}}_N\le D_k(r)+E(r)$, we obtain $E(t)\le\int_0^t[(L_{F,\lambda}+\beta_\lambda)E(r)+\beta_\lambda D_k(r)]\dd r$. Gronwall's inequality gives \eqref{eq:outer-stopping}, and $D_k\le\Delta_k$ gives \eqref{eq:outer-stopping-sup}.
	\end{proof}
	
	\begin{remark}\label{rem:no-small-time}
		The factorial denominator in \eqref{eq:main-factorial-increment} comes from the causal estimate \eqref{eq:volterra-recursion}, not from a contraction, so the alternating iteration converges for every finite $T$. This does not make a discrete optimizer or an explicit time integrator unconditionally stable; such errors are treated in \cref{sec:inexact}. The factorial mechanism is classical for waveform relaxation \cite{MiekkalaNevanlinna1987,Botchev2024}; the new input here is the $N$-uniform stability of the data-dependent score map and its connection to a computable score-objective gap.
	\end{remark}
	
	\begin{remark}\label{rem:lambda-scaling}
		The constants in \cref{thm:main-outer} degenerate as $\lambda\downarrow0$: $R_\lambda$, $M_\lambda$, $L_\lambda$, and $L_{F,\lambda}$ are $O(\lambda^{-1})$, while $L_{S,\lambda}$ and $\beta_\lambda$ are $O(\lambda^{-2})$. For $k\ge1$, the elementary bound $k!\ge(k/e)^k$ gives
		\[
		\frac{(\beta_\lambda T)^k}{(k+1)!}
		\le
		\left(\frac{e\beta_\lambda T}{k}\right)^k.
		\]
		Consequently, \eqref{eq:main-factorial-tail} implies
		\[
		\sup_{0\le t\le T}\abs{X_t^k-X_t^{\lambda,N}}_N
		\le V_\lambda T e^{L_{F,\lambda}T}\,2^{-k}
		\qquad\text{whenever }
		k\ge\max\{1,2e\beta_\lambda T\}.
		\]
		Thus the estimate certifies geometric decay beyond a sufficient threshold of order $1+\beta_\lambda T$. This is not a lower bound on the number of iterations actually needed. The prefactor $e^{L_{F,\lambda}T}$ is of size $e^{O(T/\lambda)}$. These are worst-case bounds, and for moderate $\lambda$ they are far from what is observed. For the feature space and kernel of the experiments in Appendix~\ref{sm:numerics} ($d=2$, $\lambda=0.08$, $T=1.5$) the constants of \cref{ass:score-space,ass:kernel} can be computed exactly (Table~\ref{tab:analytic-constants}): $c_0=a_0=1$, $c_1\approx0.905$, $c_D\approx1.07$, $c_2\approx1.61$, $a_1\approx0.650$, whence $\beta_\lambda\approx6.5\times10^2$ and a certified threshold $2e\beta_\lambda T\approx5.3\times10^3$; the bound \eqref{eq:main-factorial-increment} falls below $10^{-10}$ only for $k\ge2741$, whereas the experiments reach successive differences of $2.1\times10^{-10}$ at $k=15$. Thus the theorem guarantees convergence, the experiments exhibit its speed, and closing the gap between the two, that is, proving convergence with constants uniform or mildly dependent on $\lambda$, is an open problem. The estimates are therefore informative for fixed $\lambda$ and do not by themselves control the regime $\lambda\downarrow0$, which is discussed separately in \cref{sec:stable-score-learner} and \cref{sec:landau-connection}.
	\end{remark}
	
	\section{Mean-field limit at fixed regularization}
	\label{sec:mean-field}
	
	We keep $\Hh$, $A$, $\lambda>0$, and $T$ fixed and study the particle dynamics as $N\to\infty$. The equation \eqref{eq:main-mean-field} is meaningful for all measures in $\Pone$, including atomic ones, and continuity of measure-valued curves is understood in $W_1$. Set
	\begin{equation}\label{eq:mean-field-constants}
		L_b:=2a_1M_\lambda+a_0L_\lambda,
		\qquad\text{so that}\qquad C_{\mathrm{mf}}=2L_b+2a_0L_{S,\lambda}.
	\end{equation}
	
	\begin{lemma}\label{lem:mean-field-drift}
		For all $\mu,\nu\in\Pone(\R^d)$,
		\begin{align}
			\norm{b_\lambda[\mu]}_\infty&\le2a_0M_\lambda,
			&\Lips_x(b_\lambda[\mu])&\le L_b,
			\label{eq:drift-state-bounds}\\
			\norm{b_\lambda[\mu]-b_\lambda[\nu]}_\infty
			&\le K_b W_1(\mu,\nu),
			&K_b&:=L_b+2a_0L_{S,\lambda}.
			\label{eq:drift-measure-bounds}
		\end{align}
		Moreover, $(\mu,x)\mapsto b_\lambda[\mu](x)$ is jointly continuous on $\Pone(\R^d)\times\R^d$, with $\Pone$ carrying $W_1$.
	\end{lemma}
	\begin{proof}
		For $\norm{s}_\infty\le M_\lambda$ and $\Lips(s)\le L_\lambda$, the pair interaction $\Gamma_s(x,y):=A(x-y)(s(x)-s(y))$ satisfies $\abs{\Gamma_s}\le2a_0M_\lambda$ and $\abs{\Gamma_s(x,y)-\Gamma_s(x',y')}\le L_b(\abs{x-x'}+\abs{y-y'})$; integrating with $s=S_\lambda(\mu)$ proves \eqref{eq:drift-state-bounds}. For $\pi\in\Gamma(\mu,\nu)$, $b_\lambda[\mu](x)-b_\lambda[\nu](x)=-\int(\Gamma_{S_\lambda(\mu)}(x,y) -\Gamma_{S_\lambda(\nu)}(x,z))\dd\pi(y,z)$; inserting $\Gamma_{S_\lambda(\nu)}(x,y)$ gives the bound $2a_0\norm{S_\lambda(\mu)-S_\lambda(\nu)}_\infty+L_b\int\abs{y-z}\dd\pi$, and taking the infimum over $\pi$ and using \eqref{eq:score-W1} proves \eqref{eq:drift-measure-bounds}. Joint continuity follows from $\abs{b_\lambda[\mu](x)-b_\lambda[\nu](x')}\le L_b\abs{x-x'}+K_bW_1(\mu,\nu)$.
	\end{proof}
	
	We prove \cref{thm:main-mean-field} by a deterministic Dobrushin-type argument \cite{Dobrushin1979}, included to fix the constants and the solution class.
	
	\begin{proof}[Proof of \cref{thm:main-mean-field}]
		Fix $\tau>0$ and equip $C([0,\tau];\Pone(\R^d))$ with the complete metric $\sup_{t\le\tau}W_1(m_t,n_t)$. For a curve $m$, let $\Phi^m$ be the flow of the globally Lipschitz field $b_\lambda[m_t]$ and set $\Psi(m)_t=(\Phi_t^m)_\#\mu_0$. Coupling two such flows by their common initial point and using \cref{lem:mean-field-drift} gives
		\begin{equation}\label{eq:mean-field-fixed-point-map}
			W_1(\Psi(m)_t,\Psi(n)_t)
			\le K_b\int_0^t e^{L_b(t-r)}W_1(m_r,n_r)\dd r,
		\end{equation}
		so $\Psi$ is a contraction when $K_b\tau e^{L_b\tau}<1$. It maps into $C([0,\tau];\Pone)$ because the first moment of $\Psi(m)_t$ grows by at most $2a_0M_\lambda t$ and $W_1(\Psi(m)_t,\Psi(m)_s)\le2a_0M_\lambda\abs{t-s}$. Since $\tau$ depends only on the fixed constants, Banach's fixed-point theorem, restarted at the endpoints, gives a characteristic solution on $[0,T]$. For stability, let $\Phi_t$ and $\widetilde\Phi_t$ be the characteristic flows of $\mu_t$ and $\nu_t$, choose an optimal coupling $\pi_0\in\Gamma(\mu_0,\nu_0)$, and set $E(t):=\int\abs{\Phi_t(x)-\widetilde\Phi_t(y)}\dd\pi_0(x,y)$, which is finite and satisfies $W_1(\mu_t,\nu_t)\le E(t)$ and $E(0)=W_1(\mu_0,\nu_0)$. By \cref{lem:mean-field-drift}, $E(t)\le E(0)+\int_0^t(L_bE+K_bW_1(\mu_r,\nu_r))\dd r \le E(0)+C_{\mathrm{mf}}\int_0^tE(r)\dd r$, and Gronwall's inequality proves \eqref{eq:main-Dobrushin} and uniqueness. For empirical data, $Z_t=X_t^{\lambda,N}$ and $m_t=\mu_{Z_t}$ satisfy $\dot z_i=F_i(Z_t;S_\lambda(m_t))=b_\lambda[m_t](z_i)$, so uniqueness for the flow of $b_\lambda[m_t]$ gives $m_t=(\Phi_t^m)_\#\mu_0^N$. Thus $m$ is a characteristic solution, and \eqref{eq:main-Dobrushin} proves the remaining claims.
	\end{proof}
	
	\subsection{Continuum alternating iteration and commuting limits}
	We now take $N\to\infty$ at fixed $k$ first and then $k\to\infty$; if $W_1(\mu_0^N,\mu_0)\to0$, both orders lead to the same solution of \eqref{eq:main-mean-field}. For a prescribed bounded Lipschitz score $s$, define
	\begin{equation}\label{eq:prescribed-score-drift}
		B[s,\rho](x):=-\int A(x-y)\bigl(s(x)-s(y)\bigr)\dd\rho(y).
	\end{equation}
	Starting from $\rho_t^0\equiv\mu_0$, define the continuum alternating iteration by
	\begin{equation}\label{eq:continuum-outer-iteration}
		s_t^k=S_\lambda(\rho_t^k),
		\qquad
		\partial_t\rho_t^{k+1}
		+\nabla\cdot\bigl(B[s_t^k,\rho_t^{k+1}]\rho_t^{k+1}\bigr)=0,
		\qquad \rho_0^{k+1}=\mu_0.
	\end{equation}
	Solutions in this subsection are again understood in the characteristic sense.
	
	\begin{theorem}
		\label{thm:continuum-outer-iteration}
		Under \cref{ass:score-space,ass:kernel}, and for $\mu_0\in\Pone(\R^d)$, every iterate of \eqref{eq:continuum-outer-iteration} is uniquely defined in $C([0,T];\Pone(\R^d))$.  If $\delta_k(t):=W_1(\rho_t^{k+1},\rho_t^k)$, then
		\begin{equation}\label{eq:continuum-outer-iteration-factorial}
			\sup_{0\le r\le t}\delta_k(r)
			\le V_\lambda e^{L_{F,\lambda}t}
			\frac{\beta_\lambda^kt^{k+1}}{(k+1)!}.
		\end{equation}
		Consequently $\rho^k$ converges uniformly in time in $W_1$ to the unique solution $\mu$ of \eqref{eq:main-mean-field}.
		
		Let $\rho_t^{k,N}:=\mu_{X_t^{k,N}}$ be the empirical measure of the $k$th iterate generated by \cref{alg:alternating} from an empirical initial law $\mu_0^N$, and set $d_N:=W_1(\mu_0^N,\mu_0)$.  For each $k\ge0$,
		\begin{equation}\label{eq:fixed-k-particle-waveform}
			\sup_{0\le t\le T}W_1(\rho_t^{k,N},\rho_t^k)
			\le d_Ne^{L_{F,\lambda}T}\sum_{j=0}^{k}\frac{(\beta_\lambda T)^j}{j!}.
		\end{equation}
		In particular, if $d_N\to0$, both iterated limits
		\begin{equation}\label{eq:commuting-limits}
			\lim_{N\to\infty}\lim_{k\to\infty}\rho^{k,N}
			=\mu
			=\lim_{k\to\infty}\lim_{N\to\infty}\rho^{k,N}
		\end{equation}
		hold in $C([0,T];(\Pone,W_1))$, at fixed $\lambda>0$ and fixed $A$.
	\end{theorem}
	
	\begin{proof}
		If $\norm s_\infty\le M_\lambda$ and $\Lips(s)\le L_\lambda$, then $B[s,\rho]$ is $L_b$-Lipschitz in space and
		\[
		\norm{B[s,\rho]-B[s,\sigma]}_\infty\le L_bW_1(\rho,\sigma),
		\]
		so the fixed-point argument of \cref{thm:main-mean-field} gives a unique prescribed-score solution. Let $\rho$ and $\sigma$ use score paths $s$ and $\bar s$, continuous in the supremum norm and with $\norm{s_t}_\infty,\norm{\bar s_t}_\infty\le M_\lambda$ and $\Lips(s_t),\Lips(\bar s_t)\le L_\lambda$. Let $\Phi_t$ and $\bar\Phi_t$ be their characteristic flows, choose an optimal coupling $\pi_0\in\Gamma(\rho_0,\sigma_0)$, and set
		\[
		E(t):=\int\abs{\Phi_t(x)-\bar\Phi_t(y)}\dd\pi_0(x,y)\ge W_1(\rho_t,\sigma_t).
		\]
		Since $\abs{B[s,\rho]-B[\bar s,\rho]}\le2a_0\norm{s-\bar s}_\infty$, the preceding bounds give
		\[
		E(t)\le W_1(\rho_0,\sigma_0)+\int_0^t\bigl(2L_bE(r)+2a_0\norm{s_r-\bar s_r}_\infty\bigr)\dd r,
		\]
		and Gronwall's inequality yields
		\begin{equation}\label{eq:prescribed-score-stability}
			W_1(\rho_t,\sigma_t)
			\le e^{L_{F,\lambda}t}W_1(\rho_0,\sigma_0)
			+2a_0\int_0^te^{L_{F,\lambda}(t-r)}
			\norm{s_r-\bar s_r}_\infty\dd r,
		\end{equation}
		where $L_{F,\lambda}=2L_b$: one copy of $L_b$ controls the spatial argument and one the evolving law. The score path of every iteration is continuous by \eqref{eq:score-H-W1}, so the construction applies.
		
		The first iterate moves at speed at most $V_\lambda$, so $\delta_0(t)\le V_\lambda t$. For $k\ge1$, \eqref{eq:prescribed-score-stability} with equal initial laws and \eqref{eq:score-W1} give $\delta_k(t)\le\beta_\lambda\int_0^te^{L_{F,\lambda}(t-r)}\delta_{k-1}(r)\dd r$, and the induction of \cref{thm:main-outer} proves \eqref{eq:continuum-outer-iteration-factorial}. The solution $\mu$ of \cref{thm:main-mean-field} is a prescribed-score solution with score path $S_\lambda(\mu_t)$, so $E_k(t):=W_1(\rho_t^k,\mu_t)$ satisfies $E_0(t)\le V_\lambda t$ and the same recursion; hence $E_k(t)\le V_\lambda e^{L_{F,\lambda}t}\beta_\lambda^kt^{k+1}/(k+1)!$ and $\rho^k\to\mu$. For the particle statement, note that $\dot z_i=F_i(X_t^{k+1,N};S_\lambda(\rho_t^{k,N})) =B[S_\lambda(\rho_t^{k,N}),\rho_t^{k+1,N}](z_i)$, so, as in the proof of \cref{thm:main-mean-field}, $\rho^{k+1,N}$ is the prescribed-score solution with score path $S_\lambda(\rho_t^{k,N})$ and initial law $\mu_0^N$.  Hence $e_k(t):=W_1(\rho_t^{k,N},\rho_t^k)$ satisfies $e_0=d_N$ and, since the initial law is reset at each iteration,
		\begin{equation}\label{eq:fixed-k-particle-recursion}
			e_{k+1}(t)\le e^{L_{F,\lambda}t}d_N+\beta_\lambda\int_0^te^{L_{F,\lambda}(t-r)}e_k(r)\dd r.
		\end{equation}
		Thus $f_k:=e^{-L_{F,\lambda}t}e_k$ satisfies $f_0\le d_N$ and $f_{k+1}(t)\le d_N+\beta_\lambda\int_0^tf_k$, and induction gives $f_k(t)\le d_N\sum_{j=0}^k(\beta_\lambda t)^j/j!$, which is \eqref{eq:fixed-k-particle-waveform}. Hence the $N\to\infty$ limit at fixed $k$ is $\rho^k$, whose $k\to\infty$ limit is $\mu$; the reverse order follows from \cref{thm:main-outer,thm:main-mean-field}.
	\end{proof}
	
	\begin{corollary}\label{cor:mean-field-conservation}
		Assume \eqref{eq:A-structure} and $\mu_0\in\mathcal P_2(\R^d)$.  The solution of \eqref{eq:main-mean-field} satisfies
		\begin{equation}\label{eq:mean-field-conservation}
			\int x\dd\mu_t(x)=\int x\dd\mu_0(x),
			\qquad
			\int\abs{x}^2\dd\mu_t(x)=\int\abs{x}^2\dd\mu_0(x).
		\end{equation}
	\end{corollary}
	\begin{proof}
		Write $\mu_t=(\Phi_t)_\#\mu_0$. Since $\abs{\partial_t\Phi_t(x)}\le V_\lambda$ and $\abs{\Phi_t(x)}\le\abs x+V_\lambda T$, the second moment stays finite and the first and second moments may be differentiated under the $\mu_0$ integral. Symmetrizing the resulting double integrals, which converge absolutely because $A$ and the score are bounded and $\mu_t$ has a finite first moment, gives zero for the momentum because $A(-z)=A(z)$, and zero for the second moment, $-\iint(x-y)^\top A(x-y)(S_\lambda(\mu_t)(x)-S_\lambda(\mu_t)(y))\dd\mu_t\dd\mu_t=0$, by symmetry and $A(z)z=0$.
	\end{proof}
	
	\begin{proof}[Proof of \cref{cor:main-joint-limit}]
		The estimate \eqref{eq:main-joint-error} follows from the triangle inequality, 	the labeled coupling \eqref{eq:empirical-W1-coupling}, \eqref{eq:main-factorial-tail}, and \eqref{eq:main-particle-limit}. The	commutation of the two iterated limits is \eqref{eq:commuting-limits} in \cref{thm:continuum-outer-iteration}.
	\end{proof}
	
	\section{Inexact score and system solves}
	\label{sec:inexact}
	
	We call the inexact iterates \emph{admissible} if $\widetilde X^0\in C([0,T];(\R^d)^N)$ with $\widetilde X_0^0=Z^0$, if $\widetilde X^{k+1}\in AC([0,T];(\R^d)^N)$ with $\widetilde X_0^{k+1}=Z^0$, and if $t\mapsto\widetilde s_t^k$ is strongly measurable as an $\Hh$-valued map. At iteration $k$, suppose a measurable function $\eps_k:[0,T]\to[0,\infty)$ gives the \emph{global} excess-objective bound, for almost every $t$,
	\begin{equation}\label{eq:inexact-score-gap}
		0\le
		\J_{\mu_{\widetilde X_t^k}}^\lambda(\widetilde s_t^k)
		-\min_{s\in\Hh}\J_{\mu_{\widetilde X_t^k}}^\lambda(s)
		\le\eps_k(t).
	\end{equation}
	Let the system residual be
	\begin{equation}\label{eq:system-residual}
		\mathcal R_t^{k+1}
		:=\dot{\widetilde X}_t^{k+1}
		-F(\widetilde X_t^{k+1};\widetilde s_t^k),
		\qquad
		r_{k+1}(t):=\abs{\mathcal R_t^{k+1}}_N,
	\end{equation}
	and suppose $r_{k+1}\in L^1(0,T)$. By \eqref{eq:exact-gap}, the triangle inequality, and \eqref{eq:score-H-bound}, for almost every $t$,
	\begin{equation}\label{eq:score-gap-to-H}
		\norm{\widetilde s_t^k-S_\lambda(\mu_{\widetilde X_t^k})}_{\Hh}
		\le\sqrt{\frac{\eps_k(t)}{\lambda}},
		\qquad
		\norm{\widetilde s_t^k}_{\Hh}
		\le R_\lambda+\sqrt{\frac{\eps_k(t)}{\lambda}}.
	\end{equation}
	
	Define the forcing quantity
	\begin{equation}\label{eq:eta-definition}
		\eta_k(t)
		:=2a_0c_0\sqrt{\frac{\eps_k(t)}{\lambda}}+r_{k+1}(t),
	\end{equation}
	and assume $\eta_k\in L^1(0,T)$ for every $k\ge0$. Recall that $\beta_\lambda=2a_0L_{S,\lambda}$.
	
	\begin{theorem}[A posteriori error estimate for the inexact alternating iteration]
		\label{thm:inexact}
		Under \cref{ass:score-space,ass:kernel} and the preceding admissibility, score-gap, and residual assumptions, let $X^{\lambda,N}$ be the exact fixed point from \cref{thm:main-outer}. Put
		\[
		E_k(t):=\abs{\widetilde X_t^k-X_t^{\lambda,N}}_N,
		\qquad
		E_k^*:=\sup_{0\le t\le T}E_k(t).
		\]
		Then
		\begin{equation}\label{eq:inexact-one-step}
			E_{k+1}(t)
			\le
			\int_0^t e^{L_{F,\lambda}(t-r)}
			\left(\beta_\lambda E_k(r)+\eta_k(r)\right)\dd r.
		\end{equation}
		For every $m\ge1$,
		\begin{equation}\label{eq:inexact-unrolled}
			E_m^*
			\le e^{L_{F,\lambda}T}
			\left[
			E_0^*\frac{(\beta_\lambda T)^m}{m!}
			+
			\sum_{j=0}^{m-1}
			\frac{(\beta_\lambda T)^j}{j!}
			\norm{\eta_{m-1-j}}_{L^1(0,T)}
			\right].
		\end{equation}
		Consequently:
		\begin{enumerate}[label=(\roman*)]
			\item If $\norm{\eta_k}_{L^1(0,T)}\to0$, then
			$\widetilde X^k\to X^{\lambda,N}$ uniformly on $[0,T]$.
			\item If
			$\sup_k\norm{\eta_k}_{L^1(0,T)}\le\overline\eta$, then
			\begin{equation}\label{eq:inexact-floor}
				\limsup_{m\to\infty}E_m^*
				\le
				e^{(L_{F,\lambda}+\beta_\lambda)T}\overline\eta.
			\end{equation}
		\end{enumerate}
		No summability of the per-iteration errors is required in part~(i). We use the convention $0^0=1$ in the sums above.
	\end{theorem}
	
	\begin{proof}
		Write $X_t:=X_t^{\lambda,N}$, $s_t^\star:=S_\lambda(\mu_{X_t})$, $L:=L_{F,\lambda}$, and $\beta:=\beta_\lambda$. By \cref{thm:main-outer} and \eqref{eq:system-residual},
		\[
		\dot X_t=F(X_t;s_t^\star),
		\qquad
		\dot{\widetilde X}_t^{k+1}
		=F(\widetilde X_t^{k+1};\widetilde s_t^k)
		+\mathcal R_t^{k+1}
		\]
		for almost every $t$, with the same initial configuration.
		
		First, insert the exact minimizer for the previous approximate trajectory:
		\[
		\widetilde s_t^k-s_t^\star
		=
		\left(\widetilde s_t^k
		-S_\lambda(\mu_{\widetilde X_t^k})\right)
		+
		\left(S_\lambda(\mu_{\widetilde X_t^k})
		-S_\lambda(\mu_{X_t})\right).
		\]
		By \eqref{eq:rkhs-bounds-1}, \eqref{eq:score-gap-to-H}, \eqref{eq:score-W1}, and \eqref{eq:empirical-W1-coupling},
		\[
		\norm{\widetilde s_t^k-s_t^\star}_{L^\infty}
		\le
		c_0\sqrt{\frac{\eps_k(t)}{\lambda}}
		+L_{S,\lambda}E_k(t).
		\]
		Moreover, $\norm{s_t^\star}_{\Hh}\le R_\lambda$, so \eqref{eq:force-state-lip} gives
		\[
		\abs{F(Z;s_t^\star)-F(Y;s_t^\star)}_N
		\le L\abs{Z-Y}_N.
		\]
		
		Set $\Delta_t^{k+1}:=\widetilde X_t^{k+1}-X_t$. Inserting $F(\widetilde X_t^{k+1};s_t^\star)$ gives
		\[
		\begin{aligned}
			\dot\Delta_t^{k+1}
			={}&
			F(\widetilde X_t^{k+1};\widetilde s_t^k)
			-F(\widetilde X_t^{k+1};s_t^\star)\\
			&+
			F(\widetilde X_t^{k+1};s_t^\star)
			-F(X_t;s_t^\star)
			+\mathcal R_t^{k+1}.
		\end{aligned}
		\]
		Using \eqref{eq:force-score-lip} and the preceding estimates,
		\[
		\abs{\dot\Delta_t^{k+1}}_N
		\le
		LE_{k+1}(t)+\beta E_k(t)+\eta_k(t).
		\]
		Since $\Delta^{k+1}$ is absolutely continuous and $\Delta_0^{k+1}=0$,
		\begin{equation}\label{eq:inexact-pre-gronwall}
			E_{k+1}(t)
			\le
			\int_0^t
			\left(
			LE_{k+1}(r)+\beta E_k(r)+\eta_k(r)
			\right)\dd r.
		\end{equation}
		The integral form of Gronwall's inequality proves \eqref{eq:inexact-one-step}. This argument does not require differentiating the norm at its zeros.
		
		For nonnegative $f$, define
		\[
		(\mathcal Kf)(t):=\int_0^t e^{L(t-r)}f(r)\dd r.
		\]
		An induction using Tonelli's theorem gives, for $p\ge1$,
		\[
		(\mathcal K^pf)(t)
		=
		\int_0^t e^{L(t-r)}
		\frac{(t-r)^{p-1}}{(p-1)!}f(r)\dd r.
		\]
		Repeated substitution in $E_{k+1}\le\beta\mathcal KE_k+\mathcal K\eta_k$ therefore yields
		\[
		E_m
		\le
		\beta^m\mathcal K^mE_0
		+
		\sum_{j=0}^{m-1}
		\beta^j\mathcal K^{j+1}\eta_{m-1-j}.
		\]
		Bounding $e^{L(t-r)}$ by $e^{LT}$ and $E_0(r)$ by $E_0^*$ gives \eqref{eq:inexact-unrolled}.
		
		For the asymptotic conclusions, put $a_k:=\norm{\eta_k}_{L^1(0,T)}$ and $w_j:=(\beta T)^j/j!$. Then $\sum_{j\ge0}w_j=e^{\beta T}$ and
		\[
		E_m^*\le e^{LT}
		\left(E_0^*w_m+\sum_{j=0}^{m-1}w_ja_{m-1-j}\right).
		\]
		If $a_k\to0$, then $M:=\sup_k a_k<\infty$. For $m\ge J$,
		\[
		\sum_{j=0}^{m-1}w_ja_{m-1-j}
		\le
		\sum_{j=0}^{J-1}w_ja_{m-1-j}
		+M\sum_{j\ge J}w_j.
		\]
		First let $m\to\infty$ and then $J\to\infty$. Together with $w_m\to0$, this proves part~(i). If $a_k\le\overline\eta$, the convolution sum is at most $e^{\beta T}\overline\eta$, which proves part~(ii).
	\end{proof}
	
	The quantity $E_0^*$ in \eqref{eq:inexact-unrolled} admits a computable upper bound. Indeed, $\abs{\dot X_t^{\lambda,N}}_N\le V_\lambda$ by \eqref{eq:force-bound} and \eqref{eq:main-score-bound}, so
	\begin{equation}\label{eq:initial-error-bound}
		E_0^*
		\le
		\sup_{0\le t\le T}\abs{\widetilde X_t^0-Z^0}_N
		+V_\lambda T.
	\end{equation}
	For the standard initialization $\widetilde X^0\equiv Z^0$, this reduces to $E_0^*\le V_\lambda T$. After substituting this upper bound, the estimate depends only on the prescribed initial guess, the available global score-gap and system-residual bounds, and the analytic constants. Integrated sufficient conditions and a conservative implicit-midpoint system solver with a computable defect are given in Appendix~\ref{sm:inexact}.

	\begin{remark}	\label{rem:restricted-class}
		Suppose the score step is solved only over a nonlinear class $\mathcal G\subset\Hh$. Then, with $\J:=\J^\lambda_{\mu_{\widetilde X_t^k}}$, the quantity required in \eqref{eq:inexact-score-gap} decomposes as
		\[
		\underbrace{\J(\widetilde s)-\inf_{s\in\mathcal G}\J(s)}_{\text{optimization}} +\underbrace{\inf_{s\in\mathcal G}\J(s)-\inf_{s\in\Hh}\J(s)}_{\text{approximation}}.
		\]
		An optimizer gap relative only to $\mathcal G$ does not control the second term. When density-level rather than same-sample empirical accuracy is desired, a separate generalization term is also necessary.
	\end{remark}
	
	\section{Singular collision kernels via score-difference cancellation}
	\label{sec:singular-effective-interaction}
	
	At fixed $\lambda>0$ and under \cref{ass:score-space}, we extend the theory using the cancellation in $A(x-y)(s(x)-s(y))$. This mechanism was introduced in \cite{CarrilloDelgadinoWu2023} for a mollified score, where Lemma~2.5 gives the Lipschitz estimate with weight $\abs{x-y}^{\gamma+2}$ and Theorem~1.3 the mean-field limit in $W_\infty$ for $\gamma\in(-3,0]$, with an initial minimal-separation condition for $\gamma\in(-3,-2)$. Here the score is the Tikhonov minimizer, so the estimates are formulated for $s\in\Hh$ with constants proportional to $\norm s_{\Hh}$, which is what the $\Hh$-norm stability \eqref{eq:score-H-W1} feeds into the Volterra recursion: an effective bounded-Lipschitz condition covers far-field-truncated Landau kernels with $\gamma\ge-2$ and the uncut case $\gamma=-2$ (the Coulomb case when $d=2$), and for the uncut three-dimensional Coulomb kernel, integrated estimates give local factorial convergence of the continuum iteration from bounded, compactly supported initial densities.
	
	For a matrix $A$ defined on $\R^d\setminus\{0\}$ and $s\in\Hh$, set
	\begin{equation}\label{eq:effective-pair-kernel}
		K_s(x,y):=A(x-y)(s(x)-s(y)),\qquad x\ne y,
	\end{equation}
	and set $K_s(x,x)=0$ whenever this gives a continuous extension. We impose regularity directly on this pair interaction, so that the score difference can compensate for the singularity of $A$ (cf.\ \cite[Lemma~2.5]{CarrilloDelgadinoWu2023}).
	\begin{assumption}[Effective interaction regularity]
		\label{ass:effective-interaction}
		There are constants $\kappa_0,\kappa_1<\infty$ such that, for every $s\in\Hh$,
		\begin{align}
			\norm{K_s}_{L^\infty(\R^d\times\R^d)}
			&\le\kappa_0\norm{s}_{\Hh},\label{eq:effective-bound}\\
			\abs{K_s(x,y)-K_s(x',y')}
			&\le\kappa_1\norm{s}_{\Hh}
			\bigl(\abs{x-x'}+\abs{y-y'}\bigr).
			\label{eq:effective-lipschitz}
		\end{align}
	\end{assumption}
	By linearity in $s$, the first bound also controls score perturbations in $\Hh$, so we use the Hilbert-space stability constant
	\begin{equation}\label{eq:Lambda-S-definition}
		\Lambda_{S,\lambda}
		:=\frac1\lambda\left(c_2+2c_0c_1\frac{c_D}{\lambda}\right),
	\end{equation}
	so that \eqref{eq:score-H-W1} reads $\norm{S_\lambda(\mu)-S_\lambda(\nu)}_{\Hh}\le\Lambda_{S,\lambda}W_1(\mu,\nu)$. The following theorem replaces \cref{ass:kernel} by \cref{ass:effective-interaction} and retains the remaining hypotheses of the preceding results.
	
	\begin{theorem}[Extension of the main results under effective-interaction regularity]
		\label{thm:effective-interaction-extension}
		Under \cref{ass:score-space,ass:effective-interaction}, define the particle force, including the zero diagonal convention, by
		\[
		F_i(Z;s):=-\frac1N\sum_{j=1}^N K_s(z_i,z_j),
		\]
		and set
		\begin{equation}\label{eq:effective-constants}
			V_\lambda^{\rm eff}:=\kappa_0R_\lambda,
			\qquad L_{F,\lambda}^{\rm eff}:=2\kappa_1R_\lambda,
			\qquad \beta_\lambda^{\rm eff}:=\kappa_0\Lambda_{S,\lambda}.
		\end{equation}
		Then the convergence theorem for the exact iteration, its factorial error bound and stopping certificate, the characteristic mean-field theorem, and the theorem on the continuum alternating iteration and commuting limits remain valid with $(V_\lambda,L_{F,\lambda},\beta_\lambda)$ replaced by the constants in \eqref{eq:effective-constants}.  The Dobrushin exponent is
		\begin{equation}\label{eq:effective-Dobrushin-constant}
			C_{\rm mf}^{\rm eff}
			:=2\kappa_1R_\lambda+\kappa_0\Lambda_{S,\lambda}.
		\end{equation}
		
		\Cref{thm:inexact} also remains valid with $L_{F,\lambda}$ and $\beta_\lambda$ replaced by $L_{F,\lambda}^{\rm eff}$ and $\beta_\lambda^{\rm eff}$, and with the forcing quantity replaced by
		\begin{equation}\label{eq:effective-inexact-forcing}
			\eta_k^{\rm eff}(t)
			:=
			\kappa_0\sqrt{\frac{\eps_k(t)}{\lambda}}
			+r_{k+1}(t).
		\end{equation}
		Here the same admissibility and global score-gap assumptions are retained, while the integrability condition is	$\eta_k^{\rm eff}\in L^1(0,T)$ for every $k$. If $A$ is even and symmetric off the diagonal and $A(z)z=0$ for $z\ne0$, the particle conservation statements and the conservation properties of the implicit midpoint scheme (Corollary~\ref{cor:implicit-midpoint}) also remain valid.
	\end{theorem}
	
	\begin{proof}
		\emph{Step 1: Force estimates.}  Since $s\mapsto K_s$ is linear and the diagonal terms vanish, \eqref{eq:effective-bound}--\eqref{eq:effective-lipschitz} give, for $Z,Y\in(\R^d)^N$ and $s,r\in\Hh$,
		\begin{equation}\label{eq:effective-force-bounds}
			\begin{aligned}
				\abs{F(Z;s)}_N&\le\kappa_0\norm{s}_{\Hh},\\
				\abs{F(Z;s)-F(Y;s)}_N
				&\le\frac{\kappa_1\norm s_{\Hh}}{N^2}\sum_{i,j=1}^N
				\bigl(\abs{z_i-y_i}+\abs{z_j-y_j}\bigr)
				=2\kappa_1\norm{s}_{\Hh}\abs{Z-Y}_N,\\
				\abs{F(Z;s)-F(Z;r)}_N
				&=\abs{F(Z;s-r)}_N\le\kappa_0\norm{s-r}_{\Hh}.
			\end{aligned}
		\end{equation}
		In \crefrange{sec:stable-score-learner}{sec:inexact}, the collision matrix enters only through \eqref{eq:force-bound}, \eqref{eq:force-state-lip}, and \eqref{eq:force-score-lip}, applied to scores bounded by \eqref{eq:score-H-bound} whose differences are controlled by \eqref{eq:score-W1}.  By \eqref{eq:score-H-bound} and \eqref{eq:score-H-W1}, the three bounds in \eqref{eq:effective-force-bounds} replace these with speed $\kappa_0R_\lambda$, state-Lipschitz constant $2\kappa_1R_\lambda$, and score-change constant $\kappa_0\Lambda_{S,\lambda}$ per unit $W_1$-distance.
		
		\emph{Step 2: The exact alternating iteration.}  For a continuous score path with $\norm{s_t}_{\Hh}\le R_\lambda$, the second bound in \eqref{eq:effective-force-bounds} makes $F(\cdot\,;s_t)$ globally Lipschitz and the third makes it continuous in $t$, so every system step is uniquely and globally solvable.  In the proof of \cref{thm:main-outer}, the two force differences are now bounded by $2\kappa_1R_\lambda D_k(r)$ and $\kappa_0\norm{s_r^k-s_r^{k-1}}_{\Hh}\le\kappa_0\Lambda_{S,\lambda}D_{k-1}(r)$, and $D_0(t)\le\kappa_0R_\lambda t$.  The remainder of that proof and the proof of \cref{cor:outer-stopping} apply verbatim with the constants \eqref{eq:effective-constants}.
		
		\emph{Step 3: Mean-field statements.}  The drift is $b_\lambda[\mu](x)=-\int K_{S_\lambda(\mu)}(x,y)\dd\mu(y)$.  By \eqref{eq:effective-bound}--\eqref{eq:effective-lipschitz} and \eqref{eq:score-H-bound}, $\norm{b_\lambda[\mu]}_\infty\le\kappa_0R_\lambda$ and $\Lips_x(b_\lambda[\mu])\le\kappa_1R_\lambda$.  For $\mu,\nu\in\Pone(\R^d)$,
		\[
		b_\lambda[\mu](x)-b_\lambda[\nu](x)
		=-\int\bigl(K_{S_\lambda(\mu)}-K_{S_\lambda(\nu)}\bigr)(x,y)\dd\mu(y)
		-\int K_{S_\lambda(\nu)}(x,y)\dd(\mu-\nu)(y).
		\]
		The first term is at most $\kappa_0\Lambda_{S,\lambda}W_1(\mu,\nu)$ by linearity, \eqref{eq:effective-bound}, and \eqref{eq:score-H-W1}; the second is at most $\kappa_1R_\lambda W_1(\mu,\nu)$ by Kantorovich--Rubinstein duality, since $y\mapsto K_{S_\lambda(\nu)}(x,y)$ is $\kappa_1R_\lambda$-Lipschitz. Hence \cref{lem:mean-field-drift} holds with $2a_0M_\lambda$, $L_b$, and $K_b$ replaced by $\kappa_0R_\lambda$, $\kappa_1R_\lambda$, and $\kappa_1R_\lambda+\kappa_0\Lambda_{S,\lambda}$.  The proof of \cref{thm:main-mean-field} then gives the characteristic well-posedness and the Dobrushin estimate with exponent $\kappa_1R_\lambda+(\kappa_1R_\lambda+\kappa_0\Lambda_{S,\lambda}) =C_{\rm mf}^{\rm eff}$.  The identification of the particle system as a characteristic solution uses $F_i(Z;S_\lambda(\mu_Z))=b_\lambda[\mu_Z](z_i)$, which holds because the diagonal term vanishes on both sides.  For the continuum alternating iteration, $B[s,\rho](x):=-\int K_s(x,y)\dd\rho(y)$ satisfies the same spatial and measure bounds and $\norm{B[s,\rho]-B[\bar s,\rho]}_\infty\le\kappa_0\norm{s-\bar s}_{\Hh}$. Therefore \eqref{eq:prescribed-score-stability} holds with exponent $2\kappa_1R_\lambda=L_{F,\lambda}^{\rm eff}$ and with $2a_0\norm{s_r-\bar s_r}_\infty$ replaced by $\kappa_0\norm{s_r-\bar s_r}_{\Hh}$, and the proofs of \cref{thm:continuum-outer-iteration} and \cref{cor:main-joint-limit} apply with $\beta_\lambda^{\rm eff}=\kappa_0\Lambda_{S,\lambda}$.
		
		\emph{Step 4: Inexact solves.} Use the notation	$X_t:=X_t^{\lambda,N}$ and	$s_t^\star:=S_\lambda(\mu_{X_t})$ from the proof of	\cref{thm:inexact}. By \eqref{eq:exact-gap}, \eqref{eq:score-H-W1}, and	\eqref{eq:empirical-W1-coupling},
		\[
		\norm{\widetilde s_t^k-s_t^\star}_{\Hh}	\le	\sqrt{\frac{\eps_k(t)}{\lambda}}	+\Lambda_{S,\lambda}E_k(t),	\qquad	\norm{s_t^\star}_{\Hh}\le R_\lambda.
		\]
		Insert $F(\widetilde X_t^{k+1};s_t^\star)$ in the force 	difference, as in the proof of \cref{thm:inexact}.	The bounds \eqref{eq:effective-force-bounds} then give
		\[
		E_{k+1}(t)
		\le
		\int_0^t
		\left(
		L_{F,\lambda}^{\rm eff}E_{k+1}(r)
		+\beta_\lambda^{\rm eff}E_k(r)
		+\eta_k^{\rm eff}(r)
		\right)\dd r.
		\]
		Gronwall's inequality and the same iterated-convolution argument prove all conclusions of \cref{thm:inexact} with these replacements.
		
		\emph{Step 5: Conservation.}  If $A$ is even and symmetric off the diagonal and $A(z)z=0$ for $z\ne0$, then $K_s(y,x)=-K_s(x,y)$ and $(x-y)^\top K_s(x,y)=\bigl(A(x-y)(x-y)\bigr)^\top\bigl(s(x)-s(y)\bigr)=0$ for $x\ne y$, while the diagonal terms vanish.  Hence $\sum_iF_i(Z;s)=0$ and
		\[
		\sum_iz_i\cdot F_i(Z;s)
		=-\frac1{2N}\sum_{i,j}(z_i-z_j)^\top K_s(z_i,z_j)=0 .
		\]
		These are the only identities used in \cref{prop:conservation} and Corollary~\ref{cor:implicit-midpoint} (with $Z$ replaced by the midpoint $\overline Z$ in the latter).
	\end{proof}
	
	\begin{proposition}[Far-field-truncated Landau kernels with $\gamma\ge-2$]
		\label{prop:singular-landau-range}
		Let $d\ge2$, let $\chi_R\in C_c^\infty([0,\infty))$ equal one near zero, and, for $z\ne0$, define
		\begin{equation}\label{eq:origin-singular-landau-kernel}
			A_{\gamma,R}(z)
			:=C_\gamma\chi_R(\abs z)\abs z^{\gamma+2}
			\left(I-\frac{z\otimes z}{\abs z^2}\right).
		\end{equation}
		If $\gamma\ge-2$, then $K_s$ from \eqref{eq:effective-pair-kernel} extends by $K_s(x,x)=0$ and satisfies \cref{ass:effective-interaction}.  Consequently all conclusions of \cref{thm:effective-interaction-extension} hold.  This includes $-2\le\gamma<-1$, where $A_{\gamma,R}$ itself is not globally Lipschitz. For $\gamma=-2$, the conclusion also holds without a far-field cutoff, with $\chi_R\equiv1$; for $d=2$ this is the uncut two-dimensional Coulomb kernel.
	\end{proposition}
	
	\begin{proof}
		Put $p:=\gamma+2\ge0$, $P(z):=I-z\otimes z/\abs z^2$, and $M(z):=C_\gamma\chi_R(\abs z)\abs z^pP(z)$, so that $K_s(x,y)=M(x-y)\bigl(s(x)-s(y)\bigr)$ for $x\ne y$.  Let $\rho_R<\infty$ be such that $\chi_R(r)=0$ for $r\ge\rho_R$.  Throughout, $z=x-y$ and $L_s:=\norm{Ds}_\infty\le c_1\norm s_{\Hh}$.
		
		\emph{Step 1: Size and diagonal extension.}  Since $P(z)$ is an orthogonal projection and $\abs{s(x)-s(y)}\le L_s\abs z$,
		\begin{equation}\label{eq:singular-cancellation-size}
			\abs{K_s(x,y)}
			\le\abs{C_\gamma}\norm{\chi_R}_\infty L_s\abs z^{p+1}\mathbf 1_{\{\abs z<\rho_R\}}
			\le\abs{C_\gamma}\norm{\chi_R}_\infty\rho_R^{p+1}L_s .
		\end{equation}
		Since $p+1>0$, $K_s(x,y)\to0$ as $\abs{x-y}\to0$, so the extension $K_s(x,x)=0$ is continuous, and \eqref{eq:effective-bound} holds with $\kappa_0=\abs{C_\gamma}\norm{\chi_R}_\infty\rho_R^{p+1}c_1$.
		
		\emph{Step 2: Derivative bound off the diagonal.}  The map $M$ is smooth on $\R^d\setminus\{0\}$ and vanishes for $\abs z\ge\rho_R$.  Since $\abs{D(\abs z^p)}=p\abs z^{p-1}$, $\abs{DP(z)}\le C\abs z^{-1}$, and $\abs{D(\chi_R(\abs z))}\le\norm{\chi_R'}_\infty$, there is $C_M$, depending only on $(d,\gamma,C_\gamma,\chi_R)$, with
		\[
		\abs{DM(z)}\le C_M\bigl(\abs z^{p-1}+\abs z^p\bigr)
		\mathbf 1_{\{\abs z<\rho_R\}},\qquad z\ne0 .
		\]
		For $x\ne y$ and $h\in\R^d$,
		\begin{align*}
			D_xK_s(x,y)h&=\bigl(DM(z)h\bigr)\bigl(s(x)-s(y)\bigr)+M(z)Ds(x)h,\\
			D_yK_s(x,y)h&=-\bigl(DM(z)h\bigr)\bigl(s(x)-s(y)\bigr)-M(z)Ds(y)h .
		\end{align*}
		Using $\abs{s(x)-s(y)}\le L_s\abs z$ and $\abs z<\rho_R$ on the support,
		\begin{equation}\label{eq:singular-cancellation-gradient}
			\abs{D_xK_s(x,y)}+\abs{D_yK_s(x,y)}
			\le C_1L_s,
			\qquad
			C_1:=2C_M\bigl(\rho_R^p+\rho_R^{p+1}\bigr)
			+2\abs{C_\gamma}\norm{\chi_R}_\infty\rho_R^p .
		\end{equation}
		
		\emph{Step 3: Global Lipschitz bound.}  Let $w=(x,y)$ and $w'=(x',y')$, and put $w(\tau):=(1-\tau)w+\tau w'$.  The diagonal $\Delta:=\{(u,u):u\in\R^d\}$ is a linear subspace.  If the segment meets $\Delta$ in two points, it lies in $\Delta$ and $K_s$ vanishes along it. Otherwise it meets $\Delta$ in at most one parameter $\tau_*$.  On each subinterval not containing $\tau_*$, the function $\tau\mapsto K_s(w(\tau))$ is $C^1$, and by \eqref{eq:singular-cancellation-gradient} its derivative is bounded by $C_1L_s(\abs{x-x'}+\abs{y-y'})$.  Since $K_s$ is continuous at the diagonal by Step~1, integrating over both subintervals gives
		\[
		\abs{K_s(x,y)-K_s(x',y')}\le C_1L_s\bigl(\abs{x-x'}+\abs{y-y'}\bigr).
		\]
		This is \eqref{eq:effective-lipschitz} with $\kappa_1=C_1c_1$.  The constants depend only on $(d,\gamma,R,C_\gamma,\chi_R,c_1)$.
		
		\emph{Step 4: The uncut case $\gamma=-2$.}  Let $\chi_R\equiv1$, write $e=z/\abs z$ and $P_e=I-e\otimes e$, so that $A_{-2,R}(z)=C_{-2}P_e$ and $K_s(x,y)=C_{-2}P_e\bigl(s(x)-s(y)\bigr)$.  Then $\abs{K_s(x,y)}\le\abs{C_{-2}}\min\{L_s\abs z,2\norm s_\infty\}$, which gives the continuous diagonal extension and $\kappa_0=2\abs{C_{-2}}c_0$.  For $x\ne y$ and $h\in\R^d$, one has $D_xe[h]=P_eh/\abs z$, which is orthogonal to $e$.  With $v=s(x)-s(y)$,
		\[
		(D_xP_e[h])v=-(D_xe[h])(e\cdot v)-e\,(D_xe[h]\cdot v)
		\]
		is a sum of two orthogonal vectors.  Since $D_xe[h]\cdot v=D_xe[h]\cdot P_ev$, its norm is at most $\abs{D_xe[h]}\bigl((e\cdot v)^2+\abs{P_ev}^2\bigr)^{1/2} =\abs{D_xe[h]}\,\abs v\le\abs h\,\abs v/\abs z\le L_s\abs h$.  Adding $\abs{P_eDs(x)h}\le L_s\abs h$ gives $\abs{D_xK_s}\le2\abs{C_{-2}}L_s$, and the same computation with $D_ye[h]=-P_eh/\abs z$ gives $\abs{D_yK_s}\le2\abs{C_{-2}}L_s$.  The segment argument of Step~3 then gives $\kappa_1=2\abs{C_{-2}}c_1$.
	\end{proof}
	
	\begin{proposition}[Sharpness of $\gamma=-2$: very soft and Coulomb potentials]
		\label{prop:coulomb-effective-obstruction}
		Let $d=3$, let $A=A_{\gamma,R}$ be given by \eqref{eq:origin-singular-landau-kernel}, and write $P_e:=I-e\otimes e$. Suppose $s\in C_b^2(\R^3;\R^3)$ satisfies $P_eDs(x_0)e\ne0$ for some $x_0\in\R^3$ and some unit vector $e$.  If $-3<\gamma<-2$, the diagonal modulus of $K_s$ at $x_0$ is exactly of order $\abs{x-y}^{\gamma+3}$, with $\gamma+3\in(0,1)$, so $K_s$ is not Lipschitz.  If $\gamma=-3$ (the three-dimensional Coulomb exponent), $K_s$ has no continuous extension to the diagonal.  Such scores exist in the Gaussian RKHS of \cref{sec:framework}, so for that space neither $A_{\gamma,R}$ with $-3<\gamma<-2$ nor the three-dimensional Coulomb matrix satisfies \cref{ass:effective-interaction}, and the present empirical-measure Picard/Dobrushin proof does not extend to the three-dimensional Coulomb kernel merely by exploiting $A(z)z=0$.
	\end{proposition}
	
	\begin{proof}
		For $r>0$ so small that the cutoff equals one, Taylor's formula gives
		\begin{equation}\label{eq:coulomb-directional-limit}
			K_s(x_0+re,x_0)
			=C_\gamma r^{\gamma+2}P_e\bigl(s(x_0+re)-s(x_0)\bigr)
			=C_\gamma r^{\gamma+3}\bigl(P_eDs(x_0)e+\omega(r)\bigr),
		\end{equation}
		with $\abs{\omega(r)}\le\frac r2\norm{D^2s}_\infty$.  In general, $\abs{K_s(x,y)}$ is at most $C\norm{Ds}_\infty\abs{x-y}^{\gamma+3}$, so this radial order is sharp.  For $-3<\gamma<-2$, $K_s$ extends continuously by zero to the diagonal, but its difference quotient along $x_0+re$ is of order $r^{\gamma+2}\to\infty$, so no Lipschitz extension exists.  For $\gamma=-3$, the right-hand side tends to $C_{-3}P_eDs(x_0)e\ne0$; replacing $e$ by $-e$ leaves $P_e$ unchanged and changes the sign of the limit, so no continuous diagonal extension exists.  Finally, for the Gaussian kernel $K_\sigma$ and $s=K_\sigma(\cdot,y)c$,
		\[
		P_eDs(x_0)e=-\sigma^{-2}e^{-\abs{x_0-y}^2/(2\sigma^2)}\bigl((x_0-y)\cdot e\bigr)P_ec,
		\]
		which is nonzero whenever $(x_0-y)\cdot e\ne0$ and $c$ is not parallel to $e$.
	\end{proof}
	
	The preceding obstruction is pointwise.  For an absolutely continuous law in three dimensions, the residual angular singularity is integrable after characteristics are coupled by their initial labels.  This gives the following continuum result for the uncut three-dimensional Coulomb matrix. The two-point estimate \eqref{eq:Coulomb-pair-difference} below is the Coulomb case of \cite[Lemma~2.5]{CarrilloDelgadinoWu2023}; the integrated bounds \eqref{eq:Coulomb-flow-stability}--\eqref{eq:Coulomb-spatial-velocity} along labeled flows are what the continuum alternating iteration requires.
	
	\begin{lemma}	\label{lem:integrated-coulomb-stability}
		In dimension three, let
		\begin{equation}\label{eq:Coulomb-matrix}
			A_C(z):=\frac{c_C}{\abs z}
			\left(I-\widehat z\otimes\widehat z\right),
			\qquad \widehat z:=\frac z{\abs z},\quad z\ne0.
		\end{equation}
		For $s\in C_b^2(\R^3;\R^3)$, put $L_s:=\norm{Ds}_\infty$ and $H_s:=\norm{D^2s}_\infty$.  If $r=\abs{x-y}>0$, $r'=\abs{x'-y'}>0$, and $\delta=\abs{x-x'}+\abs{y-y'}$, then
		\begin{equation}\label{eq:Coulomb-pair-difference}
			\abs{K_s(x,y)-K_s(x',y')}
			\le \abs{c_C}\,\delta
			\left[H_s+L_s\left(\frac1r+\frac1{r'}\right)\right].
		\end{equation}
		
		Let $\rho_0(\dd a)=f_0(a)\dd a$, where $f_0\in L^\infty(\R^3)$ is a probability density and $D:=\operatorname{supp}\rho_0$ is compact.  All label maps below are continuous on $D$, and we write
		\[
		\norm{X}_{0,D}:=\sup_{a\in D}\abs{X(a)}.
		\]
		Define
		\begin{equation}\label{eq:initial-Riesz-constant}
			J_0:=\sup_{a\in D}\int_D\frac{f_0(b)}{\abs{a-b}}\dd b<\infty.
		\end{equation}
		For a continuous injective map $X:D\to\R^3$, set
		\begin{equation}\label{eq:Coulomb-label-velocity}
			\mathcal V_s[X](a)
			:=-\int_D A_C(X(a)-X(b))
			\bigl(s(X(a))-s(X(b))\bigr)f_0(b)\dd b,
		\end{equation}
		where the integrand at $b=a$ is set to zero.  If $X,Y\in C(D;\R^3)$ are both $m$-co-Lipschitz, that is, $\abs{X(a)-X(b)}\ge m\abs{a-b}$ and $\abs{Y(a)-Y(b)}\ge m\abs{a-b}$ for all $a,b\in D$, then
		\begin{align}
			\norm{\mathcal V_s[X]-\mathcal V_s[Y]}_{0,D}
			&\le2\abs{c_C}\left(H_s+\frac{2L_sJ_0}{m}\right)
			\norm{X-Y}_{0,D},
			\label{eq:Coulomb-flow-stability}\\
			\norm{\mathcal V_s[X]}_{0,D}
			&\le\abs{c_C}L_s,
			\label{eq:Coulomb-velocity-bound}\\
			\norm{\mathcal V_s[X]-\mathcal V_r[X]}_{0,D}
			&\le\abs{c_C}\Lips(s-r).
			\label{eq:Coulomb-score-stability}
		\end{align}
		Moreover,
		\begin{equation}\label{eq:Coulomb-spatial-velocity}
			\abs{\mathcal V_s[X](a)-\mathcal V_s[X](a')}
			\le\abs{c_C}\left(H_s+\frac{2L_sJ_0}{m}\right)
			\abs{X(a)-X(a')}.
		\end{equation}
		If $d_0=\operatorname{diam}D$, then $J_0\le2\pi\norm{f_0}_\infty d_0^2$.
	\end{lemma}
	
	\begin{proof}
		\emph{Step 1: The pair estimate.}  For $x\ne y$ write $z=x-y$, $r=\abs z$, $e=z/r$, and $P_e=I-e\otimes e$.  Since $s(x)-s(y)=\int_0^1Ds\bigl(y+\theta z\bigr)z\dd\theta$, the singular factor is absorbed into a difference quotient:
		\begin{equation}\label{eq:Coulomb-difference-quotient}
			K_s(x,y)=c_CP_e\,g(x,y),
			\qquad
			g(x,y):=\int_0^1Ds\bigl((1-\theta)y+\theta x\bigr)e\dd\theta,
			\qquad \abs{g(x,y)}\le L_s .
		\end{equation}
		Let $z'=x'-y'$, $r'=\abs{z'}$, and $e'=z'/r'$.  Then
		\[
		K_s(x,y)-K_s(x',y')
		=c_C(P_e-P_{e'})g(x,y)+c_CP_{e'}\bigl(g(x,y)-g(x',y')\bigr).
		\]
		The matrix $P_e-P_{e'}=e'\otimes e'-e\otimes e$ vanishes on $\{e,e'\}^\perp$ and has eigenvalues $\pm\sin\theta$ on $\operatorname{span}\{e,e'\}$, where $\theta$ is the angle between $e$ and $e'$; hence $\norm{P_e-P_{e'}}_{\rm op}=\sin\theta\le2\sin(\theta/2)=\abs{e-e'}$. Next, the points $p_\theta=(1-\theta)y+\theta x$ and $p'_\theta=(1-\theta)y'+\theta x'$ satisfy $\abs{p_\theta-p'_\theta}\le\delta$, so
		\[
		\begin{aligned}
			\abs{g(x,y)-g(x',y')}
			&\le\int_0^1\Bigl(\abs{Ds(p_\theta)-Ds(p'_\theta)}+\abs{Ds(p'_\theta)}\,\abs{e-e'}\Bigr)\dd\theta\\
			&\le H_s\delta+L_s\abs{e-e'} .
		\end{aligned}
		\]
		Consequently $\abs{K_s(x,y)-K_s(x',y')}\le\abs{c_C}\bigl(H_s\delta+2L_s\abs{e-e'}\bigr)$. Finally, expanding $\abs{z-z'}^2=r^2+r'^2-2z\cdot z'$ gives the identity
		\[
		\abs{e-e'}^2=2-\frac{2z\cdot z'}{rr'}
		=\frac{\abs{z-z'}^2-(r-r')^2}{rr'}\le\frac{\abs{z-z'}^2}{rr'},
		\]
		and $\abs{z-z'}\le\delta$.  By the arithmetic--geometric mean inequality, $\abs{e-e'}\le\delta/\sqrt{rr'}\le\frac\delta2(r^{-1}+r'^{-1})$, which proves \eqref{eq:Coulomb-pair-difference}.
		
		\emph{Step 2: Flow stability.}  Fix $a\in D$ and $b\in D\setminus\{a\}$, and apply \eqref{eq:Coulomb-pair-difference} with $(x,y)=(X(a),X(b))$ and $(x',y')=(Y(a),Y(b))$.  Then $\delta\le\abs{X(a)-Y(a)}+\abs{X(b)-Y(b)}\le2\norm{X-Y}_{0,D}$, and the co-Lipschitz bounds give $r,r'\ge m\abs{a-b}$.  Hence
		\[
		\abs{K_s(X(a),X(b))-K_s(Y(a),Y(b))}
		\le2\abs{c_C}\norm{X-Y}_{0,D}
		\left(H_s+\frac{2L_s}{m\abs{a-b}}\right).
		\]
		Integrating against $f_0(b)\dd b$ over $D$, and using $\int f_0=1$ and \eqref{eq:initial-Riesz-constant}, proves \eqref{eq:Coulomb-flow-stability}; the single label $b=a$ is a $\rho_0$-null set.
		
		\emph{Step 3: Velocity and score bounds.}  By \eqref{eq:Coulomb-difference-quotient}, $\abs{K_s}\le\abs{c_C}L_s$, so the integral \eqref{eq:Coulomb-label-velocity} converges absolutely and $\abs{\mathcal V_s[X](a)}\le\abs{c_C}L_s\int f_0=\abs{c_C}L_s$, which is \eqref{eq:Coulomb-velocity-bound}.  Since $K_s-K_r=K_{s-r}$, the same bound applied to $s-r$ gives \eqref{eq:Coulomb-score-stability}, because $\norm{D(s-r)}_\infty=\Lips(s-r)$.
		
		\emph{Step 4: Spatial regularity.}  Fix $a,a'\in D$ and $b\in D\setminus\{a,a'\}$, and apply \eqref{eq:Coulomb-pair-difference} with $(x,y)=(X(a),X(b))$ and $(x',y')=(X(a'),X(b))$.  Then $\delta=\abs{X(a)-X(a')}$, $r\ge m\abs{a-b}$, and $r'\ge m\abs{a'-b}$, so
		\begin{align*}
			&\abs{K_s(X(a),X(b))-K_s(X(a'),X(b))}\\
			&\qquad\le\abs{c_C}\abs{X(a)-X(a')}
			\left(H_s+\frac{L_s}m\Bigl(\frac1{\abs{a-b}}+\frac1{\abs{a'-b}}\Bigr)\right).
		\end{align*}
		Integrating against $f_0(b)\dd b$ and using \eqref{eq:initial-Riesz-constant} at $a$ and at $a'$ proves \eqref{eq:Coulomb-spatial-velocity}; the labels $b\in\{a,a'\}$ form a $\rho_0$-null set.  In particular $\mathcal V_s[X]$ is continuous on $D$.
		
		\emph{Step 5: The Riesz constant.}  For $a\in D$ one has $D-a\subset \overline{B_{d_0}(0)}$, so
		\[
		J_0\le\norm{f_0}_\infty
		\int_{B_{d_0}(0)}\frac{\dd z}{\abs z}
		=4\pi\norm{f_0}_\infty\int_0^{d_0}\rho\dd\rho
		=2\pi\norm{f_0}_\infty d_0^2.
		\]
	\end{proof}
	
	\begin{theorem}[Convergence of the continuum Coulomb alternating iteration]
		\label{thm:continuum-Coulomb-waveform}
		Let $d=3$, adopt \cref{ass:score-space} and the setting of \cref{lem:integrated-coulomb-stability}, fix $\lambda>0$, and put
		\begin{equation}\label{eq:Coulomb-waveform-constants}
			H_\lambda:=c_HR_\lambda,
			\quad
			Q_\lambda:=\abs{c_C}\bigl(H_\lambda+4L_\lambda J_0\bigr),
			\quad
			V_C:=\abs{c_C}L_\lambda,
			\quad
			\beta_C:=\abs{c_C}c_1\Lambda_{S,\lambda}.
		\end{equation}
		For $4Q_\lambda T_0<1$, that is, for $T_0<\lambda/\bigl(4\abs{c_C}(c_H+4c_1J_0)c_D\bigr)$, so that $T_0=O(\lambda)$, define the continuum alternating iteration by
		\begin{align}
			X_t^0(a)&=a,
			&s_t^k&=S_\lambda\bigl((X_t^k)_\#(f_0\dd a)\bigr),
			\label{eq:Coulomb-waveform-score}\\
			\dot X_t^{k+1}(a)&=\mathcal V_{s_t^k}[X_t^{k+1}](a),
			&X_0^{k+1}(a)&=a.
			\label{eq:Coulomb-waveform-system}
		\end{align}
		Let
		\begin{multline}\label{eq:Coulomb-local-flow-class}
			\mathcal X_{T_0}:=\bigl\{X\in C([0,T_0];C(D;\R^3)):\
			X_0=\operatorname{id}_D,\\
			\tfrac12\abs{a-b}\le\abs{X_t(a)-X_t(b)}
			\le2\abs{a-b}\quad
			(a,b\in D,\ 0\le t\le T_0)\bigr\}.
		\end{multline}
		Every prescribed-score system step has a unique solution in $\mathcal X_{T_0}$, and therefore all its flow maps satisfy
		\begin{equation}\label{eq:Coulomb-local-bilipschitz}
			\frac12\abs{a-b}\le\abs{X_t^k(a)-X_t^k(b)}
			\le2\abs{a-b}.
		\end{equation}
		With $L_{F,C}:=2Q_\lambda$, for $k\ge0$,
		\begin{equation}\label{eq:Coulomb-factorial}
			\norm{X_t^{k+1}-X_t^k}_{0,D}
			\le V_Ce^{L_{F,C}t}
			\frac{\beta_C^kt^{k+1}}{(k+1)!}.
		\end{equation}
		Consequently $X^k$ converges in $C([0,T_0];C(D;\R^3))$ to the unique solution of
		\begin{equation}\label{eq:Coulomb-fixed-flow}
			\dot X_t=\mathcal V_{S_\lambda((X_t)_\#(f_0\dd a))}[X_t],
			\qquad X_0(a)=a,
		\end{equation}
		within the class $\mathcal X_{T_0}$.
		
		More generally, let $T<\infty$.  If every prescribed-score system step in \eqref{eq:Coulomb-waveform-system} has a solution $X^{k+1}\in C([0,T];C(D;\R^3))$ and, uniformly in $k,t$,
		\begin{equation}\label{eq:Coulomb-global-colipschitz}
			\abs{X_t^k(a)-X_t^k(b)}\ge m_T\abs{a-b}
			\quad\text{for some }m_T>0,
		\end{equation}
		then the factorial estimate and the convergence conclusions hold on $[0,T]$, and the limiting equation is unique among continuous flows satisfying the same $m_T$-co-Lipschitz bound, with
		\begin{equation}\label{eq:Coulomb-conditional-LF}
			L_{F,C}:=2\abs{c_C}
			\left(H_\lambda+\frac{2L_\lambda J_0}{m_T}\right).
		\end{equation}
		Thus the uncut Coulomb continuum iteration converges on an explicit common local interval without an a priori separation assumption, and conditionally on arbitrary finite intervals.
	\end{theorem}
	
	\begin{proof}
		For $s\in\Hh$ with $\norm s_{\Hh}\le R_\lambda$, \cref{ass:score-space} gives $L_s\le L_\lambda$ and $H_s\le H_\lambda$.  Hence, for ${\tfrac12}$-co-Lipschitz maps, \eqref{eq:Coulomb-flow-stability} and \eqref{eq:Coulomb-spatial-velocity} hold with constants $2Q_\lambda$ and $Q_\lambda$, respectively, and \eqref{eq:Coulomb-velocity-bound} gives the speed bound $V_C$.  Moreover, for $s,r\in\Hh$, \eqref{eq:Coulomb-score-stability} and \eqref{eq:rkhs-bounds-1} give
		\begin{equation}\label{eq:Coulomb-score-H}
			\norm{\mathcal V_s[X]-\mathcal V_r[X]}_{0,D}\le\abs{c_C}c_1\norm{s-r}_{\Hh}.
		\end{equation}
		
		\emph{Step 1: One system step.}  Let $t\mapsto s_t$ be continuous into $\Hh$ with $\norm{s_t}_{\Hh}\le R_\lambda$.  The set $\mathcal X_{T_0}$ is a nonempty closed subset of the Banach space $C([0,T_0];C(D;\R^3))$, since the constant path $\operatorname{id}_D$ belongs to it and the defining inequalities are preserved under uniform limits.  For $X\in\mathcal X_{T_0}$ define
		\[
		(\Phi X)_t(a):=a+\int_0^t\mathcal V_{s_r}[X_r](a)\dd r.
		\]
		The integrand is continuous in $r$ with values in $C(D;\R^3)$: by \eqref{eq:Coulomb-flow-stability} and \eqref{eq:Coulomb-score-H},
		\[
		\norm{\mathcal V_{s_r}[X_r]-\mathcal V_{s_{r'}}[X_{r'}]}_{0,D}
		\le2Q_\lambda\norm{X_r-X_{r'}}_{0,D}+\abs{c_C}c_1\norm{s_r-s_{r'}}_{\Hh},
		\]
		and each $\mathcal V_{s_r}[X_r]$ is continuous on $D$ by \eqref{eq:Coulomb-spatial-velocity}.  Hence $\Phi X\in C([0,T_0];C(D;\R^3))$ and $(\Phi X)_0=\operatorname{id}_D$.  For $X,Y\in\mathcal X_{T_0}$, \eqref{eq:Coulomb-flow-stability} gives
		\[
		\sup_{0\le t\le T_0}\norm{(\Phi X)_t-(\Phi Y)_t}_{0,D}
		\le2Q_\lambda T_0\sup_{0\le t\le T_0}\norm{X_t-Y_t}_{0,D},
		\]
		and \eqref{eq:Coulomb-spatial-velocity} together with the upper bound $\abs{X_r(a)-X_r(b)}\le2\abs{a-b}$ gives
		\[
		\abs{(\Phi X)_t(a)-(\Phi X)_t(b)-(a-b)}
		\le\int_0^tQ_\lambda\abs{X_r(a)-X_r(b)}\dd r
		\le2Q_\lambda T_0\abs{a-b}.
		\]
		Since $2Q_\lambda T_0<\tfrac12$, the second estimate yields $\tfrac12\abs{a-b}\le\abs{(\Phi X)_t(a)-(\Phi X)_t(b)}\le\tfrac32\abs{a-b}$, so $\Phi$ maps $\mathcal X_{T_0}$ into itself, and the first estimate makes it a contraction.  Banach's fixed-point theorem gives a unique solution of the prescribed-score system step in $\mathcal X_{T_0}$.
		
		\emph{Step 2: Admissible score paths.}  For $X\in C([0,T_0];C(D;\R^3))$, the label coupling gives 
		\[
			W_1\bigl((X_t)_\#\rho_0,(X_{t'})_\#\rho_0\bigr)\le\norm{X_t-X_{t'}}_{0,D}.
		\]
		 Hence, by \eqref{eq:score-H-bound} and \eqref{eq:score-H-W1}, the score path $s_t=S_\lambda((X_t)_\#\rho_0)$ is continuous into $\Hh$ and bounded by $R_\lambda$.  Starting from $X^0=\operatorname{id}_D$ and applying Step~1 inductively, every iterate $X^k$ is well defined in $\mathcal X_{T_0}$, with the same $T_0$; this proves \eqref{eq:Coulomb-local-bilipschitz}.
		
		\emph{Step 3: Factorial recursion.}  Put $D_k(t):=\norm{X_t^{k+1}-X_t^k}_{0,D}$.  Since $X_t^1(a)-a=\int_0^t \mathcal V_{s_r^0}[X_r^1](a)\dd r$, the speed bound gives $D_0(t)\le V_Ct$.  For $k\ge1$, split
		\[
		\mathcal V_{s_r^k}[X_r^{k+1}]-\mathcal V_{s_r^{k-1}}[X_r^k]
		=\bigl(\mathcal V_{s_r^k}[X_r^{k+1}]-\mathcal V_{s_r^k}[X_r^k]\bigr)
		+\bigl(\mathcal V_{s_r^k}[X_r^k]-\mathcal V_{s_r^{k-1}}[X_r^k]\bigr).
		\]
		Both $X_r^{k+1}$ and $X_r^k$ are $\tfrac12$-co-Lipschitz, so the first term is at most $2Q_\lambda D_k(r)$.  By \eqref{eq:Coulomb-score-H}, \eqref{eq:score-H-W1}, and the label coupling, the second term is at most $\abs{c_C}c_1\Lambda_{S,\lambda}D_{k-1}(r)=\beta_CD_{k-1}(r)$.  Hence
		\[
		D_k(t)\le L_{F,C}\int_0^tD_k(r)\dd r
		+\beta_C\int_0^tD_{k-1}(r)\dd r .
		\]
		This is the recursion in the proof of \cref{thm:main-outer} with $(V_\lambda,L_{F,\lambda},\beta_\lambda)$ replaced by $(V_C,L_{F,C},\beta_C)$; the same Gronwall and repeated-integration argument proves \eqref{eq:Coulomb-factorial}.
		
		\emph{Step 4: The limit.}  By \eqref{eq:Coulomb-factorial}, $(X^k)$ is Cauchy in $C([0,T_0];C(D;\R^3))$; its limit $X$ lies in the closed set $\mathcal X_{T_0}$.  Put $s_t:=S_\lambda((X_t)_\#\rho_0)$.  By \eqref{eq:score-H-W1} and the label coupling, $\norm{s_t^k-s_t}_{\Hh}\le\Lambda_{S,\lambda}\norm{X_t^k-X_t}_{0,D}$, so \eqref{eq:Coulomb-flow-stability} and \eqref{eq:Coulomb-score-H} give $\mathcal V_{s_t^k}[X_t^{k+1}]\to\mathcal V_{s_t}[X_t]$ uniformly in $(t,a)$. Passing to the limit in the integral form of \eqref{eq:Coulomb-waveform-system} shows that $X$ solves \eqref{eq:Coulomb-fixed-flow}.  If $X,Y\in\mathcal X_{T_0}$ both solve \eqref{eq:Coulomb-fixed-flow}, the same two estimates give
		\[
		\norm{X_t-Y_t}_{0,D}\le(L_{F,C}+\beta_C)\int_0^t\norm{X_r-Y_r}_{0,D}\dd r,
		\]
		and Gronwall's inequality gives $X=Y$.
		
		\emph{Step 5: Arbitrary finite intervals.}  Assume that every system step exists on $[0,T]$ and that \eqref{eq:Coulomb-global-colipschitz} holds.  Then \cref{lem:integrated-coulomb-stability} applies with $m=m_T$, so Steps~2--4 hold on $[0,T]$ with $L_{F,C}$ given by \eqref{eq:Coulomb-conditional-LF}; no small-time restriction is needed for the Volterra argument itself.  The lower bound \eqref{eq:Coulomb-global-colipschitz} is preserved under uniform limits, so the limit satisfies it and the uniqueness argument of Step~4 applies among continuous flows with the same lower bound.  Moreover, \eqref{eq:Coulomb-spatial-velocity} and Gronwall's inequality give the common upper bound
		\[
		\abs{X_t^k(a)-X_t^k(b)}
		\le \exp\!\left(\abs{c_C}
		\left(H_\lambda+\frac{2L_\lambda J_0}{m_T}\right)t\right)\abs{a-b},
		\]
		so the limit flow is bi-Lipschitz.
	\end{proof}
	
	\begin{remark}[Scope of the Coulomb result]
		\label{rem:coulomb-next-step}
		\Cref{thm:continuum-Coulomb-waveform} uses the averaging of an absolutely continuous density and does not extend to empirical measures: for atoms, the direction-dependent diagonal limit in \eqref{eq:coulomb-directional-limit} remains.  It therefore proves neither a three-dimensional Coulomb particle theorem nor an $N\to\infty$ limit.  By contrast, the very-soft-potential case $\gamma\in(-3,-2)$ of \cite[Theorem~1.3]{CarrilloDelgadinoWu2023} is a particle theorem for the mollified score, proved by a $W_\infty$ argument under an initial minimal-separation condition, up to the maximal existence time of the regularized continuum solution; we do not obtain a Tikhonov analogue.  Extending the local interval to arbitrary $T$ also requires a uniform lower bi-Lipschitz bound, or an equivalent Riesz-potential estimate, that is not supplied by $f_0\in L_c^\infty$ alone.  Finally, all constants retain their poor dependence on $\lambda$, so this characteristic theorem cannot be passed to $\lambda\downarrow0$.  The conditional reverse-entropy result in \cref{cor:density-level-zero-lambda-KL} gives a different density-level route, but requires uniform full-family estimates not established here.
	\end{remark}
	
	\section{Connection to the Landau equation}
	\label{sec:landau-connection}
	
	\subsection{The regularized limit equation}
	
	The limit of the exact iteration solves \eqref{eq:main-coupled-particle}, and passing to the limit as the number of particles tends to infinity yields  \eqref{eq:main-mean-field}. Since this dynamics retains the RKHS score regularization at fixed $\lambda>0$, understanding its relation to the Landau equation \eqref{eq:intro-draft-landau-cauchy} requires examining what happens as the regularization is removed. As a first step, we show below that, under suitable assumptions and for a fixed density, the regularized score approaches the true score as $\lambda\downarrow0$; this provides a starting point for connecting the two dynamics, but does not by itself establish convergence of the time-dependent solutions.
	
	\begin{proposition}\label{prop:static-lambda-bias}
		Let $f>0$ be a $C^1$ probability density with finite Fisher information, $f\dd v\in\Pone(\R^d)$, and suppose $\nabla\log f\in\Hh$.  Then $f\in W^{1,1}(\R^d)$, integration by parts is valid for every $s\in\Hh$, and
		\begin{equation}\label{eq:static-lambda-bias}
			\int\abs{S_\lambda(f\dd v)-\nabla\log f}^2f\dd v
			+\lambda\norm{S_\lambda(f\dd v)}_{\Hh}^2
			\le\lambda\norm{\nabla\log f}_{\Hh}^2.
		\end{equation}
	\end{proposition}
	
	\begin{proof}
		By \cref{ass:score-space}, every $s\in\Hh$ is bounded with bounded divergence.  Hence \cref{lem:cutoff-ibp} gives $\nabla f\in L^1$, so $f\in W^{1,1}(\R^d)$, and the hypotheses \eqref{eq:density-level-identity-hypotheses} of \cref{prop:density-level-identity} hold for every $s\in\Hh$.  Applying \eqref{eq:density-level-identity} to the definition of $\J_{f\dd v}^\lambda$ gives, for every $s\in\Hh$,
		\begin{equation}\label{eq:density-level-objective}
			\J_{f\dd v}^\lambda(s)
			=\norm{s-\nabla\log f}_{L^2(f)}^2-\norm{\nabla\log f}_{L^2(f)}^2
			+\lambda\norm s_{\Hh}^2 .
		\end{equation}
		Since $\nabla\log f\in\Hh$ and $S_\lambda(f\dd v)$ minimizes $\J_{f\dd v}^\lambda$ over $\Hh$,
		\[
		\J_{f\dd v}^\lambda\bigl(S_\lambda(f\dd v)\bigr)
		\le\J_{f\dd v}^\lambda(\nabla\log f)
		=-\norm{\nabla\log f}_{L^2(f)}^2+\lambda\norm{\nabla\log f}_{\Hh}^2 .
		\]
		Inserting \eqref{eq:density-level-objective} with $s=S_\lambda(f\dd v)$ on the left-hand side and cancelling the common term $-\norm{\nabla\log f}_{L^2(f)}^2$ proves \eqref{eq:static-lambda-bias}.
	\end{proof}
	
	\begin{remark}[Exclusion of Maxwellian scores]\label{rem:maxwellian-exclusion}
		The assumption $\Hh\hookrightarrow C_b^2$ excludes an unbounded affine score such as the exact score of a Maxwellian on $\R^d$, so \cref{prop:static-lambda-bias} does not apply to a Gaussian or BKW benchmark. More is true for score spaces contained in $L^2(\R^d;\R^d)$ with $d\ge3$, such as the Gaussian RKHS: by the argument of \cref{rem:density-level-lambda-conditional}, no strictly positive $C^1$ probability density has its score in such a space, so \cref{prop:static-lambda-bias} applies there to no density at all. Spaces satisfying \cref{ass:score-space} to which it does apply exist, for instance the sum of the Gaussian RKHS and the span of a fixed field $\varphi\in C_b^2(\R^d;\R^d)$, with kernel $K_\sigma+\varphi\otimes\varphi$; for $\varphi(v)=v/\langle v\rangle$ it contains the score of $g\propto e^{-\langle v\rangle}$. For the following approximation statement, assume that $f>0$ is a $C^1$ probability density with $f\dd v\in\Pone(\R^d)$ and finite Fisher information. If $\nabla\log f\notin\Hh$ and $\Hh$ is dense in $L^2(f\dd v;\R^d)$, as is the case for the Gaussian RKHS and every strictly positive $C^1$ density, then \eqref{eq:density-level-objective} still gives $\norm{S_\lambda(f\dd v)-\nabla\log f}_{L^2(f)}\to0$ by density, but without a rate and with $\norm{S_\lambda(f\dd v)}_{\Hh}\to\infty$ (a bounded subsequence would converge weakly in $\Hh$, hence pointwise, to $\nabla\log f$, forcing $\nabla\log f\in\Hh$); if instead $\nabla\log f\in\Hh$, \cref{prop:static-lambda-bias} gives $\norm{S_\lambda(f\dd v)}_{\Hh}\le\norm{\nabla\log f}_{\Hh}$.  Density alone provides no quantitative convergence rate. Such rates may nevertheless be obtained within the same RKHS from approximation estimates that balance the $L^2(f)$ error against the RKHS norm; enlarging the score space is therefore not necessary in general.
	\end{remark}
	
	\subsection{A conditional zero-regularization limit via relative entropy}
	
	By \cref{thm:empirical-lambda-zero-dichotomy}, an RKHS limit as $\lambda\downarrow0$ cannot be expected for empirical measures in general, whereas for densities the integration-by-parts identity gives a route to vanishing score errors. We therefore treat the time evolution of densities separately from the fixed-$N$ theory, using relative entropy with the learned density in the first argument.  This orientation would tolerate a learned flow with smaller support, but the argument below assumes that both densities are strictly positive, together with finite Fisher information and the needed approximation properties.
	
	For this subsection let $d=3$, $c_C>0$, and let $A_C$ be given by \eqref{eq:Coulomb-matrix}.  For a nonnegative matrix $M$, write $\abs{z}_M^2=z^\top Mz$.
	
	\begin{assumption}[Hypotheses for the relative-entropy estimate]
		\label{ass:reverse-entropy-admissibility}
		Fix $T<\infty$.  Let $f$ be a strictly positive classical probability solution of the Coulomb Landau equation on $[0,T]$, with finite energy and entropy.  For some $\kappa,\nu\ge1$ and $\zeta\ge0$, assume
		\begin{align}
			K_1(t)&:=\norm{\langle v\rangle^{-\kappa}\nabla\log f_t}_{L^\infty},
			&K_1&\in L^2(0,T),\label{eq:reverse-K1}\\
			K_2(t)&:=\norm{\langle v\rangle^{-\zeta}D^2\log f_t}_{L^\infty},
			&K_2&\in L^2(0,T),\label{eq:reverse-K2}\\
			K_0(t)&:=\norm{\langle v\rangle^{-\nu}\partial_t\log f_t}_{L^\infty},
			&K_0&\in L^1(0,T).
			\label{eq:reverse-K0}
		\end{align}
		Here $\langle v\rangle=(1+\abs v^2)^{1/2}$.  Put
		\begin{equation}\label{eq:reverse-moment-order}
			m_*:=3+\max\{2\kappa,2\zeta,\nu\}.
		\end{equation}
		Let $g$ solve the score-driven equation
		\begin{equation}\label{eq:forced-score-flow}
			\partial_tg+\nabla\cdot(gU[g,s])=0,
			\qquad
			U[g,s](v):=-\int_{\R^3}A_C(v-w)(s(v)-s(w))g(w)\dd w.
		\end{equation}
		Assume that $g$ is a probability density with weak score $q_g:=\nabla g/g$ on $\{g>0\}$ and that $\KLdiv(g_0\Vert f_0)<\infty$.  Assume also that
		\begin{align}
			\sup_{0\le t\le T}\int_{\R^3}\langle v\rangle^{m_*}(f_t+g_t)\dd v
			&<\infty,\label{eq:reverse-moment-bound}\\
			\sup_{0\le t\le T}\left[
			\int\abs v^2g_t(v)\dd v+
			\abs{\int g_t\log g_t\dd v}\right]&<\infty.
			\label{eq:reverse-entropy-bound}
		\end{align}
		Finally, assume that $g$ is strictly positive, that $f$ and $g$ are $C^1$ in $t$ and $C^2$ in $v$, that $s$ is $C^1$ in $v$ and measurable in $t$, with $s$ and $Ds$ bounded on $[0,T]\times\R^3$, and that
		\begin{equation}\label{eq:reverse-Linfty-Fisher}
			\sup_{0\le t\le T}\bigl(\norm{f_t}_{L^\infty}+\norm{g_t}_{L^\infty}\bigr)<\infty,
			\qquad
			\int_0^T\!\!\int_{\R^3}g_t\abs{\nabla\log g_t}^2\dd v\dd t<\infty .
		\end{equation}
	\end{assumption}
	
	\begin{remark}\label{rem:integrability-justification}
		Put $Q_\infty:=\sup_{0\le t\le T}(\norm{f_t}_{L^\infty}+\norm{g_t}_{L^\infty})$, finite by \eqref{eq:reverse-Linfty-Fisher}, and write $\norm{Ds_t}_\infty$ for the supremum of the operator norm of $Ds_t$. Under \cref{ass:reverse-entropy-admissibility} the following bounds hold; they are all that is used to justify the formal computations in \cref{lem:reverse-entropy-slack,thm:forced-reverse-KL} and in Appendix~\ref{sm:tabary-proofs}.
		\begin{enumerate}[label=(\alph*)]
			\item (\emph{Riesz potentials and velocity fields}) By
			\cref{lem:Coulomb-weighted-loss}, the Riesz potentials of $f_t$ and $g_t$, that is, $\int\abs{v-w}^{-1}f_t(w)\dd w$ and $\int\abs{v-w}^{-1}g_t(w)\dd w$, are bounded by $C_RQ_\infty^{1/3}$ uniformly in $v$ and $t$.  Since $\norm{A_C(z)}_{\rm op}=c_C/\abs z$, $\abs{\nabla\cdot A_C(z)}=2c_C/\abs z^2$, and $\abs{s_t(v)-s_t(w)}\le\norm{Ds_t}_\infty\abs{v-w}$, one has
			\[
			\abs{U[g_t,s_t]}\le c_C\norm{Ds_t}_\infty,
			\qquad
			\abs{\nabla\cdot U[g_t,s_t]}\le4c_C\norm{Ds_t}_\infty C_RQ_\infty^{1/3}.
			\]
			By the mean value theorem, $\abs{U[f_t,\nabla\log f_t](v)}\le c_CK_2(t)(\langle v\rangle^\zeta+M_*)$, and $\abs{\log f_t(v)}\le\abs{\log f_t(0)}+K_1(t)\langle v\rangle^{\kappa+1}$.
			\item (\emph{the good term}) With $r_t=\log(g_t/f_t)$ and
			$q_{g_t}=\nabla\log g_t$, since $\abs{\nabla r_t}^2\le2\abs{q_{g_t}}^2+2K_1(t)^2\langle v\rangle^{2\kappa}$,
			\begin{equation}\label{eq:good-term-L1}
				\mathcal G_t\le2c_CC_RQ_\infty^{1/3}\int_{\R^3}g_t\abs{\nabla r_t}^2\dd v
				\le4c_CC_RQ_\infty^{1/3}\Bigl(\int_{\R^3}g_t\abs{q_{g_t}}^2\dd v
				+K_1(t)^2M_*\Bigr),
			\end{equation}
			so $\mathcal G\in L^1(0,T)$ by \eqref{eq:reverse-Linfty-Fisher} and $K_1\in L^2(0,T)$; likewise
			\[
			\mathcal L_A(s_t,g_t)\le c_CC_RQ_\infty^{1/3}\norm{s_t-q_{g_t}}_{L^2(g_t)}^2
			\]
			by \eqref{eq:Coulomb-weighted-L2}, so $\mathcal L_A(s_\cdot,g_\cdot)\in L^1(0,T)$ as well.  The Fisher-information condition in \eqref{eq:reverse-Linfty-Fisher} is used only here and in (d) below; the same bound justifies, by Fubini's theorem, the symmetrizations in $(v,w)$ in \eqref{eq:forced-extra-term} and in Step~1 of the proof of \cref{lem:reverse-entropy-slack}.
			\item (\emph{square integrability of $r_t$}) Since $x\mapsto x(\log x)^2$
			is increasing on $(0,e^{-2})$, splitting according to whether $g_t\ge e^{-\abs v^2}$ gives $\int g_t(\log^-g_t)^2\dd v\le\int\abs v^4g_t\dd v+C_0$ with an absolute constant $C_0$, while $\log^+g_t\le\log^+Q_\infty$ and $\int g_t(\log f_t)^2\dd v\le2\abs{\log f_t(0)}^2+2K_1(t)^2\int g_t\langle v\rangle^{2\kappa+2}\dd v$. Hence
			\begin{equation}\label{eq:r-square-integrable}
				\int_{\R^3}g_tr_t^2\dd v\le C\bigl(1+K_1(t)^2\bigr),
			\end{equation}
			with $C$ depending only on $Q_\infty$, $M_*$, and $\sup_{0\le t\le T}\abs{\log f_t(0)}$, which is finite because $f$ is continuous in $t$ and strictly positive.
			\item (\emph{absolute continuity of the relative entropy}) Writing
			$\partial_tg_t=-g_t\bigl(q_{g_t}\cdot U[g_t,s_t]+\nabla\cdot U[g_t,s_t]\bigr)$ and using $\abs{\partial_t\log f_t}\le K_0(t)\langle v\rangle^\nu$, (a)--(c) give
			\begin{align}
				&\int_0^T\!\!\int_{\R^3}\abs{\partial_t\bigl(g_t\log(g_t/f_t)\bigr)}\dd v\dd t
				\notag\\
				&\quad\le\int_0^T\Bigl[c_C\norm{Ds_t}_\infty
				\Bigl(\int g_t\abs{q_{g_t}}^2\dd v\Bigr)^{1/2}
				\Bigl(\int g_t(r_t+1)^2\dd v\Bigr)^{1/2}
				\notag\\
				&\qquad\qquad
				+4c_C\norm{Ds_t}_\infty C_RQ_\infty^{1/3}\int g_t\abs{r_t+1}\dd v
				+K_0(t)M_*\Bigr]\dd t<\infty,
				\label{eq:KL-abs-cont}
			\end{align}
			so $t\mapsto\KLdiv(g_t\Vert f_t)$ is absolutely continuous on $[0,T]$, with the derivative computed in \cref{thm:forced-reverse-KL} for almost every $t$.
		\end{enumerate}
		Consequently every integrand in the proofs of \cref{lem:reverse-entropy-slack,thm:forced-reverse-KL} is dominated by a function integrable in $(v,w,t)$, and, in each integration by parts, the relevant vector field and its distributional divergence with respect to the variables being differentiated are integrable. The cutoff argument of \cref{lem:cutoff-ibp} therefore shows that the integral of this divergence vanishes for almost every time: the fields on $\R^6$ in Step~1 of the proof of \cref{lem:reverse-entropy-slack} are not $C^1$ across the diagonal $v=w$, and their divergences contain the factor $\abs{g_t-f_t}(v)\abs{g_t-f_t}(w)/\abs{v-w}$, which is why the Riesz potentials of both $f_t$ and $g_t$ are needed. The moment order $m_*$ in \eqref{eq:reverse-moment-order} suffices, since the exponents that occur are at most $2\kappa+2$, $3+2\zeta$, $\kappa+\zeta$, $\nu$, and $4$.
	\end{remark}
	
	\begin{remark}\label{rem:weaker-classes}
		The estimates below extend to weaker solution classes whenever a strictly positive smooth approximation preserves the relative-entropy identity, the moment, energy, and entropy bounds, the coercive estimate, and the forcing term, and passes to the limit with lower semicontinuity of the relative entropy. We do not pursue this here; in particular, it is not established for the compact-support flow of \cref{thm:continuum-Coulomb-waveform}, whose density is not strictly positive.
	\end{remark}
	
	After relabeling the weak density of \cite{TabaryWeakStrong2025} as $g$ and its strong density as $f$, the unforced good--bad decomposition is its Lemma~2.4, the coercivity of the mollified kernel is its Lemma~2.8, and the good- and bad-term estimates below are the free-Young-parameter versions of its Lemmas~2.9 and~2.11. We do not claim those estimates as new; complete proofs of \cref{lem:weighted-pinsker,lem:cutoff-coercivity,lem:reverse-entropy-slack}, which fix every constant in the present notation, are given in Appendix~\ref{sm:tabary-proofs}. Our additional step is to insert the forcing $s-\nabla\log g$, control it by the weighted score loss, and deduce the conditional density-level consequence.
	
	Throughout, we write
	\begin{equation}\label{eq:reverse-bound-constants}
		\begin{gathered}
			M_*:=\sup_{0\le t\le T}\int_{\R^3}\langle v\rangle^{m_*}(f_t+g_t)\dd v,\\
			E:=\sup_{0\le t\le T}\int_{\R^3}\abs v^2g_t\dd v,
			\qquad
			\mathsf H:=\sup_{0\le t\le T}\abs{\int_{\R^3}g_t\log g_t\dd v},
		\end{gathered}
	\end{equation}
	all finite by \eqref{eq:reverse-moment-bound}--\eqref{eq:reverse-entropy-bound}; note $M_*\ge2$.
	
	\begin{lemma}\label{lem:weighted-pinsker}
		Let $f>0$ and $g$ be probability densities on $\R^d$ and let $\varphi\ge0$ be measurable.  Then
		\begin{equation}\label{eq:weighted-pinsker}
			\int_{\R^d}\abs{g-f}\varphi\dd v
			\le\sqrt{2\KLdiv(g\Vert f)}
			\left(\int_{\R^d}(g+f)\varphi^2\dd v\right)^{1/2}.
		\end{equation}
	\end{lemma}

	\begin{lemma}\label{lem:cutoff-coercivity}
		Let
		\begin{equation}\label{eq:cutoff-coulomb-kernel}
			\widetilde A_C(z):=c_C\langle z\rangle^{-3}\bigl(\abs z^2I-z\otimes z\bigr),
			\qquad z\in\R^3,
		\end{equation}
		so that $0\le\widetilde A_C(z)\le A_C(z)$ for $z\ne0$, $\norm{\widetilde A_C(z)}_{\rm op}\le c_C$, and $\nabla\cdot\widetilde A_C(z)=-2c_C\langle z\rangle^{-3}z$ (columnwise divergence), hence $\abs{\nabla\cdot\widetilde A_C}\le2c_C$.  Let $g$ be a probability density on $\R^3$ with $\int\abs w^2g\dd w\le E$ and $\abs{\int g\log g\dd w}\le\mathsf H$. Then there is $c_*=c_*(c_C,E,\mathsf H)>0$ such that
		\begin{equation}\label{eq:cutoff-coercivity}
			(\widetilde A_C*g)(v)\succeq c_*\langle v\rangle^{-3}I
			\qquad\text{for all }v\in\R^3 .
		\end{equation}
	\end{lemma}

	The lemma is uniform over families with common bounds $E$ and $\mathsf H$; this is used in \cref{cor:density-level-zero-lambda-KL}.
	
	\begin{lemma}
		\label{lem:reverse-entropy-slack}
		Under \cref{ass:reverse-entropy-admissibility}, put $r_t=\log(g_t/f_t)$, $h_t=g_t/f_t-1$,
		\begin{equation}\label{eq:relative-good-term}
			\begin{gathered}
				\mathcal G_t:=\frac12\iint_{\R^6}
				\abs{\nabla r_t(v)-\nabla r_t(w)}_{A_C(v-w)}^2
				g_t(v)g_t(w)\dd v\dd w,\\
				\mathcal I_t:=\int_{\R^3}\langle v\rangle^{-3}\abs{\nabla r_t(v)}^2g_t(v)\dd v,
			\end{gathered}
		\end{equation}
		so that $\mathcal G\in L^1(0,T)$ by \cref{rem:integrability-justification}(b). Let $\mathcal B_t$ be the tensorization-defect term in the Landau weak--strong identity,
		\begin{multline}\label{eq:relative-bad-term}
			\mathcal B_t:=\frac12\iint_{\R^6}
			\frac{(g_t(v)-f_t(v))(g_t(w)-f_t(w))}{f_t(v)f_t(w)}\\
			{}\times(\nabla_v-\nabla_w)\cdot
			\left[A_C(v-w)(\nabla_v-\nabla_w)(f_t(v)f_t(w))\right]
			\dd v\dd w,
		\end{multline}
		and let $c_*=c_*(c_C,E,\mathsf H)$ be the constant of \cref{lem:cutoff-coercivity} with $E,\mathsf H$ from \eqref{eq:reverse-bound-constants}.  Then, for almost every $t\in[0,T]$:
		\begin{enumerate}[label=(\roman*)]
			\item (\emph{decomposition})
			\begin{equation}\label{eq:landau-part-decomposition}
				\int_{\R^3}g_t\bigl(U[g_t,\nabla\log g_t]-U[f_t,\nabla\log f_t]\bigr)
				\cdot\nabla r_t\dd v=-\mathcal G_t+\mathcal B_t,
			\end{equation}
			where $\mathcal B_t$ also has the once-integrated form
			\begin{multline}\label{eq:relative-bad-term-once}
				\mathcal B_t=-\iint_{\R^6}f_t(v)f_t(w)\,h_t(w)\,
				\nabla h_t(v)^\top A_C(v-w)\\
				\times\bigl(\nabla\log f_t(v)-\nabla\log f_t(w)\bigr)\dd v\dd w ;
			\end{multline}
			\item (\emph{good term})
			\begin{equation}\label{eq:relative-good-lower}
				\mathcal G_t\ge\frac{c_*}2\mathcal I_t-C_G(t)\KLdiv(g_t\Vert f_t),
				\qquad
				C_G(t):=\frac{c_C^2\bigl(2+K_1(t)\bigr)^2M_*^2}{c_*};
			\end{equation}
			\item (\emph{bad term}) for every $\delta>0$,
			\begin{equation}\label{eq:relative-bad-young}
				\abs{\mathcal B_t}\le\delta\,\mathcal I_t
				+\frac{C_B(t)}{\delta}\KLdiv(g_t\Vert f_t),
				\qquad
				C_B(t):=2c_C^2K_2(t)^2M_*^2 ;
			\end{equation}
			\item (\emph{coercive slack}) for every $\theta\in(0,1)$,
			\begin{equation}\label{eq:relative-bad-slack}
				\mathcal B_t\le\theta\,\mathcal G_t+C_\theta(t)\KLdiv(g_t\Vert f_t),
				\qquad
				C_\theta(t):=\theta C_G(t)+\frac{2C_B(t)}{\theta c_*}.
			\end{equation}
		\end{enumerate}
		Since $K_1,K_2\in L^2(0,T)$, the functions $C_G$, $C_B$, and $C_\theta$ belong to $L^1(0,T)$; they are uniform over families for which the bounds in \cref{ass:reverse-entropy-admissibility} are uniform.  The only moments used in the proof are those of orders $2\kappa$, $3$, and $3+2\zeta$; together with those listed in \cref{rem:integrability-justification}, all are at most $m_*$, which is why \eqref{eq:reverse-moment-order} suffices.
	\end{lemma}
	The proofs of \cref{lem:weighted-pinsker,lem:cutoff-coercivity,lem:reverse-entropy-slack} are given in Appendix~\ref{sm:tabary-proofs}.
	
	The following theorem is the weak--strong estimate of \cite{TabaryWeakStrong2025} with the score-forcing term of \cite[Theorem~4.2]{Ilin2025} inserted in the reverse orientation, that is, with the learned density in the first argument. The new element is the reduction of the forcing to the weighted score loss $\mathcal L_A$ without a bound on $f/g$.

	\begin{theorem}[Relative-entropy stability with score forcing]
		\label{thm:forced-reverse-KL}
		Adopt \cref{ass:reverse-entropy-admissibility}.  Let $e_t=s_t-q_{g_t}$ and define
		\begin{align}
			\mathcal E_t&:=\frac12\iint_{\R^6}
			\abs{e_t(v)-e_t(w)}_{A_C(v-w)}^2g_t(v)g_t(w)\dd v\dd w,
			\label{eq:forced-pair-error}\\
			\mathcal L_A(s_t,g_t)&:=\int_{\R^3}
			e_t(v)^\top(A_C*g_t)(v)e_t(v)g_t(v)\dd v.
			\label{eq:weighted-score-loss}
		\end{align}
		Then
		\begin{equation}\label{eq:pair-to-weighted-loss}
			\mathcal E_t\le2\mathcal L_A(s_t,g_t).
		\end{equation}
		Both $\mathcal L_A(s_t,g_t)$ and $\mathcal G_t$ belong to $L^1(0,T)$ by \cref{rem:integrability-justification}(b). Moreover, for a nonnegative $C\in L^1(0,T)$ depending only on the admissibility bounds,
		\begin{equation}\label{eq:forced-reverse-KL-differential}
			\frac{\dd}{\dd t}\KLdiv(g_t\Vert f_t)
			\le C(t)\KLdiv(g_t\Vert f_t)+\mathcal L_A(s_t,g_t)
		\end{equation}
		for almost every $t\in(0,T)$.  Consequently, for every $t\in[0,T]$,
		\begin{align}
			\KLdiv(g_t\Vert f_t)
			&\le e^{\int_0^tC(r)\dd r}\KLdiv(g_0\Vert f_0)\nonumber\\
			&\quad+\int_0^t e^{\int_s^tC(r)\dd r}
			\mathcal L_A(s_s,g_s)\dd s.
			\label{eq:forced-reverse-KL-gronwall}
		\end{align}
	\end{theorem}
	
	\begin{proof}
		\emph{Step 1: Entropy identity.}  By \cref{ass:reverse-entropy-admissibility} and \cref{rem:integrability-justification}, mass conservation and the two continuity equations give, with $r_t=\log(g_t/f_t)$,
		\begin{align*}
			\frac{\dd}{\dd t}\KLdiv(g_t\Vert f_t)
			&=\int\partial_tg_t\,(r_t+1)\dd v-\int g_t\,\partial_t\log f_t\dd v\\
			&=\int g_tU[g_t,s_t]\cdot\nabla r_t\dd v
			-\int g_tU[f_t,\nabla\log f_t]\cdot\nabla r_t\dd v ,
		\end{align*}
		where the last term uses $\int g\,\partial_tf/f=-\int(g/f)\nabla\cdot(fU_f)=\int fU_f\cdot\nabla(g/f) =\int gU_f\cdot\nabla r$; the term $\int g_t\partial_t\log f_t$ is finite because $\abs{\partial_t\log f_t}\le K_0(t)\langle v\rangle^\nu$ and $\nu\le m_*$. This is
		\begin{equation}\label{eq:relative-entropy-continuity}
			\frac{\dd}{\dd t}\KLdiv(g_t\Vert f_t)
			=\int g_t\bigl(U[g_t,s_t]-U[f_t,\nabla\log f_t]\bigr)
			\cdot\nabla\log\frac{g_t}{f_t}\dd v.
		\end{equation}
		
		\emph{Step 2: Decomposition into Landau part and forcing.}  Since $s=q_g+e$ and $U$ is linear in the score, $U[g,s]=U[g,q_g]+U[g,e]$.  By \cref{lem:reverse-entropy-slack}(i) the exact Landau part equals $-\mathcal G_t+\mathcal B_t$, while the score forcing contributes, after symmetrization in $(v,w)$,
		\begin{multline}\label{eq:forced-extra-term}
			\mathcal R_t:=\int g_tU[g_t,e_t]\cdot\nabla r_t\dd v
			=-\frac12\iint_{\R^6}
			(e_t(v)-e_t(w))^\top A_C(v-w)\\
			{}\times(\nabla r_t(v)-\nabla r_t(w))g_t(v)g_t(w)\dd v\dd w.
		\end{multline}
		Hence
		\begin{equation}\label{eq:forced-entropy-decomposition}
			\frac{\dd}{\dd t}\KLdiv(g_t\Vert f_t)=-\mathcal G_t+\mathcal B_t+\mathcal R_t .
		\end{equation}
		
		\emph{Step 3: The forcing term.}  Cauchy--Schwarz in the semidefinite $A_C$ metric, first pointwise and then in $L^2(g_t\otimes g_t)$, gives $\abs{\mathcal R_t}\le\sqrt{\mathcal E_t\mathcal G_t} \le\frac12\mathcal E_t+\frac12\mathcal G_t$.  Moreover $\abs{e(v)-e(w)}^2_{A_C(v-w)}\le2\abs{e(v)}^2_{A_C(v-w)}+2\abs{e(w)}^2_{A_C(v-w)}$; integrating against $g(v)g(w)$ and using the symmetry of the two resulting terms gives \eqref{eq:pair-to-weighted-loss}.
		
		\emph{Step 4: Absorption.}  Apply \cref{lem:reverse-entropy-slack}(iv) with $\theta=\frac12$ in \eqref{eq:forced-entropy-decomposition}:
		\[
		\begin{aligned}
			\frac{\dd}{\dd t}\KLdiv(g_t\Vert f_t)
			&\le-\mathcal G_t+\Bigl(\frac12\mathcal G_t+C_{1/2}(t)\KLdiv(g_t\Vert f_t)\Bigr)
			+\frac12\mathcal E_t+\frac12\mathcal G_t\\
			&=C_{1/2}(t)\KLdiv(g_t\Vert f_t)+\frac12\mathcal E_t ,
		\end{aligned}
		\]
		and \eqref{eq:pair-to-weighted-loss} turns $\frac12\mathcal E_t$ into $\mathcal L_A(s_t,g_t)$.  This is \eqref{eq:forced-reverse-KL-differential} with the explicit integrable coefficient $C(t):=C_{1/2}(t)=\frac12C_G(t)+4C_B(t)/c_*$.  Gronwall's inequality proves \eqref{eq:forced-reverse-KL-gronwall}.  The imported unforced ingredients are Steps~1--2 together with \cref{lem:reverse-entropy-slack}; the forcing estimate of Step~3, which replaces the factor $\norm{f/g}_\infty$ of \cite[Theorem~4.2]{Ilin2025} by absorption into $\mathcal G_t$, and the reduction to $\mathcal L_A$ are the only new steps.
	\end{proof}
	
	\begin{lemma}\label{lem:Coulomb-weighted-loss}
		Let $g$ be a probability density in $L^\infty(\R^3)$ and put $Q=\norm g_\infty$.  Then
		\begin{equation}\label{eq:Coulomb-Riesz-Linfty}
			\sup_{v\in\R^3}\int_{\R^3}\frac{g(w)}{\abs{v-w}}\dd w
			\le C_RQ^{1/3},
			\qquad C_R:=\frac32(4\pi)^{1/3}.
		\end{equation}
		If, in addition, $g$ admits a weak score $q_g$ and $s-q_g\in L^2(g\dd v;\R^3)$, then
		\begin{equation}\label{eq:Coulomb-weighted-L2}
			\mathcal L_A(s,g)
			\le c_CC_RQ^{1/3}\norm{s-q_g}_{L^2(g)}^2.
		\end{equation}
	\end{lemma}
	
	\begin{proof}
		For every $R>0$, splitting the integral at $\abs{v-w}=R$ and using $\int g=1$,
		\[
		\int\frac{g(w)}{\abs{v-w}}\dd w
		\le Q\int_{\abs z\le R}\frac{\dd z}{\abs z}+\frac1R
		=2\pi QR^2+\frac1R.
		\]
		The choice $R=(4\pi Q)^{-1/3}$ gives $2\pi QR^2+R^{-1}=\bigl(\tfrac12+1\bigr)(4\pi Q)^{1/3}=C_RQ^{1/3}$, which proves \eqref{eq:Coulomb-Riesz-Linfty}.  Since the projection in $A_C$ has operator norm one, $(A_C*g)(v)\preceq c_C\bigl(\int g(w)\abs{v-w}^{-1}\dd w\bigr)I \preceq c_CC_RQ^{1/3}I$, which proves \eqref{eq:Coulomb-weighted-L2}.
	\end{proof}
	
	For a density with a weak score $q_g\in L^2(g)$ and valid integration by parts, define the density-level approximation modulus
	\begin{equation}\label{eq:density-level-approximation-modulus}
		\mathfrak b_\lambda(g):=\inf_{h\in\Hh}
		\left\{\norm{h-q_g}_{L^2(g)}^2+\lambda\norm h_{\Hh}^2\right\}.
	\end{equation}
	By \eqref{eq:density-level-objective} with $f$ replaced by $g$, $\J_{g\dd v}^\lambda(h)$ and the bracket in \eqref{eq:density-level-approximation-modulus} differ by the constant $\norm{q_g}_{L^2(g)}^2$, so they have the same minimizer $S_\lambda(g\dd v)$ over $\Hh$.  Hence
	\begin{equation}\label{eq:density-level-modulus-score-error}
		\norm{S_\lambda(g\dd v)-q_g}_{L^2(g)}^2
		+\lambda\norm{S_\lambda(g\dd v)}_{\Hh}^2
		=\mathfrak b_\lambda(g).
	\end{equation}
	
	\begin{corollary}[Conditional zero-regularization limit for densities]
		\label{cor:density-level-zero-lambda-KL}
		For every $\lambda\in(0,1]$, let $g^\lambda$ solve
		\begin{equation}\label{eq:density-level-regularized-Coulomb}
			\partial_tg^\lambda+\nabla\cdot
			\left(g^\lambda U[g^\lambda,S_\lambda(g^\lambda\dd v)]\right)=0
		\end{equation}
		on $[0,T]$.  Assume that the pairs $(g^\lambda,f)$ satisfy \cref{ass:reverse-entropy-admissibility} with bounds uniform in $\lambda$, and write $Q_*$ for a uniform bound on $\norm{g_t^\lambda}_\infty$ in \eqref{eq:reverse-Linfty-Fisher}. Put $\mathfrak B_T(\lambda):=\int_0^T\mathfrak b_\lambda(g_t^\lambda)\dd t$. Then, for $C_T<\infty$ independent of $\lambda$,
		\begin{equation}\label{eq:density-level-modulus-KL}
			\sup_{0\le t\le T}\KLdiv(g_t^\lambda\Vert f_t)
			\le C_T\left[
			\KLdiv(g_0^\lambda\Vert f_0)
			+c_CC_RQ_*^{1/3}\mathfrak B_T(\lambda)\right].
		\end{equation}
		If $\KLdiv(g_0^\lambda\Vert f_0)\to0$ and $\mathfrak B_T(\lambda)\to0$, then
		\begin{equation}\label{eq:density-level-modulus-convergence}
			\sup_{0\le t\le T}\KLdiv(g_t^\lambda\Vert f_t)\longrightarrow0,
			\qquad
			\sup_{0\le t\le T}\norm{g_t^\lambda-f_t}_{L^1}\longrightarrow0.
		\end{equation}
		If, more strongly, $\KLdiv(g_0^\lambda\Vert f_0)=O(\lambda)$ and
		\begin{equation}\label{eq:uniform-density-level-source}
			\nabla\log g_t^\lambda\in\Hh,
			\qquad
			\sup_{\lambda,t}\norm{\nabla\log g_t^\lambda}_{\Hh}\le M_{\Hh},
		\end{equation}
		then
		\begin{equation}\label{eq:density-level-zero-lambda-rate}
			\sup_{0\le t\le T}\KLdiv(g_t^\lambda\Vert f_t)=O_T(\lambda),
			\qquad
			\sup_{0\le t\le T}\norm{g_t^\lambda-f_t}_{L^1}=O_T(\sqrt\lambda).
		\end{equation}
	\end{corollary}
	
	\begin{proof}
		\emph{Step 1: A $\lambda$-uniform Gronwall coefficient.}  Since the bounds in \cref{ass:reverse-entropy-admissibility} hold uniformly in $\lambda$, the constants $M_*$, $E$, and $\mathsf H$ in \eqref{eq:reverse-bound-constants}, and hence $c_*$ from \cref{lem:cutoff-coercivity}, can be chosen independently of $\lambda$.  The coefficient $C=\frac12C_G+4C_B/c_*$ of \cref{thm:forced-reverse-KL} is then bounded by a fixed function in $L^1(0,T)$, so $\overline C:=\sup_\lambda\int_0^TC(r)\dd r<\infty$.
		
		\emph{Step 2: The forcing.}  Apply \cref{thm:forced-reverse-KL} with $g=g^\lambda$ and $s_t=S_\lambda(g_t^\lambda\dd v)$; this score is $C^1$ with $s$ and $Ds$ bounded by \cref{ass:score-space}.  By \cref{lem:Coulomb-weighted-loss} with $Q\le Q_*$, and by \eqref{eq:density-level-modulus-score-error},
		\[
		\mathcal L_A(s_t,g_t^\lambda)
		\le c_CC_RQ_*^{1/3}\norm{S_\lambda(g_t^\lambda\dd v)-q_{g_t^\lambda}}_{L^2(g_t^\lambda)}^2
		\le c_CC_RQ_*^{1/3}\,\mathfrak b_\lambda(g_t^\lambda).
		\]
		Inserting this into \eqref{eq:forced-reverse-KL-gronwall} and using $e^{\int_s^tC}\le e^{\overline C}$ gives \eqref{eq:density-level-modulus-KL} with $C_T:=e^{\overline C}$.
		
		\emph{Step 3: Convergence.}  If $\KLdiv(g_0^\lambda\Vert f_0)\to0$ and $\mathfrak B_T(\lambda)\to0$, the first statement in \eqref{eq:density-level-modulus-convergence} follows from \eqref{eq:density-level-modulus-KL}, and the second from Pinsker's inequality $\norm{g-f}_{L^1}\le\sqrt{2\KLdiv(g\Vert f)}$.
		
		\emph{Step 4: The rate.}  Under \eqref{eq:uniform-density-level-source}, the competitor $h=\nabla\log g_t^\lambda=q_{g_t^\lambda}$ in \eqref{eq:density-level-approximation-modulus} gives $\mathfrak b_\lambda(g_t^\lambda)\le\lambda M_{\Hh}^2$, hence $\mathfrak B_T(\lambda)\le\lambda TM_{\Hh}^2$.  Together with $\KLdiv(g_0^\lambda\Vert f_0)=O(\lambda)$, \eqref{eq:density-level-modulus-KL} gives the first estimate in \eqref{eq:density-level-zero-lambda-rate}, and Pinsker's inequality gives the second.
	\end{proof}
	
	\begin{remark}[Conditional hypotheses]
		\label{rem:density-level-lambda-conditional}
		The source condition \eqref{eq:uniform-density-level-source} is satisfied, for a fixed density, by the augmented space of \cref{rem:maxwellian-exclusion}; whether an evolving family $g^\lambda_t$ solving \eqref{eq:density-level-regularized-Coulomb} satisfies it uniformly in $\lambda$ and $t$ is not known to us, so the rate \eqref{eq:density-level-zero-lambda-rate} is at present a criterion without a verified instance. The assumptions $g\in L^\infty$ and compact support alone imply neither finite Fisher information nor a uniform rate for $\mathfrak b_\lambda(g)$.  Moreover,
		\begin{equation}\label{eq:modulus-H-bound}
			\norm{S_\lambda(g\dd v)}_{\Hh}
			\le\sqrt{\mathfrak b_\lambda(g)/\lambda}.
		\end{equation}
		Thus $\mathfrak b_\lambda(g)\to0$ alone does not provide a $\lambda$-uniform bound on the learned score in $\Hh$. A uniform estimate $\mathfrak b_\lambda(g_t^\lambda)\le C\lambda$ does give such a bound along the assumed solution family. In particular, \eqref{eq:uniform-density-level-source} yields $\norm{S_\lambda(g_t^\lambda\dd v)}_{\Hh}\le M_{\Hh}$ through \eqref{eq:static-lambda-bias}. This controls the scores along that family, but does not by itself control every iteration or establish a common existence interval for the Coulomb alternating iteration.
		
		The source condition also restricts the choice of score space. For the Gaussian RKHS on $\R^3$, it cannot hold for any strictly positive $C^1$ probability density: the Gaussian RKHS is contained in unweighted $L^2(\R^3;\R^3)$, whereas the score of such a density cannot belong to this space. Indeed, if $\nabla\log g\in L^2(\R^3)$, Cauchy--Schwarz along almost every ray implies that $\log g(r\theta)$ converges to a finite limit as $r\to\infty$, contradicting the integrability of $g$. The source-condition rate therefore does not apply to the Gaussian RKHS. The approximation-modulus criterion does not require exact score membership and remains available, provided its hypotheses are verified for the evolving density family.
		
		Finally, the source condition is incompatible with a nonzero compactly supported density, since a globally bounded score prevents the log-density from tending to $-\infty$ at a finite support boundary. Thus the $O(\lambda)$ source-condition branch is a separate full-support density-level criterion, not a regularization-removal theorem for the compact-support characteristic flow of \cref{thm:continuum-Coulomb-waveform}.
	\end{remark}
	
	Static RKHS score-estimation rates \cite{ZhouShiZhu2020} and the density-level KL estimates of \cite{HuangWang2025,IlinHuWang2025,Ilin2025} do not control the dependent particle clouds generated here or compare an empirical measure with a smooth density, so the result above is a separate continuum statement that cannot yet be composed with the particle theorem. The orientation of the relative entropy matters in the compact-support Coulomb setting: under the hypotheses of \cite{HendersonSnelsonTarfulea2019}, the physical solution instantaneously spreads mass throughout velocity space, whereas the learned bi-Lipschitz characteristic flow from compactly supported data keeps compact support. Hence the condition $\norm{f_t/g_t}_\infty<\infty$ in the Coulomb a posteriori estimate of \cite{Ilin2025}, with $f$ physical and $g$ learned, fails and $\KLdiv(f_t\Vert g_t)=+\infty$. The orientation in \cref{thm:forced-reverse-KL} is not affected by this mismatch as such, but the present proof assumes strict positivity of $g$ (\cref{ass:reverse-entropy-admissibility,rem:weaker-classes}) and the uniform requirements in \cref{rem:density-level-lambda-conditional}, so it does not apply to the compact-support flow either.
	
	\subsection{Failure of the unregularized density-level Picard iteration}
	
	It is tempting to set $\lambda=0$ in the continuum alternating iteration \eqref{eq:continuum-outer-iteration}, identify the exact density-level minimizer with $\nabla\log f^k$, and claim convergence of the resulting unregularized density-level alternating iteration. The score map contains a spatial derivative, so this is an iteration for an unbounded generator.  Ordinary Banach-space Picard reasoning is invalid. The following elementary example shows actual divergence.
	
	\begin{proposition}[Divergence of the frozen heat-flow Picard iteration]
		\label{prop:heat-counterexample}
		On the one-dimensional torus, consider
		\begin{equation}\label{eq:heat-picard}
			u^{k+1}(t)=u_0+\int_0^t\Delta u^k(r)\dd r,
			\qquad u^0(t)=u_0.
		\end{equation}
		There exists real-analytic $u_0$ such that, for every fixed $t>0$, $\norm{u^k(t)}_{L^2}\to\infty$ along the iteration.
	\end{proposition}
	
	\begin{proof}
		Take
		\[
		u_0(x)=\sum_{m=1}^\infty e^{-m}\cos(mx),
		\]
		which is real-analytic because its Fourier coefficients decay exponentially. Since $\Delta\cos(mx)=-m^2\cos(mx)$, induction on \eqref{eq:heat-picard} shows that each $u^k$ is smooth and
		\begin{equation}\label{eq:heat-picard-series}
			u^k(t)=\sum_{j=0}^k\frac{t^j\Delta^ju_0}{j!}
			=\sum_{m=1}^\infty e^{-m}P_k(-tm^2)\cos(mx),
			\qquad
			P_k(-x):=\sum_{j=0}^k\frac{(-x)^j}{j!}.
		\end{equation}
		Fix $t>0$.  At the Fourier mode $m=k$, put $x=tk^2$.  For $k\ge3/t$ one has $x\ge3k$, so consecutive terms of $\sum_{j<k}x^j/j!$ have ratio $j/x\le1/3$ when read downward from $j=k$, so $\sum_{j<k}x^j/j!\le\frac{x^k}{k!}\sum_{i\ge1}3^{-i}=\frac12\frac{x^k}{k!}$ and hence $\abs{P_k(-tk^2)}\ge x^k/k!-\sum_{j<k}x^j/j!\ge(tk^2)^k/(2k!)$. By Stirling's bound $k!\le e\sqrt k\,(k/e)^k$, the $m=k$ Fourier coefficient of $u^k(t)$ satisfies
		\[
		e^{-k}\abs{P_k(-tk^2)}
		\ge\frac{e^{-k}(tk^2)^k}{2k!}
		\ge\frac{(tk)^k}{2e\sqrt k},
		\]
		which diverges as $k\to\infty$.  Since $\norm{u^k(t)}_{L^2}\ge\sqrt\pi\,e^{-k}\abs{P_k(-tk^2)}$, the $L^2$ norm diverges.
	\end{proof}
	
	The heat equation itself has the score form $\partial_tf=\Div(f\nabla\log f)$.  The linearization about $f\equiv1$ of the semi-implicit alternating transport step $\partial_tf^{k+1}=\Div(f^{k+1}\nabla\log f^k)$, which is the unregularized density-level analogue of the alternating iteration \eqref{eq:intro-iteration} for the heat equation, is exactly \eqref{eq:heat-picard}.  Thus a same-Sobolev contraction theorem for the unregularized density-level score iteration cannot hold in this generality. Possible remedies include analytic-vector assumptions with Gaussian Fourier decay, a spectral or RKHS cutoff, or a genuinely parabolic implicit step.

\section{Concluding remarks and open problems}
\label{sec:conclusion}

The viewpoint of this paper is that the learned score does not have to approximate the true score for the alternating scheme to converge to the regularized coupled dynamics; it has to be stable. Once the regularized score map is Lipschitz in $W_1$, the feedback between score fitting and particle evolution can be followed step by step: at fixed $\lambda>0$, exact solves converge factorially and inexact ones obey a posteriori error bounds. For bounded Lipschitz effective interactions these estimates do not depend on the particle number, so, together with mean-field stability, the alternating-iteration and particle limits can be taken in either order. Because we work with the interaction after score cancellation, the same theory covers far-field-truncated Landau kernels with $\gamma\ge-2$ and the uncut two-dimensional Coulomb kernel.

The three-dimensional Coulomb kernel is where the particle theory stops (\cref{prop:coulomb-effective-obstruction}). For bounded, compactly supported densities, averaging over common initial labels still gives factorial convergence of the continuum iteration at fixed $\lambda$, on a guaranteed local interval whose length can be chosen proportional to $\lambda$ (\cref{thm:continuum-Coulomb-waveform}). Removing the regularization is a different matter. At fixed $N$, the unregularized empirical score problem has a minimizer exactly when a compatibility condition holds; otherwise a component of the regularized score grows like $\lambda^{-1}$ (\cref{thm:empirical-lambda-zero-dichotomy}). For strictly positive densities, a relative-entropy argument shows that the regularized density-level flow converges to the Coulomb Landau equation as $\lambda\downarrow0$, but only under uniform regularity and approximation assumptions that are not proved here (\cref{cor:density-level-zero-lambda-KL}). A heat-equation counterexample also shows why an unregularized Picard argument requires additional assumptions (\cref{prop:heat-counterexample}).

We see two natural next steps. The first is a joint particle and regularization limit. One possible route would be to develop a common solution framework connecting the compactly supported flow construction with the strictly positive density criterion, establish uniform flow and score-approximation estimates on a common time interval, and control Coulomb collisions at the particle level, perhaps through a cutoff that is removed afterwards. The second concerns the constants. In the time-discretized experiments of \cref{sm:numerics}, successive iterates differ by $2.1\times10^{-10}$ at $k=15$, whereas the factorial bound for the exact iteration certifies increments below $10^{-10}$ only for $k\ge2741$, and the constants deteriorate further as $\lambda\downarrow0$ (\cref{rem:lambda-scaling}). Understanding why the iteration converges so much faster than our estimates predict, and obtaining constants that depend only mildly on $\lambda$, would make the theory a genuine guide for choosing iteration counts and solver tolerances.

	\appendix
	
	\section{Finite-dimensional feature spaces}
	\label{sm:features}
	
	For implementation one may specialize Assumption~\ref{ass:score-space} to a finite-dimensional feature space, replacing the $dN$ block system by an $m\times m$ system.  In the rest of this section, set the score space itself to $\Hh:=\Hh_m$; accordingly, $S_\lambda$ denotes minimization over $\Hh_m$. If $\Hh_m$ is instead treated only as a restricted class inside a larger score space, the formulas below certify the restricted optimum and a separate approximation gap remains, as explained in Remark~\ref{rem:restricted-class}.
	
	\begin{proposition}
		\label{prop:finite-feature-score}
		Let $\phi:\R^d\to\R^m$ be $C_b^2$, and consider the vector-valued feature space
		\begin{equation}\label{eq:feature-space}
			\Hh_m:=\{s_W:s_W(x)=W\phi(x),\;W\in\R^{d\times m}\},
			\qquad \norm{s_W}_{\Hh_m}:=\norm W_F.
		\end{equation}
		Assume that this coefficient representation is injective.  For $\mu\in\Pone(\R^d)$ set
		\begin{equation}\label{eq:feature-moments}
			G_\mu:=\int\phi(x)\phi(x)^\top\dd\mu(x),
			\qquad
			g_{\mu,a}:=\int\partial_a\phi(x)\dd\mu(x),
			\quad a=1,\ldots,d.
		\end{equation}
		Then the $a$th row $w_{\mu,a}^{\lambda\,\top}$ of the coefficient matrix of $S_\lambda(\mu)$ is
		\begin{equation}\label{eq:feature-closed-form}
			w_{\mu,a}^{\lambda}
			=-(G_\mu+\lambda I_m)^{-1}g_{\mu,a}.
		\end{equation}
		For $W^\lambda=(w_{\mu,1}^\lambda,\ldots,w_{\mu,d}^\lambda)^\top$ and every $W\in\R^{d\times m}$, the exact objective gap is
		\begin{equation}\label{eq:feature-gap}
			\J_\mu^\lambda(s_W)-\J_\mu^\lambda(s_{W^\lambda})
			=\sum_{a=1}^d
			(w_a-w_{\mu,a}^\lambda)^\top
			(G_\mu+\lambda I_m)(w_a-w_{\mu,a}^\lambda).
		\end{equation}
	\end{proposition}
	
	\begin{proof}
		Let $w_a^\top$ denote the rows of $W$.  Since $(s_W)_a=w_a^\top\phi$ and $\Div s_W=\sum_{a=1}^dw_a^\top\partial_a\phi$, the definitions \eqref{eq:feature-moments} and $\norm W_F^2=\sum_a\abs{w_a}^2$ give
		\[
		\J_\mu^\lambda(s_W)
		=\sum_{a=1}^d
		\left[w_a^\top H_\mu w_a+2w_a^\top g_{\mu,a}\right],
		\qquad H_\mu:=G_\mu+\lambda I_m .
		\]
		Since $G_\mu\succeq0$ and $\lambda>0$, $H_\mu$ is positive definite, so the $d$ row problems are independent strictly convex quadratic problems.  Their normal equations $H_\mu w_{\mu,a}^\lambda+g_{\mu,a}=0$ give \eqref{eq:feature-closed-form}.  Using $H_\mu w_{\mu,a}^\lambda=-g_{\mu,a}$, completing the square in each row gives
		\[
		w_a^\top H_\mu w_a+2w_a^\top g_{\mu,a}
		-\bigl[(w_{\mu,a}^\lambda)^\top H_\mu w_{\mu,a}^\lambda
		+2(w_{\mu,a}^\lambda)^\top g_{\mu,a}\bigr]
		=(w_a-w_{\mu,a}^\lambda)^\top H_\mu(w_a-w_{\mu,a}^\lambda),
		\]
		and summing over $a$ proves \eqref{eq:feature-gap}.
	\end{proof}
	
	\begin{corollary}
		\label{cor:linear-residual-certificate}
		In the setting of \cref{prop:finite-feature-score}, let $H_\mu:=G_\mu+\lambda I_m$ and, for an approximate coefficient matrix $\widetilde W$, define the row residuals
		\[
		\rho_a:=H_\mu\widetilde w_a+g_{\mu,a},
		\qquad a=1,\ldots,d.
		\]
		Then the global excess objective relative to the selected score space $\Hh=\Hh_m$, as required in Theorem~\ref{thm:inexact}, is exactly
		\begin{equation}\label{eq:linear-residual-gap}
			\J_\mu^\lambda(s_{\widetilde W})
			-\J_\mu^\lambda(S_\lambda(\mu))
			=\sum_{a=1}^d\rho_a^\top H_\mu^{-1}\rho_a
			\le\frac1\lambda\sum_{a=1}^d\abs{\rho_a}^2.
		\end{equation}
		Thus a conjugate-gradient or direct-solve residual certifies the \emph{global} score gap without knowing the exact minimizer.
	\end{corollary}
	
	\begin{proof}
		Since $\widetilde w_a-w_{\mu,a}^\lambda=H_\mu^{-1}\rho_a$, substitute into \eqref{eq:feature-gap}; the inequality follows from $H_\mu\succeq\lambda I_m$.
	\end{proof}
	
	\begin{remark}\label{rem:rff}
		For fixed frequencies $\omega_\ell$ and phases $b_\ell$, the features
		\[
		\phi_\ell(x)=\sqrt{\frac2m}\cos(\omega_\ell\cdot x+b_\ell)
		\]
		are bounded with bounded derivatives of every order.  Hence the resulting finite-dimensional RKHS satisfies Assumption~\ref{ass:score-space}; all constants are finite for each fixed feature draw, after linearly dependent features have been removed so that the coefficient representation is injective.  Formula \eqref{eq:feature-closed-form} requires one $m\times m$ positive-definite solve per time sample, and does not invoke a nonconvex optimizer.  The $N$-independent constants are uniform only for a fixed feature space; no uniformity is claimed if the feature count or frequencies vary with $N$; see \cite{RahimiRecht2007} for the random-feature construction.
	\end{remark}
	
	\begin{remark}
		\label{rem:finite-feature-lambda-zero}
		For the feature model \eqref{eq:feature-space}, the unregularized coefficient problem is well posed and has a unique minimizer when $G_{\mu_Z}\succeq\kappa I_m$ for some $\kappa>0$; necessarily $m\le N$. A dynamic zero-regularization argument based on \eqref{eq:feature-closed-form} with $\lambda=0$ would need this lower bound uniformly along all particle paths and iterations.  Full rank at one configuration does not prevent the inverse and its data-stability constant from diverging near a later rank loss.  Even under a uniform lower bound, a fixed feature space retains a separate approximation bias relative to the physical score.
	\end{remark}
	
	\section{Complements to the inexact theorem}
	\label{sm:inexact}
	
	\begin{remark}\label{rem:integrated-losses}
		In the setting of Theorem~\ref{thm:inexact}, Cauchy--Schwarz in $t$, together with Jensen's inequality in the particle average, gives
		\begin{align}
			\int_0^T\sqrt{\eps_k(t)}\dd t
			&\le
			\sqrt{T\int_0^T\eps_k(t)\dd t},
			\label{eq:integrated-score-loss}\\
			\int_0^Tr_{k+1}(t)\dd t
			&\le
			\sqrt{
				T\int_0^T\frac1N\sum_{i=1}^N
				\abs{\mathcal R_{i,t}^{k+1}}^2\dd t
			}.
			\label{eq:integrated-system-loss}
		\end{align}
		Consequently,
		\[
		\norm{\eta_k}_{L^1(0,T)}
		\le
		2a_0c_0
		\sqrt{\frac{T}{\lambda}\int_0^T\eps_k(t)\dd t}
		+
		\sqrt{
			T\int_0^T\frac1N\sum_{i=1}^N
			\abs{\mathcal R_{i,t}^{k+1}}^2\dd t
		}.
		\]
		Thus finite time-integrated global score-gap bounds and finite time-integrated mean-square system residuals ensure $\eta_k\in L^1(0,T)$ at each iteration. If both integrals on the right-hand side tend to zero as $k\to\infty$, then $\norm{\eta_k}_{L^1(0,T)}\to0$, and Theorem~\ref{thm:inexact} yields uniform convergence of the iteration. The finite-iteration estimate requires only finiteness of these quantities at the iterations used; no separate uniform-in-$k$ integrated gap budget or time-uniform score-gap bound is required.
	\end{remark}
	
	\begin{corollary}
		\label{cor:implicit-midpoint}
		Fix a prescribed score value $s$ with $\norm s_\infty\le M$ and $\Lips(s)\le L$, and let $L_F=2(2a_1M+a_0L)$.  The implicit midpoint step
		\begin{equation}\label{eq:implicit-midpoint}
			Z^{n+1}=Z^n+hF\left(\frac{Z^n+Z^{n+1}}2;s\right)
		\end{equation}
		has a unique solution for every step size $h>0$ with $hL_F/2<1$.  If $A$ is even, the step preserves the empirical momentum exactly.  If in addition $A=A^\top$ and $A(z)z=0$, it also preserves empirical kinetic energy exactly.  Finally, for any absolutely continuous dense-output reconstruction $\widehat X$ of the numerical system values, the contribution of the system solve to the hypotheses of Theorem~\ref{thm:inexact} is measured by the computable defect
		\begin{equation}\label{eq:dense-output-defect}
			r(t)=\abs{\dot{\widehat X}_t-F(\widehat X_t;s_t)}_N,
		\end{equation}
		where $s_t$ is the prescribed score path whose value on the current step plays the role of $s$ in \eqref{eq:implicit-midpoint}. For admissible iterates, this defect and a global score-gap certificate $\eps_k$ holding almost everywhere yield the explicit error bound of Theorem~\ref{thm:inexact}, provided that $\eta_k\in L^1(0,T)$ at each iteration. The integrated conditions in \cref{rem:integrated-losses} are sufficient for this integrability. A system-defect certificate alone does not verify the score hypothesis.
	\end{corollary}
	
	\begin{proof}
		The fixed-point map $Y\mapsto Z^n+hF((Z^n+Y)/2;s)$ has Lipschitz constant $hL_F/2$, proving the first claim.  Summing \eqref{eq:implicit-midpoint} over particles proves momentum conservation.  With $\overline Z=(Z^n+Z^{n+1})/2$,
		\[
		\frac1N\sum_i\bigl(\abs{z_i^{n+1}}^2-\abs{z_i^n}^2\bigr)
		=\frac{2h}{N}\sum_i\overline z_i\cdot F_i(\overline Z;s)=0
		\]
		by the same pairwise calculation as Proposition~\ref{prop:conservation}.  The last claim is the definition of the system residual in \eqref{eq:system-residual}: with $\widetilde X^{k+1}=\widehat X$ one has $r_{k+1}=r$, while the score-gap term $\eps_k$ must be supplied separately.
	\end{proof}
	
	\section{Numerical experiments}
	\label{sm:numerics}
	
	The proof is analytic; the following deterministic experiments check four observable mechanisms: factorial decay of the iteration, stability as $N$ varies, persistent-bias scaling consistent with the error-floor bound, and the two branches of the fixed-$N$ zero-regularization dichotomy.  All source code, raw trajectories, CSV data, and plotting commands are stored separately in \texttt{experiments/}; no numerical output is used in a proof.
	
	\subsection{Setup}
	
	We take $d=2$ and the bounded Landau-structured matrix $A_{\rm test}$ from \eqref{eq:numerical-kernel}.  The score class consists of $40$ scalar random Fourier features---$20$ fixed sine/cosine pairs---with frequency scale $0.85$, and every score step uses the exact SPD solve \eqref{eq:feature-closed-form} with $\lambda=0.08$.  The initial cloud is a fixed two-component Gaussian mixture (seed $20260815$).  Unless stated otherwise, $N=48$, $T=1.5$, and the time mesh has $150$ subintervals.  Each frozen-score system is integrated by classical RK4 using a linearly interpolated coefficient path.  This is therefore a time-semidiscrete version of Algorithm~\ref{alg:alternating}: the score minimizer is exact at every mesh node in the selected feature space, while temporal score interpolation and RK4 both create discretization defects.  The refinement study measures their combined effect.

	For this feature space the constants of Assumption~\ref{ass:score-space} can be computed exactly.  Writing $\phi(x)=q^{-1/2}\bigl(\cos(\omega_\ell\cdot x),\sin(\omega_\ell\cdot x)\bigr)_{\ell=1}^{q}$ with $q=20$ frequencies $\omega_\ell\in\R^2$, one has $\abs{\phi(x)}\equiv1$ and $D\phi(x)^\top D\phi(x)=\Sigma:=q^{-1}\sum_{\ell}\omega_\ell\omega_\ell^\top$ for every $x$, so that $c_0=1$, $c_1=\norm{\Sigma}_{\rm op}^{1/2}$, $c_D=(\operatorname{tr}\Sigma)^{1/2}$, and $c_2=\bigl\|q^{-1}\sum_\ell\abs{\omega_\ell}^2\omega_\ell\omega_\ell^\top\bigr\|_{\rm op}^{1/2}$; for $A_{\rm test}$ one has $a_0=1$ and $a_1=9/(8\sqrt3)$ (the maximal radial derivative, attained at $\abs z=1/\sqrt3$).  Table~\ref{tab:analytic-constants} lists the values for the frequency draw of the experiments together with the constants of Theorem~\ref{thm:main-outer} at $\lambda=0.08$, $T=1.5$.  The certified threshold $2e\beta_\lambda T$ of Remark~\ref{rem:lambda-scaling} is about $5.3\times10^3$, and the bound \eqref{eq:main-factorial-increment} falls below $10^{-10}$ only for $k\ge2741$, whereas the observed increments reach $2.1\times10^{-10}$ at $k=15$.
	\begin{table}[H]
		\centering
		\caption{Constants of Assumptions~\ref{ass:score-space} and~\ref{ass:kernel}
			for the feature space and kernel of the experiments, and the resulting
			constants of Theorem~\ref{thm:main-outer} ($\lambda=0.08$, $T=1.5$).}
		\label{tab:analytic-constants}
		\begin{tabular}{@{}llll@{}}
			\toprule
			$c_0$ & $1$ & $R_\lambda=M_\lambda$ & $13.4$\\
			$c_1$ & $0.905$ & $L_\lambda$ & $12.1$\\
			$c_D$ & $1.07$ & $L_{S,\lambda}$ & $3.24\times10^{2}$\\
			$c_2$ & $1.61$ & $\beta_\lambda=2a_0L_{S,\lambda}$ & $6.48\times10^{2}$\\
			$a_0$ & $1$ & $L_{F,\lambda}$, $V_\lambda$ & $59.2$, $26.8$\\
			$a_1$ & $0.650$ & $2e\beta_\lambda T$ & $5.28\times10^{3}$\\
			\bottomrule
		\end{tabular}
	\end{table}
	
	\subsection{Convergence of the alternating iteration and dependence on \texorpdfstring{$N$}{N}}
	
	The initial successive-iterate increment is $6.4628\times10^{-1}$; by iteration $15$ it is $2.1456\times10^{-10}$.  Applying the map once more gives the fixed-point defect $2.3865\times10^{-11}$.  The left panel of \cref{fig:outer-convergence} plots the increments and the tail to the last iterate.  The right panel plots the heuristic ratio $(k+1)D_k^*/D_{k-1}^*$, where $D_k^*:=\sup_tD_k(t)$.  Its bounded behavior and the fitted envelope are consistent with factorial decay, but do not follow as a converse to the theorem's one-sided bound.  The fitted diagnostic has constant $1.2023$; it is an empirical descriptor, not a measurement of the generally conservative analytic quantities $\beta_\lambda$ and $e^{L_{F,\lambda}T}$.
	
	\begin{figure}[H]
		\centering
		\includegraphics[width=\textwidth]{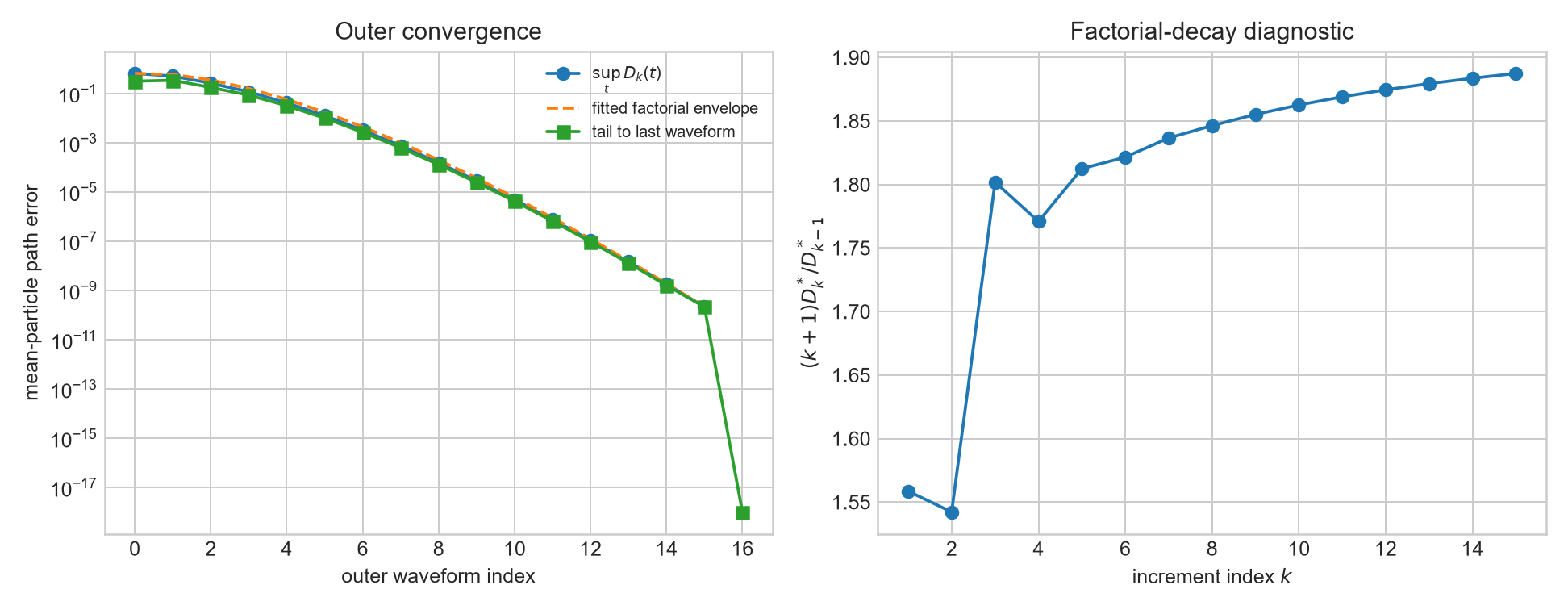}
		\caption{Successive-iterate decay, computed tail, and factorial-decay
			diagnostic for the baseline run.}
		\label{fig:outer-convergence}
	\end{figure}
	
	Using the same feature draw and nested initial clouds, the fitted constants for $N=24,48,96$ are respectively $1.2425$, $1.2023$, and $1.2133$, while the final one-more-map defects are $3.90$, $2.39$, and $2.68$ times $10^{-11}$.  The nearly overlapping curves in \cref{fig:N-sweep} are consistent with the $N$-independent deterministic constants.  They are not a statistical proof of a sampling rate, which Theorem~\ref{thm:main-mean-field} deliberately does not assert.
	
	\begin{figure}[H]
		\centering
		\includegraphics[width=\textwidth]
		{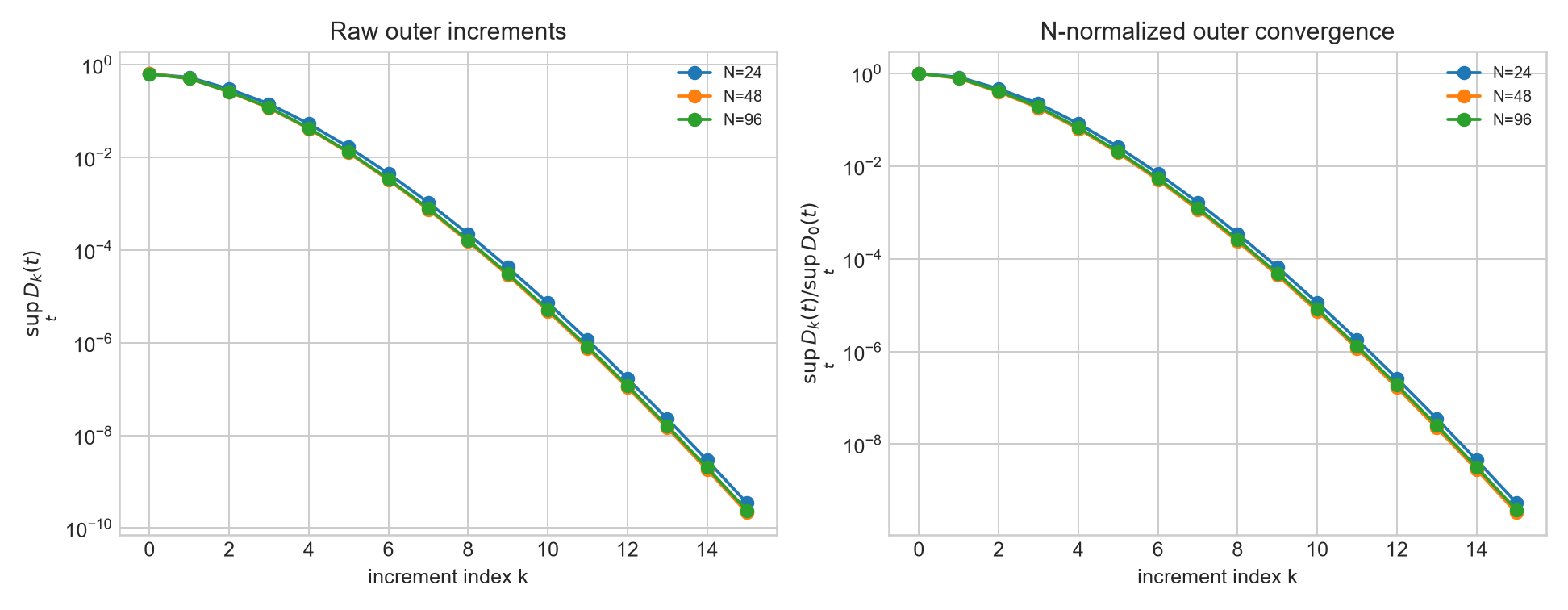}
		\caption{Raw and initial-increment-normalized convergence of the iteration for three
			particle counts with a fixed feature space.}
		\label{fig:N-sweep}
	\end{figure}
	
	As an independent reference, we also integrate the coupled regularized ODE \eqref{eq:main-coupled-particle} directly, recomputing the exact finite-feature score normal equations at all four RK4 stages.  At $\Delta t=0.01$, the converged iterate differs from this monolithic path by $1.4903\times10^{-6}$ in the supremum mean-particle distance (\cref{fig:monolithic-reference}).  The monolithic solver's two-level effective order relative to the $300$-step reference is $4.085$, and its maximum score normal-equation relative residual is $3.08\times10^{-16}$, where the relative residual is the Frobenius norm of the residual of the normal equations \eqref{eq:feature-closed-form} divided by the larger of one and the Frobenius norm of their right-hand side.  This comparison separates fixed-point convergence of the iteration from the temporal interpolation/discretization errors left by the semi-discrete implementation of the alternating iteration.
	
	\begin{figure}[H]
		\centering
		\includegraphics[width=0.9\textwidth]
		{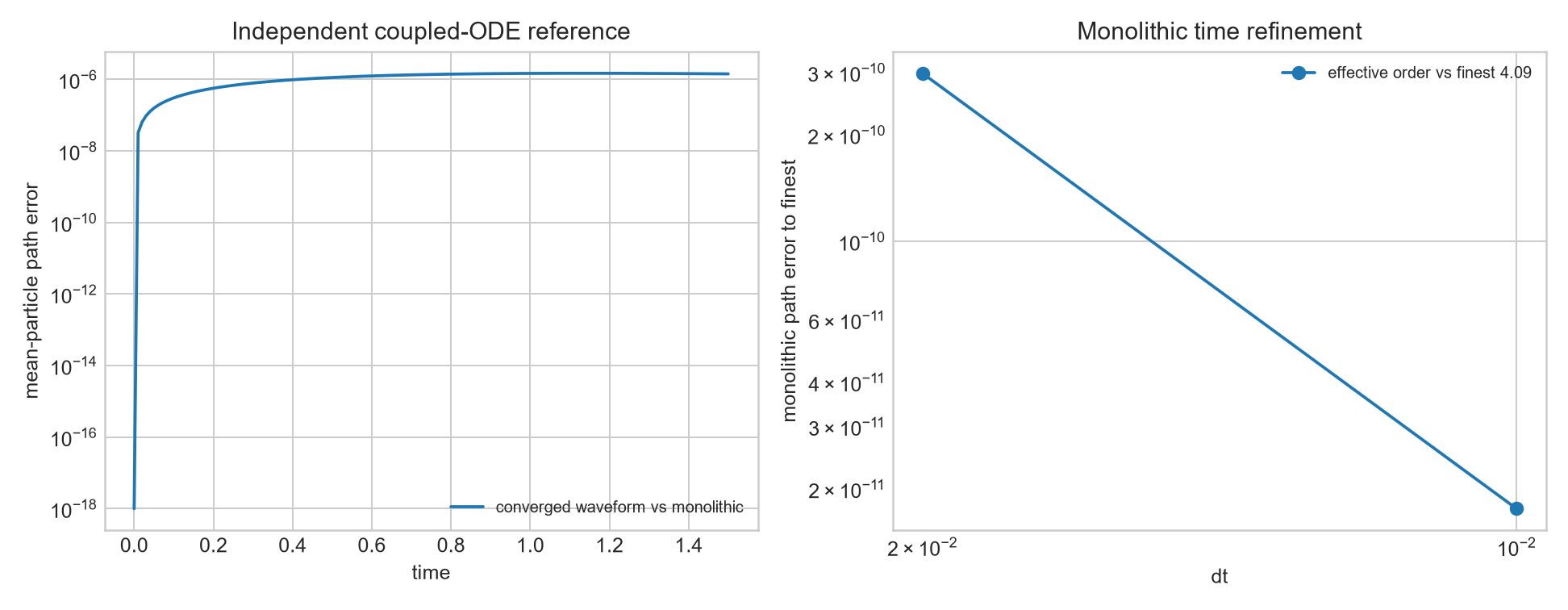}
		\caption{Converged iterate versus a direct current-state score--system solve,
			and independent refinement of the monolithic RK4 reference.}
		\label{fig:monolithic-reference}
	\end{figure}
	
	\subsection{Inexact score steps and conservation}
	
	To isolate the score term in Theorem~\ref{thm:inexact}, we add the same deterministic coefficient perturbation at every iteration, with amplitudes $0.01,0.02,0.04,0.08,0.16$.  At every time node of the final (converged) biased iterate the code computes the exact quadratic gap \eqref{eq:feature-gap} of the perturbed coefficients, rather than a surrogate training loss, and uses trapezoidal integration between nodes.  This is a nodal diagnostic, not an a.e.-in-time certificate for Theorem~\ref{thm:inexact}.  The largest discrepancy between the direct objective difference and the quadratic identity is $1.80\times10^{-15}$.  The converged biased trajectory error is linear in $\int_0^T\sqrt{\eps(t)}\dd t$ (\cref{fig:inexact-floor}): the log--log slope is $0.99962$, the through-origin fit has $R^2=0.9999998$, and the measured ratio is approximately $0.0368$.  This is consistent with the square-root scaling in \eqref{eq:eta-definition}; it neither proves that scaling nor claims the upper-bound constant is sharp.
	
	\begin{figure}[H]
		\centering
		\includegraphics[width=\textwidth]{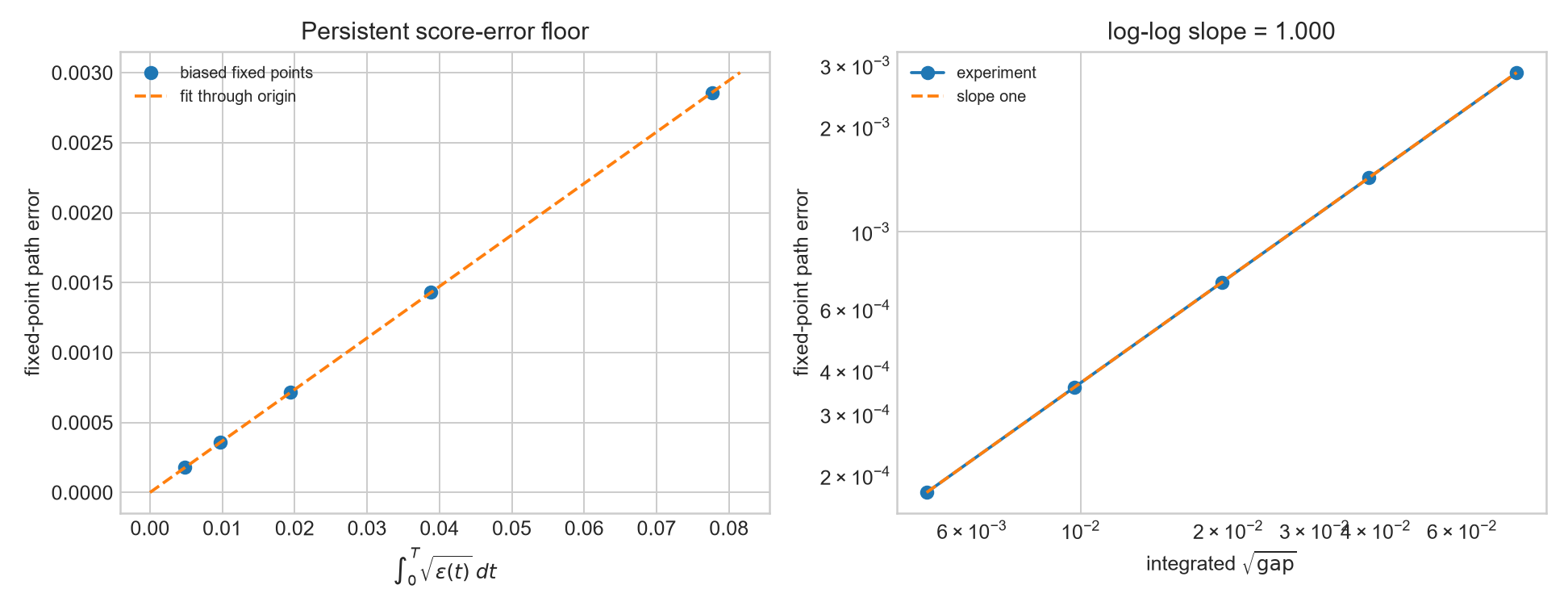}
		\caption{Persistent trajectory error versus the trapezoidal integral of the
			square root of the exactly evaluated nodal score-objective gap.}
		\label{fig:inexact-floor}
	\end{figure}
	
	At the time nodes, pairwise cancellation keeps the maximum momentum drift below $4.54\times10^{-16}$ (\cref{fig:invariants}).  The maximum nodal kinetic-energy drift over all baseline iterates is $2.59\times10^{-11}$; for the final iterate it decreases from $8.67\times10^{-11}$ to $3.43\times10^{-13}$ as $\Delta t$ is refined from $0.02$ to $0.005$.  The midpoint chord-residual diagnostic decreases from $1.11\times10^{-5}$ to $7.02\times10^{-7}$ over the same sweep. Selected diagnostics are collected in \cref{tab:numerical-audit}; complete time-resolved values are in the accompanying CSV files.
	
	\begin{figure}[H]
		\centering
		\includegraphics[width=0.9\textwidth]{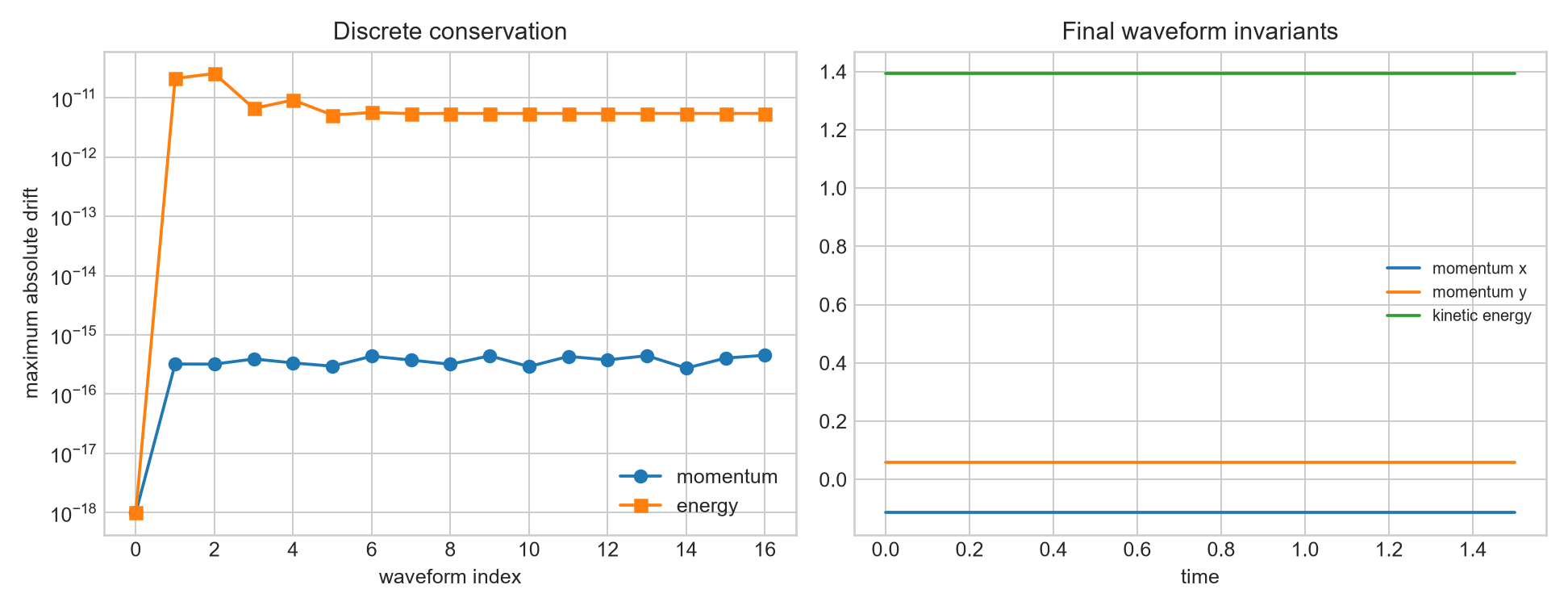}
		\caption{Momentum and kinetic-energy diagnostics evaluated at the time nodes
			for all iterates, including the final one. The nonzero nodal
			energy drift is due to RK4 time discretization.}
		\label{fig:invariants}
	\end{figure}
	
	\begin{table}[H]
		\centering
		\caption{Summary of numerical diagnostics.}
		\label{tab:numerical-audit}
		\begin{tabular}{@{}ll@{}}
			\toprule
			Diagnostic & Value\\
			\midrule
			Baseline $\sup_tD_0(t)$, $\sup_tD_{15}(t)$ & $6.4628\times10^{-1}$, $2.1456\times10^{-10}$\\
			One-more-map defect & $2.3865\times10^{-11}$\\
			Max.\ score normal-eq.\ relative residual (baseline) & $3.22\times10^{-16}$\\
			Iterate--monolithic path difference & $1.4903\times10^{-6}$\\
			Monolithic effective order vs. 300-step reference & $4.085$\\
			Maximum nodal momentum, energy drift & $4.54\times10^{-16}$, $2.59\times10^{-11}$\\
			$N=24,48,96$ fitted constants & $1.2425$, $1.2023$, $1.2133$\\
			Inexact-floor log--log slope; $R^2$ & $0.99962$; $0.9999998$\\
			Maximum exact-gap identity error & $1.80\times10^{-15}$\\
			\bottomrule
		\end{tabular}
	\end{table}
	
	\subsection{The zero-regularization dichotomy at fixed \texorpdfstring{$N$}{N}}
	
	We isolate Theorem~\ref{thm:empirical-lambda-zero-dichotomy} from all dynamical effects by sweeping $15$ logarithmically spaced values $10^{-8}\le\lambda\le10^{-1}$.  In a scalar four-feature problem with $N=24$, the empirical Gram matrix is uniformly coercive: $\lambda_{\min}(G)=0.764515$ and $\operatorname{cond}(G)=1.80993$. The regularized coefficient error relative to the unregularized minimizer has log--log slope $0.999999$ over the seven smallest values of $\lambda$, in agreement with \eqref{eq:empirical-lambda-compatible-rate}.
	
	The second calculation uses the exact one-dimensional Gaussian RKHS, not a feature truncation, with $N=5$ and $\sigma=0.55$.  Kernel and derivative-section inner products give
	\[
	\norm{b_\perp}_{\Hh}=0.253219,
	\qquad
	\frac{\norm{b_\perp}_{\Hh}^2}{\norm{b}_{\Hh}^2}=0.300914,
	\]
	while the maximum evaluation of $b_\perp$ at a training point is $6.08\times10^{-17}$.  Thus the unstable component is numerically invisible on the training set.  As $\lambda\downarrow0$, the fitted values at those points approach their finite limit with slope $0.999983$, but $\norm{S_\lambda}_{\Hh}$, a fixed off-grid value, and $-\min\J_{\mu_Z}^\lambda$ have slopes $-1.000000$, $-0.999996$, and $-0.999975$, respectively.  At $\lambda=10^{-8}$ the training-point RMS error is only $1.49\times10^{-7}$, whereas $\norm{S_\lambda}_{\Hh}=2.53\times10^7$.  This directly illustrates why a converged nodal solve does not certify an unregularized RKHS score; see \cref{fig:lambda-dichotomy}.
	
	\begin{figure}[H]
		\centering
		\includegraphics[width=\textwidth]
		{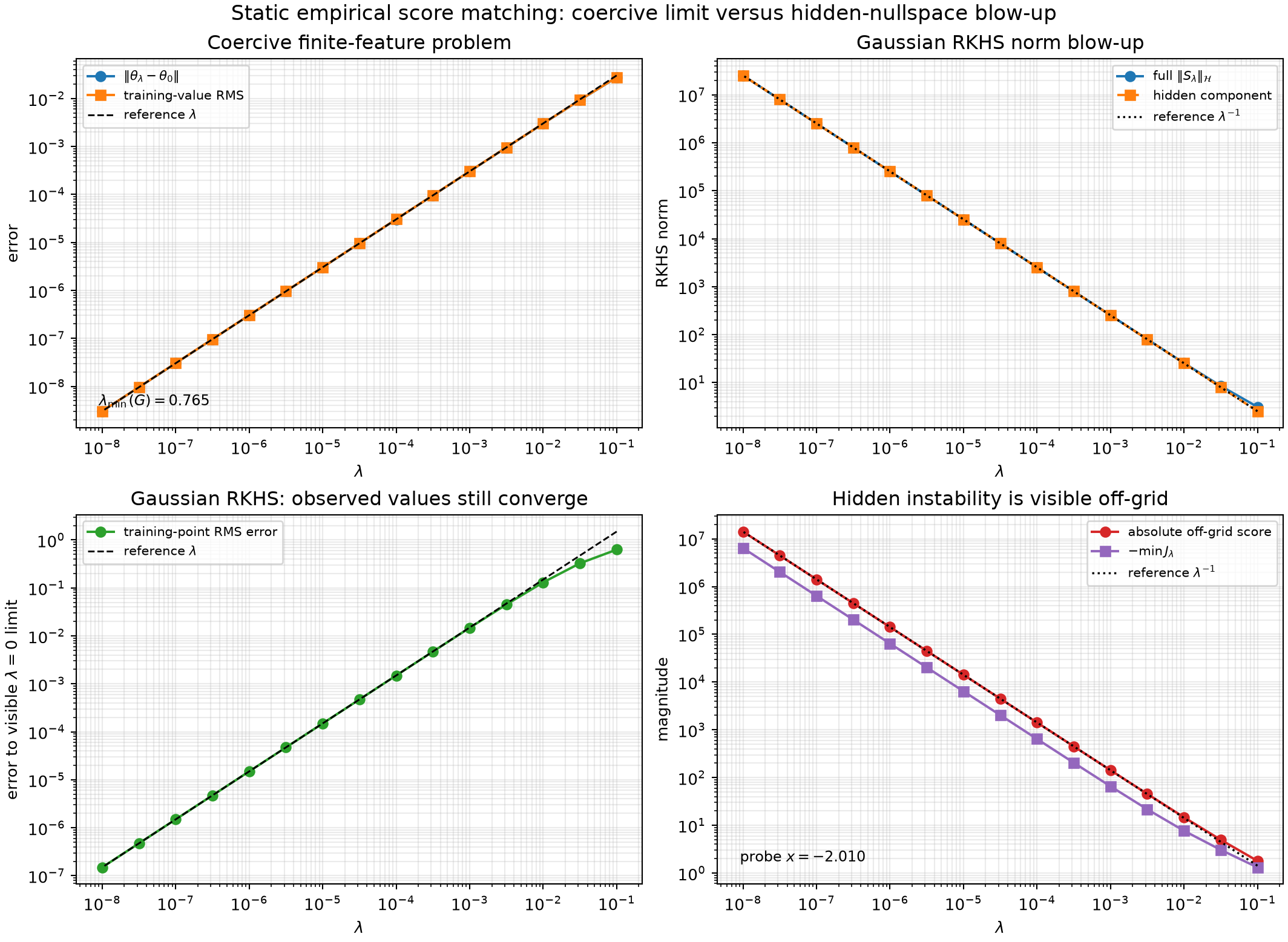}
		\caption{Zero-regularization dichotomy at fixed $N$.  The coercive finite-feature problem
			converges at first order (top left).  In the exact Gaussian RKHS, the hidden
			null-space component makes the full norm diverge like $\lambda^{-1}$ (top
			right), even while training values converge (bottom left); an off-grid value
			and the negative optimum reveal the same instability (bottom right).  Dashed
			lines show reference powers.  This finite-sample static test makes no claim
			about a dynamic Landau, Coulomb, mean-field, or simultaneous $(N,\lambda)$
			limit.}
		\label{fig:lambda-dichotomy}
	\end{figure}
	
	The executable validator checks finiteness, a six-order increment reduction, the fixed-point defect, conservation tolerances, the normal-equation residual, the monolithic reference and its refinement order, the nodal exact-gap identity, refinement monotonicity, and all asymptotic slopes in the static $\lambda$ sweep. The tests of the alternating iteration support the fixed-$(\lambda,A)$ theorem only, and the zero-regularization test supports only the algebraic dichotomy.  They do not test a dynamic singular-Coulomb or simultaneous particle--zero-regularization limit, nor a universal empirical-measure rate.
	
	\section{Proofs of the relative-entropy lemmas}
	\label{sm:tabary-proofs}
	
	We prove Lemmas~\ref{lem:weighted-pinsker}, \ref{lem:cutoff-coercivity}, and~\ref{lem:reverse-entropy-slack}. The first is elementary. As explained in section~\ref{sec:landau-connection}, Lemma~\ref{lem:cutoff-coercivity} corresponds to Lemma~2.8 of \cite{TabaryWeakStrong2025}, and Lemma~\ref{lem:reverse-entropy-slack} combines the good--bad decomposition of its Lemma~2.4 with free-Young-parameter versions of its Lemmas~2.9 and~2.11. The notation and standing assumptions are those of section~\ref{sec:landau-connection}.
	
	\begin{proof}[Proof of Lemma~\ref{lem:weighted-pinsker}]
		For $y>0$, $\log y\ge1-1/y$; on the set $\{g>0\}$, applying this with $y=(g/f)^{1/2}$ gives $g\log(g/f)\ge2g-2\sqrt{fg}$, and the inequality extends to $\{g=0\}$ with the convention $0\log0=0$ (both sides vanish). Integrating, $\KLdiv(g\Vert f)\ge2-2\int\sqrt{fg}=\int(\sqrt g-\sqrt f)^2$.  Since $(\sqrt g+\sqrt f)^2\le2(g+f)$,
		\[
		\int\frac{(g-f)^2}{g+f}\dd v
		=\int(\sqrt g-\sqrt f)^2\frac{(\sqrt g+\sqrt f)^2}{g+f}\dd v
		\le2\int(\sqrt g-\sqrt f)^2\dd v\le2\KLdiv(g\Vert f).
		\]
		Finally, by Cauchy--Schwarz, $\int\abs{g-f}\varphi =\int\frac{\abs{g-f}}{(g+f)^{1/2}}(g+f)^{1/2}\varphi \le(\int\frac{(g-f)^2}{g+f})^{1/2}(\int(g+f)\varphi^2)^{1/2}$.
	\end{proof}
	
	\begin{proof}[Proof of Lemma~\ref{lem:cutoff-coercivity}]
		The stated properties of $\widetilde A_C$ are direct: $\abs z^2\le\langle z\rangle^3$ gives the operator bound; the derivative of $\langle z\rangle^{-3}$ does not contribute to the divergence because $(\abs z^2\delta_{ij}-z_iz_j)z_i=0$, and $\partial_i(\abs z^2\delta_{ij}-z_iz_j)=2z_j-4z_j=-2z_j$ in three dimensions.
		
		For a unit vector $\xi$, $\xi^\top\widetilde A_C(z)\xi=c_C\langle z\rangle^{-3}\abs{P_\xi^\perp z}^2$, where $P_\xi^\perp z=z-(z\cdot\xi)\xi$.  Fix $v$ and $\xi$.  Put $R:=\max\{1,2\sqrt E\}$, so that $\int_{\{\abs w>R\}}g\le E/R^2\le\frac14$. Put $H_+:=\int g\log^+g\dd w$; since $g\log^-g\le(2/e)\sqrt g$ on $\{g\le e^{-\abs w^2}\}$ and $\log^-g\le\abs w^2$ on $\{e^{-\abs w^2}<g<1\}$, one has $H_+\le\mathsf H+E+C_0$ with an absolute constant $C_0$.  For $M:=e^{4(\mathsf H+E+C_0+1)}$, which satisfies $H_+/\log M\le\frac14$, and every measurable $S$,
		\[
		\int_Sg\dd w\le\int_{\{g>M\}}g\dd w+M\abs S
		\le\frac{H_+}{\log M}+M\abs S\le\frac14+M\abs S .
		\]
		Let $S:=\{w:\abs w\le R,\ \abs{P_\xi^\perp(v-w)}\le\rho\}$, a piece of a cylinder of radius $\rho$ and length at most $2R$, so $\abs S\le2\pi R\rho^2$; choose $\rho^2:=1/(8\pi RM)$, so that $M\abs S\le\frac14$.  Then $W:=\{\abs w\le R\}\setminus S$ carries $g$-mass at least $\frac14$, and for $w\in W$ one has $\abs{P_\xi^\perp(v-w)}\ge\rho$ and $\langle v-w\rangle\le1+\abs v+R\le\sqrt2(1+R)\langle v\rangle$.  Hence
		\begin{align*}
			\xi^\top(\widetilde A_C*g)(v)\xi
			&\ge c_C\int_W\langle v-w\rangle^{-3}\abs{P_\xi^\perp(v-w)}^2g(w)\dd w\\
			&\ge\frac{c_C\rho^2}{4\cdot2^{3/2}(1+R)^3}\langle v\rangle^{-3}
			=:c_*\langle v\rangle^{-3},
		\end{align*}
		and $c_*$ depends only on $(c_C,E,\mathsf H)$.
	\end{proof}
	
	\begin{proof}[Proof of Lemma~\ref{lem:reverse-entropy-slack}]
		Fix $t$ and suppress it.  All integrations by parts below are justified by Assumption~\ref{ass:reverse-entropy-admissibility} and Remark~\ref{rem:integrability-justification}.  We use the identities $g-f=fh$ and $g\nabla r=f\nabla h$, write $q_f=\nabla\log f$, $q_g=\nabla\log g=q_f+\nabla r$, and $\Delta_{vw}\phi:=\phi(v)-\phi(w)$.
		
		\emph{Step 1: Decomposition.}  Since $U[\cdot,\cdot]$ is linear in the score,
		\begin{align*}
			U[g,q_g](v)-U[f,q_f](v)
			&=-\int A_C(v-w)\,\Delta_{vw}\nabla r\;g(w)\dd w\\
			&\quad-\int A_C(v-w)\,\Delta_{vw}q_f\;(g-f)(w)\dd w .
		\end{align*}
		Multiply by $g(v)\nabla r(v)$ and integrate in $v$.  Because $A_C$ is even and symmetric, symmetrizing the first term in $(v,w)$ gives $-\frac12\iint(\Delta_{vw}\nabla r)^\top A_C(v-w)(\Delta_{vw}\nabla r)\,g(v)g(w) =-\mathcal G$.  In the second term insert $g(v)\nabla r(v)=f(v)\nabla h(v)$ and $g(w)-f(w)=f(w)h(w)$; this is exactly the right-hand side of \eqref{eq:relative-bad-term-once}, which proves \eqref{eq:landau-part-decomposition} with $\mathcal B$ given by \eqref{eq:relative-bad-term-once}.  To identify it with \eqref{eq:relative-bad-term}, symmetrize once more and use $f(v)f(w)\Delta_{vw}q_f=(\nabla_v-\nabla_w)(f(v)f(w))$ and $h(w)\nabla h(v)-h(v)\nabla h(w)=(\nabla_v-\nabla_w)(h(v)h(w))$:
		\[
		\mathcal B=-\frac12\iint
		\bigl[(\nabla_v-\nabla_w)(h(v)h(w))\bigr]^\top
		A_C(v-w)(\nabla_v-\nabla_w)(f(v)f(w))\dd v\dd w .
		\]
		Integrating by parts with respect to the first-order operator $\nabla_v-\nabla_w$ (an ordinary integration by parts in $v$ and one in $w$) gives \eqref{eq:relative-bad-term}.
		
		\emph{Step 2: Lower bound for the good term.}  Since $0\le\widetilde A_C\le A_C$,
		\begin{align*}
			\mathcal G\ge\widetilde{\mathcal G}
			&:=\frac12\iint\abs{\Delta_{vw}\nabla r}^2_{\widetilde A_C(v-w)}g(v)g(w)\dd v\dd w\\
			&=\underbrace{\int\nabla r^\top(\widetilde A_C*g)\nabla r\,g\dd v}_{=:\widetilde{\mathcal G}_1}
			-\underbrace{\iint\nabla r(v)^\top\widetilde A_C(v-w)\nabla r(w)\,g(v)g(w)\dd v\dd w}_{=:\widetilde{\mathcal G}_2}.
		\end{align*}
		By Lemma~\ref{lem:cutoff-coercivity}, $\widetilde{\mathcal G}_1\ge c_*\mathcal I$. For $\widetilde{\mathcal G}_2$ write $g\nabla r=f\nabla h$ and integrate by parts in $v$:
		\begin{align*}
			\widetilde{\mathcal G}_2
			&=\iint f(v)f(w)\,\partial_ih(v)\,(\widetilde A_C)_{ij}(v-w)\,\partial_jh(w)\dd v\dd w\\
			&=-\iint h(v)\Bigl[\partial_if(v)(\widetilde A_C)_{ij}(v-w)
			+f(v)\bigl(\partial_i(\widetilde A_C)_{ij}\bigr)(v-w)\Bigr]f(w)\partial_jh(w)\dd v\dd w.
		\end{align*}
		Using $\norm{\widetilde A_C}_{\rm op}\le c_C$, $\abs{\nabla\cdot\widetilde A_C}\le2c_C$, $\abs{\nabla f}=f\abs{q_f}\le K_1\langle v\rangle^\kappa f$, and $\langle v\rangle^\kappa\ge1$,
		\[
		\abs{\widetilde{\mathcal G}_2}
		\le c_C(2+K_1)\int\abs{g-f}\langle v\rangle^\kappa\dd v
		\int g\abs{\nabla r}\dd w .
		\]
		Lemma~\ref{lem:weighted-pinsker} with $\varphi=\langle v\rangle^\kappa$ and $2\kappa\le m_*$ gives $\int\abs{g-f}\langle v\rangle^\kappa\le(2\KLdiv(g\Vert f))^{1/2}M_*^{1/2}$, while Cauchy--Schwarz and $3\le m_*$ give $\int g\abs{\nabla r}\le\mathcal I^{1/2}(\int\langle v\rangle^3g)^{1/2} \le\mathcal I^{1/2}M_*^{1/2}$.  Hence $\abs{\widetilde{\mathcal G}_2} \le\sqrt2\,c_C(2+K_1)M_*\,\KLdiv(g\Vert f)^{1/2}\mathcal I^{1/2}$, and Young's inequality $xy\le\frac{c_*}2x^2+\frac1{2c_*}y^2$ yields $\abs{\widetilde{\mathcal G}_2}\le\frac{c_*}2\mathcal I+C_G\KLdiv(g\Vert f)$. Combining the two bounds proves \eqref{eq:relative-good-lower}.
		
		\emph{Step 3: Bad term.}  Use the form \eqref{eq:relative-bad-term-once}. For $v\ne w$ the projection in $A_C$ has operator norm one, so by the mean value theorem along the segment $[v,w]$, on which $\langle\cdot\rangle\le\max\{\langle v\rangle,\langle w\rangle\}$,
		\[
		\abs{A_C(v-w)\,\Delta_{vw}q_f}
		\le\frac{c_C\abs{q_f(v)-q_f(w)}}{\abs{v-w}}
		\le c_C\sup_{[v,w]}\norm{D^2\log f}
		\le c_CK_2\bigl(\langle v\rangle^\zeta+\langle w\rangle^\zeta\bigr):
		\]
		the Coulomb singularity is cancelled exactly by the score difference, and only the weighted Hessian bound $K_2$ enters.  Therefore
		\[
		\abs{\mathcal B}\le c_CK_2\Bigl[
		\int g\abs{\nabla r}\langle v\rangle^\zeta\dd v\int\abs{g-f}\dd w
		+\int g\abs{\nabla r}\dd v\int\abs{g-f}\langle w\rangle^\zeta\dd w\Bigr].
		\]
		By Cauchy--Schwarz and $3+2\zeta\le m_*$, $\int g\abs{\nabla r}\langle v\rangle^\zeta \le\mathcal I^{1/2}(\int\langle v\rangle^{3+2\zeta}g)^{1/2} \le\mathcal I^{1/2}M_*^{1/2}$, and likewise $\int g\abs{\nabla r}\le\mathcal I^{1/2}M_*^{1/2}$; by Lemma~\ref{lem:weighted-pinsker} with $\varphi\equiv1$ and $\varphi=\langle w\rangle^\zeta$ ($2\zeta\le m_*$), both $\int\abs{g-f}$ and $\int\abs{g-f}\langle w\rangle^\zeta$ are at most $(2\KLdiv(g\Vert f))^{1/2}M_*^{1/2}$.  Hence $\abs{\mathcal B}\le2\sqrt2\,c_CK_2M_*\,\KLdiv(g\Vert f)^{1/2}\mathcal I^{1/2}$, and Young's inequality $xy\le\delta x^2+y^2/(4\delta)$ gives \eqref{eq:relative-bad-young}.
		
		\emph{Step 4: Absorption.}  By \eqref{eq:relative-good-lower}, $\mathcal I\le\frac2{c_*}\bigl(\mathcal G+C_G\KLdiv(g\Vert f)\bigr)$. Inserting this into \eqref{eq:relative-bad-young} with $\delta=\theta c_*/2$ gives $\mathcal B\le\abs{\mathcal B} \le\theta\mathcal G+\bigl(\theta C_G+2C_B/(\theta c_*)\bigr)\KLdiv(g\Vert f)$, which is \eqref{eq:relative-bad-slack}.  The functions $C_G$ and $C_B$ are quadratic in $K_1$ and $K_2$ with coefficients depending only on $c_C$, $M_*$, $E$, and $\mathsf H$, so they are integrable and uniform over families with uniform bounds.
	\end{proof}

		\noindent{\bf Acknowledgments}
	Y. Kim is supported by National Research Foundation of Korea (NRF) grants funded by the Korea government (MSIT) (RS-2023-00219980 and RS-2026-25490291). D. Lee is supported by National Research Foundation of Korea (NRF) grants funded by the Korea government (MSIT) (RS-2023-00212304 and RS-2023-00219980). S.-J. Son is supported by National Research Foundation of Korea (NRF) grants funded by the Korea government (MSIT and MOE) (RS-2023-00210484 and RS-2025-25419038).
	\medskip
	
	\noindent\textbf{Declaration}
	The authors used Codex (OpenAI) and Claude (Anthropic) to assist with manuscript preparation, the review and refinement of mathematical arguments, and computational checks. The authors reviewed and revised the AI-assisted material. The authors assume responsibility for all content.
\medskip

	\noindent\textbf{Data availability}
	The source code and numerical data supporting the experiments in this work are publicly available at \url{https://github.com/sungjunson/landau-score-convergence}. All numerical data were generated by the accompanying code using the configurations described in the paper; no external datasets were used.

\end{document}